\UseRawInputEncoding
\documentclass[12pt,reqno]{article}

\usepackage{amsfonts}
\usepackage{mathrsfs}
\usepackage{bm}
\usepackage{amsmath,amsthm,amsfonts,amssymb,latexsym,mathrsfs,color}

\usepackage[colorlinks=true,
linkcolor=webblue, filecolor=webbrown, citecolor=webred ]{hyperref}
\definecolor{webblue}{rgb}{0,.5,0}
\definecolor{webred}{rgb}{0,.5,0}
\definecolor{webbrown}{rgb}{.6,0,0}

\newtheorem{thm}{Theorem}[section]
\newtheorem{lem}[thm]{Lemma}
\newtheorem{cor}[thm]{Corollary}
\newtheorem{prop}[thm]{Proposition}

\theoremstyle{definition}
\newtheorem{definition}[thm]{Definition}

\newtheorem{cl}[thm]{Claim}

\newtheorem{rem}[thm]{Remark}

\numberwithin{equation}{section}

\newcommand{\sgn}{\text{\rm sgn}}

\newcommand{\balpha}{{\bm{\alpha}}}
\newcommand{\bbeta}{{\bm{\beta}}}
\newcommand{\bgamma}{{\bm{\gamma}}}

\newcommand{\bmu}{{\bm{\mu}}}

\newenvironment{scarray}{
             \textfont0=\scriptfont0
             \scriptfont0=\scriptscriptfont0
             \textfont1=\scriptfont1
             \scriptfont1=\scriptscriptfont1
             \textfont2=\scriptfont2
             \scriptfont2=\scriptscriptfont2
             \textfont3=\scriptfont3
             \scriptfont3=\scriptscriptfont3
           
           \begin{array}{c}}{\end{array}}
\def\hboxscript#1{ {\hbox{\scriptsize\em #1}} }
\allowdisplaybreaks[4]

\title{Total positivity from a kind of lattice paths
\thanks{Supported partially by the National Natural Science Foundation of China (Nos. 12022105, 11971206)}}
\author{Yu-Jie Cui, Bao-Xuan Zhu\footnote{Corresponding author
\newline\hspace*{5mm}
   {\it Email address:}
    bxzhu@jsnu.edu.cn(B.-X. Zhu)}}
\date{\footnotesize  School of Mathematics and Statistics,
          Jiangsu Normal University,
         Xuzhou 221116, P. R. China\\
         }
\begin{document}

\maketitle
\begin{abstract}
Total positivity of matrices is deeply studied and plays an
important role in various branches of mathematics. The main purpose
of this paper is to study total positivity of a matrix
$M=[M_{n,k}]_{n,k}$ generated by the weighted lattice paths in
$\mathbb{N}^2$ from the origin $(0,0)$ to the point $(k,n)$
consisting of types of steps: $(0,1)$ and $(1,t+i)$ for $0\leq i\leq
\ell$, where each step $(0,1)$ from height~$n-1$ gets the
weight~$b_n(\textbf{y})$ and each step $(1,t+i)$ from height~$n-t-i$
gets the weight $a_n^{(i)}(\textbf{x})$.

Using an algebraic method, we prove that the $\textbf{x}$-total
positivity of the weight matrix $[a_i^{(i-j)}(\textbf{x})]_{i,j}$
implies that of $M$. Furthermore, using the
Lindstr\"{o}m-Gessel-Viennot lemma, we obtain that both $M$ and the
Toeplitz matrix of each row sequence of $M$ with $t\geq1$ are
$\textbf{x}$-totally positive under the following three cases
respectively: (1) $\ell=1$, (2) $\ell=2$ and restrictions for
$a_n^{(i)}$, (3) general $\ell$ and both $a^{(i)}_n$ and $b_n$ are
independent of $n$. In addition, for the case (3), we show that the
matrix $M$ is a Riordan array, present its explicit formula and
prove total positivity of the Toeplitz matrix of the each column of
$M$. In particular, from the results for Toeplitz-total positivity,
we also obtain the P\'olya frequency and log-concavity of the
corresponding sequence.

Finally, as applications, we in a unified manner establish total
positivity and the Toeplitz-total positivity for many well-known
combinatorial triangles, including the Pascal triangle, the Pascal
square, the Delannoy triangle, the Delannoy square, the signless
Stirling triangle of the first kind, the Legendre-Stirling triangle
of the first kind, the Jacobi-Stirling triangle of the first kind,
the Brenti's recursive matrix, and so on. In particular, we extend
many known results of Brenti (J. Combin. Theory Ser. A, 1995), Yu
(Adv. Appl. Math., 2009),  Mongelli (Adv. Appl. Math., 2012), Zhu
(Proc. Amer. Math. Soc., 2014), Mu and Zheng (J. Integer Seq.,
2017), and so on.
  \\[7pt]

\end{abstract}
\tableofcontents
\section{Introduction}
\subsection{Motivation}
Lattice paths have been extensively studied in combinatorics and
have many applications in various domains such as computer science,
biology and physics \cite{Aig01,Bre95,KP99,Moh79,Nar79,SX08}.

Many well known combinatorial numbers can be interpreted by certain
lattice paths. For example, binomial coefficients have a great
number of nice properties in combinatorics and number theory. It is
known that the binomial coefficient $\binom{n}{k}$ counts the number
of lattice paths in $\mathbb{N}^2$ from $(0,0)$ to $(k,n-k)$ with
horizontal steps $(1,0)$ and vertical steps $(0,1)$ (see Figure~1
for a lattice path from $(0,0)$ to $(5,3)$), which satisfies the
recurrence relation
\begin{equation}\label{rec+binom}
    \binom{n}{k}=\binom{n-1}{k-1}+\binom{n-1}{k}.
\end{equation}

\begin{center}
\setlength{\unitlength}{1.2cm}
\begin{picture}(12,5.1)(-6.3,-2)
\thicklines\put(-3,-1){\line(1,0){6}}
\thicklines\put(-3,-1){\line(0,1){4}}

\thicklines\put(-3,-1){\line(1,0){1}}\thicklines\put(-3,-0.975){\line(1,0){1}}\thicklines\put(-3,-0.985){\line(1,0){1}}
\thicklines\put(-2,-1){\line(0,1){1}}\thicklines\put(-1.975,-0.975){\line(0,1){1}}\thicklines\put(-1.985,-0.985){\line(0,1){1}}
\thicklines\put(-2,0){\line(1,0){1}}\thicklines\put(-2,0.025){\line(1,0){1}}\thicklines\put(-2,0.015){\line(1,0){1}}
\thicklines\put(-1,0){\line(0,1){1}}\thicklines\put(-0.975,0.025){\line(0,1){1}}\thicklines\put(-0.985,0.015){\line(0,1){1}}
\thicklines\put(-1,1){\line(1,0){2}}\thicklines\put(-1,1.025){\line(1,0){2}}\thicklines\put(-1,1.015){\line(1,0){2}}
\thicklines\put(1,1){\line(0,1){1}}\thicklines\put(1.025,1.025){\line(0,1){1}}\thicklines\put(1.015,1.015){\line(0,1){1}}
\thicklines\put(1,2){\line(1,0){1}}\thicklines\put(1,2.025){\line(1,0){1}}\thicklines\put(1,2.015){\line(1,0){1}}

\put(-1.99,-1){\circle*{0.12}}\put(-1,-1){\circle*{0.12}}\put(0,-1){\circle*{0.12}}\put(1,-1){\circle*{0.12}}
\put(2,-1){\circle*{0.12}}

\put(-3,-1){\circle*{0.12}}\put(-3,0){\circle*{0.12}}\put(-3,1){\circle*{0.12}}\put(-3,2){\circle*{0.12}}

\put(-1.99,0.01){\circle*{0.12}}\put(-0.99,0.01){\circle*{0.12}}\put(-0.99,1.01){\circle*{0.12}}\put(0,1.01){\circle*{0.12}}
\put(1.01,1.01){\circle*{0.12}}\put(2,2.01){\circle*{0.12}}\put(1.01,2.01){\circle*{0.12}}
\scriptsize
\thicklines\put(2.63,-1.05){\textbf{$\longrightarrow$}}\thicklines\put(-3.06,2.85){\textbf{$\uparrow$}}
\put(-3.2,2.9){$y$}\put(2.9,-1.2){$x$}
\put(-3.65,-1.2){$(0,$}\put(-3.35,-1.2){$0)$}
\put(-3.65,-0.05){$(0,$}\put(-3.35,-0.05){$1)$}\put(-3.65,0.95){$(0,$}\put(-3.35,0.95){$2)$}
\put(-3.65,1.95){$(0,$}\put(-3.35,1.95){$3)$}
\put(-2.26,-1.3){$(1,$}\put(-1.93,-1.3){$0)$}\put(-1.26,-1.3){$(2,$}\put(-0.96,-1.3){$0)$}
\put(-0.26,-1.3){$(3,$}\put(0.07,-1.3){$0)$}\put(0.74,-1.3){$(4,$}\put(1.07,-1.3){$0)$}
\put(1.74,-1.3){$(5,$}\put(2.07,-1.3){$0)$}

\put(2.1,1.94){$(5,$}\put(2.43,1.94){$3)$}
\normalsize
\put(-3.5,-2){Figure~1: A lattice path from $(0,0)$ to $(5,3)$.}
\end{picture}
\end{center}
The oldest Pascal triangle $\mathcal {P}=[\binom{n}{k}]_{n,k}$ is
formed by binomial coefficients. In \cite[pp. 137]{Kar68}, Karlin
proved that the Pascal triangle $\mathcal {P}$ is totally positive.
In addition, for the total positivity of $\mathcal {P}$, one found
many different ways (see \cite{CLW15,FZ00,Pin10,Zhu14} for
instance). Let $C_i=\binom{n+\delta i}{k+\sigma i}$, where $n\geq
k$. Then for integers $\sigma>\delta>0$ and $\sigma>k$, the Toeplitz
matrix $[C_{i-j}]_{i,j\geq0}$ of the finite sequence $(C_i)_{i}$ is
totally positive, which was conjectured by Su and Wang \cite{SW08}
and confirmed by Yu \cite{Yu09}. If binomial coefficients are
arranged as a symmetric square matrix $\mathcal
{P}^{\ulcorner}=[\binom{n+k}{k}]_{n,k}$, then $\mathcal
{P}^{\ulcorner}$ is called {\it the Pascal square}. If let $\mathcal
{P}^{\ulcorner}_{n,k}=\binom{n+k}{k}$, then $\mathcal
{P}^{\ulcorner}_{n,k}$ counts the number of lattice paths in
$\mathbb{N}^2$ from the point $(0,0)$ to the point $(k,n)$ using
steps $(1,0)$ and $(0,1)$ (see Figure~2 for a lattice path from
$(0,0)$ to $(5,3)$)
 and satisfies the recurrence relation
\begin{equation}\label{rec+Pas+square}
    \mathcal
{P}^{\ulcorner}_{n,k}=\mathcal {P}^{\ulcorner}_{n,k-1}+\mathcal
{P}^{\ulcorner}_{n-1,k}.
\end{equation}

\begin{center}
\setlength{\unitlength}{1.2cm}
\begin{picture}(12,5.1)(-6.3,-2)
\thicklines\put(-3,-1){\line(1,0){6}}
\thicklines\put(-3,-1){\line(0,1){4}}\thicklines\put(-2.975,-0.975){\line(0,1){1}}\thicklines\put(-2.985,-0.985){\line(0,1){1}}

\thicklines\put(-3,0){\line(1,0){2}}\thicklines\put(-3,0.025){\line(1,0){2}}\thicklines\put(-3,0.015){\line(1,0){2}}
\thicklines\put(-1,0){\line(0,1){1}}\thicklines\put(-0.975,0.025){\line(0,1){1}}\thicklines\put(-0.985,0.015){\line(0,1){1}}
\thicklines\put(-1,1){\line(1,0){1}}\thicklines\put(-1,1.025){\line(1,0){1}}\thicklines\put(-1,1.015){\line(1,0){1}}
\thicklines\put(0,1){\line(0,1){1}}\thicklines\put(0.025,1.025){\line(0,1){1}}\thicklines\put(0.015,1.015){\line(0,1){1}}
\thicklines\put(0,2){\line(1,0){2}}\thicklines\put(0,2.025){\line(1,0){2}}\thicklines\put(0,2.015){\line(1,0){2}}

\put(-2,-1){\circle*{0.12}}\put(-1,-1){\circle*{0.12}}\put(0,-1){\circle*{0.12}}\put(1,-1){\circle*{0.12}}
\put(2,-1){\circle*{0.12}}

\put(-2.99,-1){\circle*{0.12}}\put(-2.99,0.01){\circle*{0.12}}\put(-3,1){\circle*{0.12}}\put(-3,2){\circle*{0.12}}

\put(-2,0.01){\circle*{0.12}}\put(-0.99,0.01){\circle*{0.12}}\put(-0.98,1.01){\circle*{0.12}}\put(0.01,1.01){\circle*{0.12}}
\put(0.01,2.01){\circle*{0.12}}\put(1,2.01){\circle*{0.12}}\put(2,2.01){\circle*{0.12}}
\scriptsize
\thicklines\put(2.63,-1.05){\textbf{$\longrightarrow$}}\thicklines\put(-3.06,2.85){\textbf{$\uparrow$}}
\put(-3.2,2.9){$y$}\put(2.9,-1.2){$x$}
\put(-3.65,-1.2){$(0,$}\put(-3.35,-1.2){$0)$}
\put(-3.65,-0.05){$(0,$}\put(-3.35,-0.05){$1)$}\put(-3.65,0.95){$(0,$}\put(-3.35,0.95){$2)$}
\put(-3.65,1.95){$(0,$}\put(-3.35,1.95){$3)$}
\put(-2.26,-1.3){$(1,$}\put(-1.93,-1.3){$0)$}\put(-1.26,-1.3){$(2,$}\put(-0.96,-1.3){$0)$}
\put(-0.26,-1.3){$(3,$}\put(0.07,-1.3){$0)$}\put(0.74,-1.3){$(4,$}\put(1.07,-1.3){$0)$}
\put(1.74,-1.3){$(5,$}\put(2.07,-1.3){$0)$}

\put(2.1,1.94){$(5,$}\put(2.43,1.94){$3)$}
\normalsize
\put(-3.5,-2){Figure~2: A lattice path from $(0,0)$ to $(5,3)$.}
\end{picture}
\end{center}
In fact, $\mathcal {P}^{\ulcorner}=\mathcal {P}\mathcal {P}^T$. Note
that the product of two matrices preserves their total positivity.
So the Pascal square $\mathcal {P}^{\ulcorner}$ is also totally
positive.

The Delannoy number $D_{n,k}$ from the name of Henri Delannoy (see
\cite{BS05} for historical remarks) is the number of lattice paths
in $\mathbb{N}^2$ from $(0,0)$ to $(k,n)$ using horizontal steps
$(1,0)$, vertical steps $(0,1)$, and diagonal steps $(1,1)$ (see
Figure~3 for a lattice path from $(0,0)$ to $(5,3)$).

\begin{center}
\setlength{\unitlength}{1.2cm}
\begin{picture}(12,5.1)(-6.3,-2)
\thicklines\put(-3,-1){\line(1,0){6}}
\thicklines\put(-3,-1){\line(0,1){4}}
\thicklines\put(-3,-1){\line(0,1){4}}
\thicklines\put(-3,-1){\line(1,1){1}}\thicklines\put(-3.005,-0.975){\line(1,1){0.97}}\thicklines\put(-3.007,-0.965){\line(1,1){0.97}}
\thicklines\put(-2,0){\line(1,0){1}}\thicklines\put(-2,0.025){\line(1,0){1}}\thicklines\put(-2,0.015){\line(1,0){1}}
\thicklines\put(-1,0){\line(0,1){1}}\thicklines\put(-0.975,0.025){\line(0,1){1}}\thicklines\put(-0.985,0.015){\line(0,1){1}}
\thicklines\put(-1,1){\line(1,0){2}}\thicklines\put(-1,1.025){\line(1,0){2}}\thicklines\put(-1,1.015){\line(1,0){2}}
\thicklines\put(1,1){\line(1,1){1}}\thicklines\put(0.995,1.025){\line(1,1){0.97}}\thicklines\put(0.993,1.035){\line(1,1){0.97}}

\put(-2,-1){\circle*{0.13}}\put(-1,-1){\circle*{0.13}}\put(0,-1){\circle*{0.13}}\put(1,-1){\circle*{0.13}}
\put(2,-1){\circle*{0.13}}

\put(-3,-1){\circle*{0.13}}\put(-3,0){\circle*{0.13}}\put(-3,1){\circle*{0.13}}\put(-3,2){\circle*{0.13}}

\put(-2,0){\circle*{0.13}}\put(-1,0){\circle*{0.13}}\put(-0.99,1){\circle*{0.13}}\put(0,1.01){\circle*{0.13}}
\put(1,1.02){\circle*{0.13}}\put(2,2){\circle*{0.13}}
\scriptsize
\thicklines\put(2.63,-1.05){\textbf{$\longrightarrow$}}\thicklines\put(-3.06,2.85){\textbf{$\uparrow$}}
\put(-3.2,2.9){$y$}\put(2.9,-1.2){$x$}
\put(-3.65,-1.2){$(0,$}\put(-3.35,-1.2){$0)$}
\put(-3.65,-0.05){$(0,$}\put(-3.35,-0.05){$1)$}\put(-3.65,0.95){$(0,$}\put(-3.35,0.95){$2)$}
\put(-3.65,1.95){$(0,$}\put(-3.35,1.95){$3)$}
\put(-2.26,-1.3){$(1,$}\put(-1.93,-1.3){$0)$}\put(-1.26,-1.3){$(2,$}\put(-0.96,-1.3){$0)$}
\put(-0.26,-1.3){$(3,$}\put(0.07,-1.3){$0)$}\put(0.74,-1.3){$(4,$}\put(1.07,-1.3){$0)$}
\put(1.74,-1.3){$(5,$}\put(2.07,-1.3){$0)$}

\put(2.1,1.94){$(5,$}\put(2.43,1.94){$3)$}
\normalsize
\put(-3.5,-2){Figure~3: A lattice path from $(0,0)$ to $(5,3)$.}
\end{picture}
\end{center}
 The Delannoy numbers share many similar properties with binomial
coefficients. For instance, $D_{n,k}$ satisfies the recurrence
relation
\begin{equation}\label{rec+Dela+squre}
    D_{n,k}=D_{n,k-1}+D_{n-1,k-1}+D_{n-1,k}
\end{equation}
with $D_{0,k} = D_{k,0}=1$. The matrix $\mathcal
{D}^{\ulcorner}=[D_{n,k}]_{n,k}$ is called {\it the Delannoy square}
and the infinite lower triangular matrix $\mathcal
{D}=[D_{n-k,k}]_{n,k}$ is called {\it the Delannoy triangle}. Let
$d_{n,k}=D_{n-k,k}$. Clearly, $d_{n,k}$ counts the number of lattice
paths in $\mathbb{N}^2$ from $(0,0)$ to $(k,n-k)$ using steps $(1,
0)$, $(0, 1)$ and $(1,1)$ (see Figure~4 for a lattice path from
$(0,0)$ to $(5,3)$) and satisfies
\begin{equation}\label{rec+dela+tri}
 d_{n,k}=d_{n-2,k-1}+d_{n-1,k-1}+d_{n-1,k},
\end{equation}
where $d_{0,0}=1$ and $d_{n,k}=0$ unless $n\geq k \geq0$. Brenti
\cite{Bre95} proved that both $\mathcal {D}^{\ulcorner}$ and
$\mathcal {D}$ are totally positive. Let $D_i=(d_{n+\delta
i,k+\sigma i})_{i}$, where $n\geq k$. Then for integers
$\sigma>\delta>0$ and $\sigma>k$, using a combinatorial
interpretation, Yu \cite{Yu09} proved that the Toeplitz matrix
$[D_{i-j}]_{i,j\geq0}$ is totally positive. Recently, Wang, Zheng
and Chen also studied some analytic properties of Delannoy numbers
\cite{WZC19}.

\begin{center}
\setlength{\unitlength}{1.2cm}
\begin{picture}(12,5.1)(-6.3,-2)
\thicklines\put(-3,-1){\line(1,0){6}}
\thicklines\put(-3,-1){\line(0,1){4}}

\thicklines\put(-3,-1){\line(1,1){1}}\thicklines\put(-3.005,-0.975){\line(1,1){0.97}}\thicklines\put(-3.007,-0.965){\line(1,1){0.97}}
\thicklines\put(-2,0){\line(1,0){2}}\thicklines\put(-2,0.025){\line(1,0){2}}\thicklines\put(-2,0.015){\line(1,0){2}}
\thicklines\put(0,0){\line(0,1){2}}\thicklines\put(0.025,0.025){\line(0,1){2}}\thicklines\put(0.015,0.015){\line(0,1){2}}
\thicklines\put(0,2){\line(1,0){2}}\thicklines\put(0,2.025){\line(1,0){2}}\thicklines\put(0,2.015){\line(1,0){2}}

\put(-2,-1){\circle*{0.12}}\put(-1,-1){\circle*{0.12}}\put(0,-1){\circle*{0.12}}\put(1,-1){\circle*{0.12}}
\put(2,-1){\circle*{0.12}}

\put(-3,-1){\circle*{0.12}}\put(-3,0){\circle*{0.12}}\put(-3,1){\circle*{0.12}}\put(-3,2){\circle*{0.12}}

\put(-2,0){\circle*{0.12}}\put(-1,0.01){\circle*{0.12}}\put(0.01,0.01){\circle*{0.12}}\put(0.01,1){\circle*{0.12}}
\put(0.015,2.01){\circle*{0.12}}\put(1,2.01){\circle*{0.12}}\put(2,2.01){\circle*{0.12}}
\scriptsize
\thicklines\put(2.63,-1.05){\textbf{$\longrightarrow$}}\thicklines\put(-3.06,2.85){\textbf{$\uparrow$}}
\put(-3.2,2.9){$y$}\put(2.9,-1.2){$x$}
\put(-3.65,-1.2){$(0,$}\put(-3.35,-1.2){$0)$}
\put(-3.65,-0.05){$(0,$}\put(-3.35,-0.05){$1)$}\put(-3.65,0.95){$(0,$}\put(-3.35,0.95){$2)$}
\put(-3.65,1.95){$(0,$}\put(-3.35,1.95){$3)$}
\put(-2.26,-1.3){$(1,$}\put(-1.93,-1.3){$0)$}\put(-1.26,-1.3){$(2,$}\put(-0.96,-1.3){$0)$}
\put(-0.26,-1.3){$(3,$}\put(0.07,-1.3){$0)$}\put(0.74,-1.3){$(4,$}\put(1.07,-1.3){$0)$}
\put(1.74,-1.3){$(5,$}\put(2.07,-1.3){$0)$}

\put(2.1,1.94){$(5,$}\put(2.43,1.94){$3)$}
\normalsize
\put(-3.5,-2){Figure~4: A lattice path from $(0,0)$ to $(5,3)$.}
\end{picture}
\end{center}

In \cite{Bre95}, more generally, Brenti considered a class of
recursive matrices. Precisely, let
 $\textbf{x}=\{x_n\}_{n\geq0}$,
$\textbf{y}=\{y_n\}_{n\geq0}$ and $\textbf{z}=\{z_n\}_{n\geq0}$ be
sets of indeterminates and define an infinite matrix $
A=[A_{n,k}]_{n,k\geq0}$ by
\begin{equation}\label{rec+Brenti}
A_{0,0}=1,\qquad
A_{n,k}=z_{n}A_{n-t,k-1}+y_{n}A_{n-1-t,k-1}+x_{n}A_{n-1,k}
\end{equation}
for $n+k\geq1$, where $t\in \mathbb{N}$, $A_{n,k}=0$ if either $n<0$
or $k<0$. Obviously, this recursive matrix $A$ is a generalization
of the Pascal triangle, the Delannoy triangle, and their square
matrices. Using a planar network method, Brenti \cite[Theorem
4.3]{Bre95} proved that the matrix $A$ and the Toeplitz matrix of
the each row sequence of $A$ are $(\textbf{x,y,z})$-totally
positive.

Motivated by these, in this paper, we introduce a generalized
lattice path as follows:
\begin{definition} \label{def+path}
Fix integers $t,\ell\in \mathbb{N}$ and let
$a^{(i)}_n(\textbf{x})$  and $ b_n(\textbf{y})$ be
 polynomials with nonnegative coefficients in $\textbf{x}$ and $\textbf{y}$, respectively.
Define $\mathscr{M}_{n,k}$ to be the set of lattice
paths in $\mathbb{N}^2$ from $(0,0)$ to $(k,n)$ with steps $(0,1)$
and $(1,t+i)$ for $0\leq i\leq \ell$, where each step $(0,1)$ from
height~$n-1$ gets the weight~$b_n(\textbf{y})$ and each step $(1,t+i)$ from
height~$n-t-i$ gets the weight $a_n^{(i)}(\textbf{x})$.
\end{definition}

For example, for $t=0$ and $\ell=2$, one can find a weighted lattice
path from $(0,0)$ to $(5,4)$ in Figure $5$.

\begin{center}
\setlength{\unitlength}{1.2cm}
\begin{picture}(12,6.1)(-6.3,-2)
\thicklines\put(-3,-1){\line(1,0){6}}
\thicklines\put(-3,-1){\line(0,1){5}}

\thicklines\put(-3,-1){\line(1,1){1}}\thicklines\put(-3.005,-0.975){\line(1,1){0.97}}\thicklines\put(-3.007,-0.965){\line(1,1){0.97}}
\thicklines\put(-2,0){\line(1,0){1}}\thicklines\put(-2,0.025){\line(1,0){1}}\thicklines\put(-2,0.015){\line(1,0){1}}
\thicklines\put(1,2){\line(0,1){1}}\thicklines\put(1.025,2.025){\line(0,1){1}}\thicklines\put(1.015,2.015){\line(0,1){1}}
\thicklines\put(1,3){\line(1,0){1}}\thicklines\put(1,3.025){\line(1,0){1}}\thicklines\put(1,3.015){\line(1,0){1}}
\thicklines\put(0,2){\line(1,0){1}}\thicklines\put(0,2.025){\line(1,0){1}}\thicklines\put(0,2.015){\line(1,0){1}}
\thicklines\put(-1,0){\line(1,2){1}}\thicklines\put(-1.009,0.035){\line(1,2){1}}\thicklines\put(-1.019,0.020){\line(1,2){1}}

\put(-2,-1){\circle*{0.12}}\put(-1,-1){\circle*{0.12}}\put(0,-1){\circle*{0.12}}\put(1,-1){\circle*{0.12}}
\put(2,-1){\circle*{0.12}}

\put(-3,-1){\circle*{0.12}}\put(-3,0){\circle*{0.12}}\put(-3,1){\circle*{0.12}}\put(-3,2){\circle*{0.12}}\put(-3,3){\circle*{0.12}}

\put(-2,0){\circle*{0.12}}\put(-1,0.01){\circle*{0.12}}\put(1.01,3){\circle*{0.12}}
\put(0.015,2.01){\circle*{0.12}}\put(1.01,2.01){\circle*{0.12}}\put(2,3.01){\circle*{0.12}}
\scriptsize
\thicklines\put(2.63,-1.05){\textbf{$\longrightarrow$}}\thicklines\put(-3.06,3.85){\textbf{$\uparrow$}}
\put(-3.2,3.9){$y$}\put(2.9,-1.2){$x$}
\put(-3.65,-1.2){$(0,$}\put(-3.35,-1.2){$0)$}
\put(-3.65,-0.05){$(0,$}\put(-3.35,-0.05){$1)$}\put(-3.65,0.95){$(0,$}\put(-3.35,0.95){$2)$}
\put(-3.65,1.95){$(0,$}\put(-3.35,1.95){$3)$}\put(-3.65,2.95){$(0,$}\put(-3.35,2.95){$4)$}
\put(-2.26,-1.3){$(1,$}\put(-1.93,-1.3){$0)$}\put(-1.26,-1.3){$(2,$}\put(-0.96,-1.3){$0)$}
\put(-0.26,-1.3){$(3,$}\put(0.07,-1.3){$0)$}\put(0.74,-1.3){$(4,$}\put(1.07,-1.3){$0)$}
\put(1.74,-1.3){$(5,$}\put(2.07,-1.3){$0)$}
{\color{blue}\put(-2.7,-0.3){$a_1^{(1)}$}}
{\color{blue}\put(-1.65,0.15){$a_1^{(0)}$}}
{\color{blue}\put(-0.99,1.15){$a_3^{(2)}$}}
{\color{blue}\put(0.1,2.15){$a_3^{(0)}$}}
{\color{blue}\put(1.05,3.15){$a_4^{(0)}$}}
{\color{blue}\put(0.76,2.5){$b_4$}}
\put(1.8,3){$(5,$}\put(2.13,3){$4)$}
\normalsize
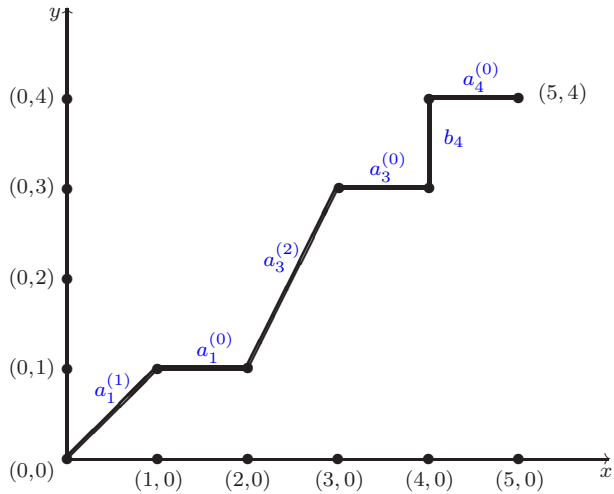
\put(-6,-2){Figure~5: A weighted
lattice path from $(0,0)$ to $(5,4)$ for $t=0,\ell=2$.}
\end{picture}
\end{center}

The \emph{weight} of a lattice path is the product of the weights of
its steps.
\begin{definition} \label{def+Matrix}
Define a matrix $M=[M_{n,k}]_{n,k}$, where $M_{n,k}$ denotes the sum
of the weights of all lattice paths in $\mathscr{M}_{n,k}$.
\end{definition}

One will see that the Pascal triangle, the Pascal square, the
Delannoy triangle, the Delannoy square and Brenti's recursive
matrices can be viewed as the special cases of the matrix $M$. In
addition, $M$ is also a generalization of more other famous
combinatorial matrices, such as the signless Stirling triangle of
the first kind, the Jacobi-Stirling triangle of the first kind, and
so on.

The main objective of this paper is to study total positivity from
the matrix $M$. We will extend positivity properties in literature
for the Pascal triangle, the Delannoy triangle and Brenti's
recursive matrices to those of $M$ and present some criteria for
total positivity from $M$. Applying our results, we will deal with
total
positivity of many combinatorial matrices in a unified manner. 

\subsection{Background and definitions for total positivity}
A matrix of real numbers is {\it \textbf{totally positive}} if all
its minors are nonnegative. Total positivity of matrices has
received considerable attention in recent times. This is mainly due
to its occurrence in several branches of mathematics, such as
classical analysis \cite{Sc30}, representation theory
\cite{Lus94,Lusztig_98,Lusztig_08,Rie03}, network analysis
\cite{Pos06}, cluster algebras \cite{BFZ96,FZ99}, positive
Grassmannians and integrable systems \cite{KW11,KW14}. In
combinatorics, more and more combinatorial matrices have been proved
to be totally positive, such as recursive matrices
\cite{Bre95,CLW15}, Riordan arrays \cite{CLW152,CW19,MMW22,Zhu21},
the Jacobi-Stirling triangle \cite{Mon12}, Delannoy-like triangles
\cite{MZ16}, Catalan-Stieltjes matrices \cite{PZ16}, Narayana
triangles of types $A$ and $B$ \cite{WY18}, and the generalized
Jacobi-Stirling triangle \cite{Zhu14}. We also refer the reader to
monographs \cite{Kar68,Pin10} for more details.

For a sequence $\textbf{a}=(a_k)_{k\ge 0}$ of real numbers, denote
by $\mathcal {T}(\textbf{a})=[a_{i-j}]_{i,j\ge 0}$ its Toeplitz
matrix. The Toeplitz-total positivity plays an important role in
different fields. For instance, total positivity of $\mathcal
{T}(\textbf{a})$ is closely related to the P\'olya frequency of
$\textbf{a}$. The sequence $\textbf{a}$ is called a {\it
\textbf{P\'olya frequency}} sequence if its infinite Toeplitz matrix
$\mathcal {T}(\textbf{a})$ is totally positive. In classical
analysis, one of important applications of the P\'olya frequency is
to characterize real rootedness of polynomials and entire functions.
For example, the fundamental representation theorem for P\'olya
frequency sequences states for $a_0=1$ that $\textbf{a}$  is a
P\'olya frequency sequence if and only if its generating function
has the form
$$\sum_{n\ge 0}a_nz^n
=\frac{\prod_{j\ge 1}(1+\alpha_jz)}{\prod_{j\ge
1}(1-\beta_jz)}e^{\gamma z}$$ in some open disk centered at the
origin, where $\alpha_j,\beta_j,\gamma\ge 0$ and $\sum_{j\ge
1}(\alpha_j+\beta_j)<+\infty$ (see Karlin~\cite[pp.~412]{Kar68} for
instance). In addition, from the positivity of minors of order $2$,
we immediately see that the P\'olya frequency of $\textbf{a}$
implies its log-concavity defined by $a_n^2- a_{n+1}a_{n-1}\geq0$
for $n\geq1$. Define a new sequence
$\mathcal{L}(\textbf{a})=(b_k)_{k\ge 0}$ by $b_0=a_0^2$ and
$b_{k+1}=a^2_{k+1}-a_{k}a_{k+2}$ for $k\geq0$. Then the sequence
$\textbf{a}$ is log-concave if and only if the sequence
$\mathcal{L}(\textbf{a})$ is nonnegative, i.e., all $b_k$ are
nonnegative. Generally, the sequence $\textbf{a}$ is called {\it
$r$-log-concave} if $\mathcal {L}^{k}(\textbf{a})$ is a nonnegative
sequence for all $1\leq k\leq r$. Call $\textbf{a}$ {\it infinitely
log-concave} if $\mathcal{L}^i(\textbf{a})$ is nonnegative for all
$i\geq 1$, where $\mathcal{L}^i=\mathcal{L}(\mathcal{L}^{i-1})$. In
fact, Br\"{a}ndr\'{e}n \cite{Bra09} proved the stronger property:
the P\'olya frequency of a finite sequence $\textbf{a}$ implies its
infinite log-concavity. Log-concave and P\'olya frequency sequences
often occur in many branches of mathematics and have been
extensively investigated. We refer the reader to Adiprasito, Huh and
Katz \cite{AH18}, Br\"and\'en \cite{Bra15}, Br\"{a}nd\'{e}n and Huh
\cite{BH20}, Brenti~\cite{Bre89,Bre94}, Eur and Huh \cite{EH20}, Huh
\cite{Huh12,Huh15}, Stanley \cite{Sta89}, and Wang and Yeh
\cite{WYjcta05,WY07} for the log-concavity and P\'olya frequency in
combinatorics.

Because some combinatorial objects with respect to one or more
statistics often generate multivariate polynomials, our more
interest is to study sequences and matrices of polynomials in one or
more indeterminates $\textbf{x}$. Let $\mathbb{R}$ denote the set of
all real numbers and $\textbf{x}=\{x_i\}_{i\in{I}}$ be a set of
indeterminates. Following Brenti \cite{Bre95}, a matrix with entries
in $\mathbb{R}[\textbf{x}]$ is
 \textbf{$\textbf{x}$-totally positive} if all its minors are
polynomials with nonnegative coefficients in the indeterminates
$\textbf{x}$ and is \textbf{$\textbf{x}$-totally positive of order
$r$} if all its minors of order $k\le r$ are polynomials with
nonnegative coefficients in the indeterminates $\textbf{x}$.
Clearly, $\textbf{x}$-total positivity (resp. $\textbf{x}$-total
positivity of order $r$) reduces to total positivity (resp. total
positivity of order $r$) for $\textbf{x}$ being nonnegative real
numbers. In addition, for $\textbf{x}$ being nonnegative real
numbers, $\textbf{x}$-total positivity of the Toeplitz matrix
implies that the original sequence is a P\'olya frequency and
log-concave sequence.

\subsection{Statements of main results}\label{Statements of main results}
For convenience, we give some notation. Let
$(a^{(i)}_n(\textbf{x}))_{n,i \ge 0}$ be the set of polynomials with
nonnegative coefficients in $\textbf{x}$. Define an $l$-banded
infinitely lower triangle
$\mathcal{A}=\left[\mathcal{A}_{i,j}\right]_{i,j}$, where
\begin{eqnarray*}
          \mathcal{A}_{i,j}  =\begin{cases} a_i^{(i-j)}(\textbf{x})&\text{for}~i-l\leq j\leq i;\\
                                               0  &\text{others},
                                 \end{cases}
\end{eqnarray*}
or in other words,
$$\mathcal{A}=\left[\mathcal{A}_{i,j}\right]_{i,j}=\left(
                   \begin{array}{ccccccc}
                                 a_0^{(0)} \\
                                 a_1^{(1)}& a_1^{(0)} \\
                                 a_2^{(2)}& a_2^{(1)} & a_2^{(0)}\\
                                 a_3^{(3)}& a_3^{(2)} & a_3^{(1)} & a_3^{(0)} \\
                                 \vdots & \vdots & \vdots & \vdots & \ddots\\
                                 a_\ell^{(\ell)}& a_\ell^{(\ell-1)} & a_\ell^{(\ell-2)} & a_\ell^{(\ell-3)}&\cdots&a_\ell^{(0)} \\
                                & \ddots & \ddots & \ddots &\ddots & \cdots& \ddots\\
                  \end{array}
              \right).$$

Our main result is as follows.

\begin{thm}\label{thm+main}
Let $M$ be the matrix in Definition \ref{def+Matrix}. Then we have
  \begin{itemize}
       \item [\rm (i)]
if $\mathcal{A}$ is $\textbf{x}$-totally positive of order $r$, then
 $M$ is $(\textbf{x},\textbf{y})$-totally positive of order $r$;
         \item [\rm (ii)]
if $\mathcal{A}$ is $\textbf{x}$-totally positive, then $M$ is $(\textbf{x},\textbf{y})$-totally positive.
 \end{itemize}
\end{thm}
We shall present an algebraic proof for Theorem \ref{thm+main} in
Section \ref{section+proof+thm+main}. In general, the matrix
$\mathcal{A}$ is not totally positive. In what follows we will
discuss total positivity of $\mathcal{A}$ under some restricted
conditions.

If $\ell=1$, then $\mathcal{A}$ reduces to a bidiagonal matrix and
is obviously $\textbf{x}$-totally positive. In this situation, we
establish the following consequence, which contains two results on
total positivity: one for the matrix $M$ and the other for the
Toeplitz matrix of the each row sequence of $M$.

\begin{thm}\label{thm+bidiag}
Let $M$ be the matrix in Definition \ref{def+Matrix}. If $\ell=1$,
then
\begin{itemize}
\item [\rm (i)]
the matrix $M$ is $(\textbf{x},\textbf{y})$-totally positive;
\item [\rm (ii)]
the Toeplitz matrix of the each row sequence of $M$ with $t\geq 1$ is
$(\textbf{x},\textbf{y})$-totally positive.
\end{itemize}
\end{thm}

If $\ell=2$, then $\mathcal{A}$ reduces to a tridiagonal matrix. In
general, a tridiagonal matrix is not totally positive. But we have
the following sufficient conditions.

\begin{thm}\label{thm+tridiag}
 Let $M$ be the matrix in Definition \ref{def+Matrix} with $\ell=2$.
If there exists polynomial sequences with
 nonnegative coefficients $(\alpha_n(\textbf{x}))_n$, $(\beta_n(\textbf{x}))_n$,
$(\lambda_n(\textbf{x}))_n$ and $(\mu_n(\textbf{x}))_n$ such that
\begin{itemize}
\item [\rm (i)]
   $a_n^{(0)}=\alpha_n\lambda_n$ for $n\geq0$;
\item [\rm (ii)]
   $a_n^{(1)}=\alpha_n\mu_n+\beta_n\lambda_{n-1}$ for $n\geq1$;
\item [\rm (iii)]
   $a_n^{(2)}=\beta_n\mu_{n-1}$ for $n\geq2$,
\end{itemize}
then we have
\begin{itemize}
\item [\rm (i)]
 the matrix $M$ is $(\textbf{x},\textbf{y})$-totally positive;
\item [\rm (ii)]
  the Toeplitz matrix of the each row  sequence of $M$ with $t\geq 1$ is
$(\textbf{x},\textbf{y})$-totally positive.
\end{itemize}
\end{thm}

Note that conclusion (i) of Theorem \ref{thm+tridiag} is immediate
from Theorem \ref{thm+main} (ii). In order to complete the proof of
Theorem \ref{thm+tridiag} (ii), in Section
\ref{section+proof+thm+tirdiag}, we will present a combinatorial
proof for (i) and (ii) using the Lindstr\"{o}m-Gessel-Viennot lemma,
which expresses the sum of the signed weight of non-intersecting
paths in a directed acyclic planar graph in terms of a determinant.

Taking $\alpha_n=a_n^{(0)},\lambda_n=1$ for $n\geq 0$, $\mu_n=0$ and
$\beta_n=a_n^{(1)}$ for $n\geq 1$ in Theorem \ref{thm+tridiag}, we
immediately obtain Theorem \ref{thm+bidiag}.

By Theorem \ref{thm+tridiag}, we also obtain:

\begin{cor}\label{coro+1.7}
Let $M$ be the matrix in Definition \ref{def+Matrix} with $\ell=2$.
Then both the matrix $M$ and the Toeplitz matrix of the each row sequence of $M$ with $t\geq 1$ are
$(\textbf{x},\textbf{y})$-totally positive under any of the following four conditions:
\begin{itemize}
\item [\rm (i)]
   $a_1^{(1)}=a_1^{(0)}$, $a_n^{(1)}=a_n^{(0)}+a_n^{(2)}$ for $n\geq 2$;
\item [\rm (ii)]
   $a_1^{(1)}=a_2^{(2)}$, $a_n^{(1)}=a_{n-1}^{(0)}+a_{n+1}^{(2)}$ for $n\geq 2$;
\item [\rm (iii)]
   $a_1^{(1)}=1$, $a_n^{(1)}=a_{n-1}^{(0)}a_{n}^{(2)}+1$ for $n\geq 2$;
\item [\rm (iv)]
   $a_1^{(1)}=a_{1}^{(0)}a_{2}^{(2)}$, $a_n^{(1)}=a_{n}^{(0)}a_{n+1}^{(2)}+1$ for $n\geq 2$.
\end{itemize}
\end{cor}

\textbf{Proof of  Corollary \ref{coro+1.7}:}
\begin{itemize}
  \item [\rm (i)]
  The desired result under the condition (i) follows from Theorem \ref{thm+tridiag} by taking
  $\alpha_n=a_n^{(0)},\lambda_n=1$ for $n\geq 0$, $\mu_n=1$ for $n\geq 1$,
  $\beta_1=0,\beta_n=a_n^{(2)}$ for $n\geq 2$.
  \item [\rm (ii)]
  The desired result under the condition (ii) follows from Theorem \ref{thm+tridiag} by taking
  $\alpha_n=1,\lambda_n=a_n^{(0)}$ for $n\geq 0$, $\mu_n=a_{n+1}^{(2)}$ for $n\geq 1$,
  $\beta_1=0,\beta_n=1$ for $n\geq 2$.
  \item [\rm (iii)]
  The desired result under the condition (iii) follows from Theorem \ref{thm+tridiag} by taking
  $\alpha_n=1,\lambda_n=a_n^{(0)}$ for $n\geq 0$, $\mu_n=1$ for $n\geq 1$,
  $\beta_1=0,\beta_n=a_n^{(2)}$ for $n\geq 2$.
  \item [\rm (iv)]
  The desired result under the condition (iv) follows from Theorem \ref{thm+tridiag} by taking
  $\alpha_n=a_n^{(0)},\lambda_n=1$ for $n\geq 0$, $\mu_n=a_{n+1}^{(2)}$ for $n\geq 1$,
  $\beta_1=0,\beta_n=1$ for $n\geq 2$.
\end{itemize}
This completes the proof of Corollary \ref{coro+1.7}.\qed

The Riordan array is an important tool in combinatorics and has many
 applications to a wide range of subjects, such as enumerative
combinatorics, combinatorial sums, recurrence relations, computer
science, and among other topics (see
\cite{CLW152,HS16,HS09,SGWW91,Spr94} for instance). Recall the
definition of the Riordan array. Let $f(z)=\sum_{n\geq0}f_nz^n$ and
$g(z)=\sum_{n\geq0}g_nz^n$ be two formal power series. A {\it
Riordan array}, denoted by $\mathcal
{R}\left(g(z),f(z)\right)=[R_{n,k}]_{n,k}$ is an infinite
matrix whose ordinary generating function of the $k$th column is
$g(z)f^k(z)$ for $k\geq0$. In other words, for $n,k\geq0$,
$$R_{n,k}=[z^n]g(z)f^k(z).$$
Then
\begin{eqnarray}\label{generating funtion+GR}
\sum_{n\geq0}\sum_{k\geq0}R_{n,k}q^k z^n=\frac{g(z)}{1-qf(z)}.
\end{eqnarray}
A Riordan array is called {\it proper} if $g_0f_1\neq0$ and $f_0=0$
and {\it improper} if $g_0f_0\neq0$. Generally, a proper Riordan
array is an infinite lower-triangular matrix and an improper Riordan
array is an infinite rectangular matrix. For instance, the Pascal triangle
is a proper Riordan array $\mathcal{R}\left(\frac{1}{1-z},\frac{z}{1-z}\right)$ and
the Pascal square is an improper Riordan array $\mathcal{R}\left(\frac{1}{1-z},\frac{1}{1-z}\right)$.
The Delannoy triangle is a proper Riordan array $\mathcal{R}\left(\frac{1}{1-z},\frac{z(1+z)}{1-z}\right)$
and the Delannoy square is an improper Riordan array $\mathcal{R}\left(\frac{1}{1-z},\frac{1+z}{1-z}\right)$.

Finally, for the matrix $\mathcal{A}$ with a general $\ell$, we
obtain the following result closely related to the Riordan array.

\begin{thm}\label{thm+fun+ell}
Let $M$ be the matrix in Definition \ref{def+Matrix} and
$(b_i(\textbf{y}))_{i\geq0}$ be a polynomial sequence with
nonnegative coefficients in $\textbf{y}$. If there exists
indeterminates $\balpha=(\alpha_i)_{i=1}^{\ell}$ and
$\bbeta=(\beta_i)_{i=1}^{\ell}$ such that the generating function
$$\sum_{i= 0}^{\ell}a^{(i)}_nz^i
=\prod_{j= 1}^{\ell}(\alpha_jz+\beta_j)$$ for all $n\in \mathbb{N}$,
then we have the following results:
  \begin{itemize}
\item [\rm (i)]
the matrix $M$ and the Toeplitz matrix of the each row sequence of $M$ for
$t\geq1$ are $(\balpha,\bbeta,\textbf{y})$-totally positive;
\item [\rm (ii)]
if $b_n=\gamma$ for $n\geq0$, then the Toeplitz matrix of the each
column sequence of $M$ is $(\balpha,\bbeta,\gamma)$-totally positive;
\item [\rm (iii)]
if $b_n=\gamma$ for $n\geq0$ and $t\geq1$, then the Toeplitz matrix
of the sequence $(M_{n+\delta i,k+\sigma i})_{i}$  is
$(\balpha,\bbeta,\gamma)$-totally positive for $0\leq k\leq n$,
$0<\delta$ and $max\{k,\delta\}<\sigma$;
\item [\rm (iv)]
if $b_n=\gamma$ for $n\geq0$, then the matrix $M$ is the Riordan array
 $\mathcal {R}\left(\frac{1}{1-\gamma z},\frac{\sum_{i=
0}^{\ell}a^{(i)}_nz^{t+i}}{1-\gamma z}\right)$;
\item [\rm (v)]
if $b_n=\gamma$ for $n\geq0$, then
$$
   M_{n,k}=
    \!\!\!
   \sum_{\begin{scarray}
            i\geq 0  \\[-1mm]
            0\leq c_1,c_2,\ldots,c_{\ell}\leq k  \\[-1mm]
            i+\sum\limits_{j=1}^{\ell}c_j=n-tk \\[-1mm]
         \end{scarray}
        }
   \!\!\!
   \binom{k+i}{i}{\gamma}^i\prod_{j= 1}^{\ell}\binom{k}{c_j}{\alpha_j}^{c_j}{\beta_j}^{k-c_j};$$
\item [\rm (vi)]
if $b_n=\gamma$ for $n\geq0$, then we have
$$\sum_{n\geq0}\sum_{k\geq0}M_{n,k}q^kz^n=\frac{1}{1-\gamma z-q\sum_{i= 0}^{\ell}a^{(i)}_nz^{t+i}}.$$
 \end{itemize}
\end{thm}

The Riordan array result in Theorem \ref{thm+fun+ell} (iv) also
arose in \cite{RS15}. We shall arrange the proof of Theorem
\ref{thm+fun+ell} in Section \ref{section+Proof+Thm+full+ell}. It is
well-known that a finite sequence $s_0,s_1,\ldots,s_n$ is a P\'olya
frequency sequence if and only if $\sum_{k= 0}^ns_kz^k$ has only
non-positive real zeros (\cite[pp.~399]{Kar68}). Therefore, the
following consequence is immediate from (i), (ii) and (iii) of
Theorem \ref{thm+fun+ell}.

\begin{prop}\label{main+proppppp}
Let $b_n$ be a nonnegative constant.
If for any two different $n$, polynomials $\sum_{i=
0}^{\ell}a^{(i)}_nz^i$ have $\ell$ same non-positive real zeros,
then
\begin{itemize}
\item [\rm (i)]
the matrix $M$ in Definition \ref{def+Matrix} is totally positive;
\item [\rm (ii)]
each column of $M$ is a P\'olya frequency and infinitely log-concave
sequence;
\item [\rm (iii)]
each row of $M$ for $t\geq1$ is a P\'olya frequency and infinitely
log-concave sequence;
\item [\rm (iv)]
if $0\leq k\leq n$, $0<\delta$ and $max\{k,\delta\}<\sigma$, then
$(M_{n+\delta i,k+\sigma i})_{i}$  is a P\'olya frequency and
infinitely log-concave sequence for $t\geq1$.
\end{itemize}
\end{prop}

\section{The proof of Theorem \ref{thm+main}}\label{section+proof+thm+main}
By Definition \ref{def+path}, we immediately have the following
recurrence relation for the matrix $M$ in Definition
\ref{def+Matrix}.

\begin{prop}\label{prop+rec+1+M}
Let $M=[M_{n,k}]_{n,k}$ be the matrix in Definition \ref{def+Matrix}
. For $n\geq t$, $M_{n,k}$ satisfies the recurrence relation
\begin{equation}\label{rec+1+M}
M_{n,k}=a^{(0)}_n(\textbf{x})M_{n-t,k-1}+a^{(1)}_n(\textbf{x})M_{n-t-1,k-1}+\ldots+a^{(\ell)}_n(\textbf{x})M_{n-t-\ell,k-1}+b_n(\textbf{y})M_{n-1,k},\\
\end{equation}
where initial conditions $M_{0,0}=1$, $M_{n,k}=0$ for $n<0$ or $k<0$, and
$$M_{n,k} =\begin{cases}
                                (a^{(0)}_0)^{k}      & \text{for} ~t=0,~n=0,~k\geq1;\\
                                \prod\limits_{i=0}^{n-1}b_{n-i}      & \text{for}~t\geq2,~1\leq n \leq t-1,~k=0;\\
                                 0                                    & \text{for}~t\geq1,~0\leq n \leq t-1,~k\geq1.
         \end{cases}$$
\end{prop}

Furthermore, we also have another recurrence relation for $M_{n,k}$
as follows.
\begin{prop}\label{prop+rec+2+M}
Let $M=[M_{n,k}]_{n,k}$ be the matrix in Definition
\ref{def+Matrix}. Then $M_{n,k}$ satisfies the recurrence relation
\begin{equation*}
M_{n,k}=\sum\limits_{j=0}^{\ell}\left(\sum\limits_{m=0}^ja_{n-m}^{(j-m)}\prod\limits_{i=0}^{m-1}b_{n-i}\right)M_{n-t-j,k-1}+\sum\limits_{j\geq0}\left(\sum\limits_{m=0}^{\ell}a_{n-j-m-1}^{(l-m)}\prod\limits_{i=0}^{j+m}b_{n-i}\right)M_{n-t-\ell-j-1,k-1},
\end{equation*}
where $\prod\limits_{i=0}^{-1}b_{n-i}=1$, $M_{0,0}=1$ and
$M_{n,0}=\prod\limits_{i=1}^{n}b_i$ for $n\geq 1 $.
\end{prop}
\begin{proof}
We will present its two different proofs: one is induction, and the
other is a combinatorial proof.

\textbf{Induction:} We will give the first proof by induction on
$n$. First, we show that it holds for the first $t+1$ rows in terms
of the following two cases $t=0$ and $t\geq 1$:

Case 1: For $t=0$, we have $M_{0,0}=1$, $M_{0,k}=(a^{(0)}_0)^{k}$
for $k\geq1$, which is consistent with Definition \ref{def+Matrix}.

Case 2: For $t\geq 1$, we obtain
$$M_{n,k} =\begin{cases}
                                1     & \text{for} ~k=0, n=0;\\
                                \prod\limits_{i=1}^{n}b_{i}      & \text{for} ~k=0,~1\leq n \leq t;\\
                                    0                             & \text{for} ~k=1,~0\leq n \leq t-1;\\
                                 a^{(0)}_{t}                     & \text{for}~k=1,~n=t;\\
                                   0     & \text{for}~k\geq2,~0\leq n \leq t,\\
         \end{cases}$$
which is consistent
with Definition \ref{def+Matrix}.

So we proceed to the inductive step ($n\geq t+1$). Assume that it
holds for $n-1$. Then we will prove it holds for $n$. It is obvious
that $M_{n,0}=\prod\limits_{i=1}^{n}b_i$. By Proposition
\ref{prop+rec+1+M} and the induction hypothesis, we have
\begin{eqnarray}\label{eq+E+indu}
&&M_{n,k}\nonumber\\
&=&\sum_{j=0}^{\ell}a^{(j)}_nM_{n-t-j,k-1}+b_nM_{n-1,k}\nonumber\\
&=&\sum_{j=0}^{\ell}a^{(j)}_nM_{n-t-j,k-1}+b_n\sum\limits_{j=0}^{\ell}\left(\sum\limits_{m=0}^ja_{n-m-1}^{(j-m)}\prod\limits_{i=0}^{m-1}b_{n-1-i}\right)M_{n-t-j-1,k-1}+\nonumber\\
&&b_n\sum\limits_{j\geq0}\left(\sum\limits_{m=0}^{\ell}a_{n-j-m-2}^{(l-m)}\prod\limits_{i=0}^{j+m}b_{n-i-1}\right)M_{n-t-\ell-j-2,k-1}\nonumber\\
&=&\sum_{j=0}^{\ell}a^{(j)}_nM_{n-t-j,k-1}+\sum\limits_{j=0}^{\ell}\left(\sum\limits_{m=0}^ja_{n-m-1}^{(j-m)}\prod\limits_{i=0}^{m}b_{n-i}\right)M_{n-t-j-1,k-1}+\nonumber\\
&&\sum\limits_{j\geq0}\left(\sum\limits_{m=0}^{\ell}a_{n-j-m-2}^{(l-m)}\prod\limits_{i=0}^{j+m+1}b_{n-i}\right)M_{n-t-\ell-j-2,k-1}\nonumber\\
&=&\underbrace{\sum_{j=0}^{\ell}a^{(j)}_nM_{n-t-j,k-1}+\sum\limits_{j=0}^{\ell-1}\left(\sum\limits_{m=0}^ja_{n-m-1}^{(j-m)}\prod\limits_{i=0}^{m}b_{n-i}\right)M_{n-t-j-1,k-1}}+\nonumber\\
&&\underbrace{\left(\sum\limits_{m=0}^{\ell}a_{n-m-1}^{(\ell-m)}\prod\limits_{i=0}^{m}b_{n-i}\right)M_{n-t-\ell-1,k-1}+\sum\limits_{j\geq0}\sum\limits_{m=0}^{\ell}a_{n-j-m-2}^{(l-m)}\prod\limits_{i=0}^{j+m+1}b_{n-i}M_{n-t-\ell-j-2,k-1}}\nonumber\\
\end{eqnarray} for $n\geq t+1$ and $k\geq 1$.

Taking $m+1\rightarrow m$ and $j+1\rightarrow j$, we have
\begin{eqnarray}\label{eq+E+indu+1}
\sum\limits_{j=0}^{\ell-1}\left(\sum\limits_{m=0}^ja_{n-m-1}^{(j-m)}\prod\limits_{i=0}^{m}b_{n-i}\right)M_{n-t-j-1,k-1}&=&\sum\limits_{j=1}^{\ell}\left(\sum\limits_{m=1}^{j}a_{n-m}^{(j-m)}\prod\limits_{i=0}^{m-1}b_{n-i}\right)M_{n-t-j,k-1}.\nonumber\\
\end{eqnarray}
Similarly, taking $j+1\rightarrow j$, we also have
\begin{eqnarray}\label{eq+E+indu+2}
\sum\limits_{j\geq0}\sum\limits_{m=0}^{\ell}a_{n-j-m-2}^{(\ell-m)}\prod\limits_{i=0}^{j+m+1}b_{n-i}M_{n-t-\ell-j-2,k-1}
=\sum\limits_{j\geq1}\sum\limits_{m=0}^{\ell}a_{n-j-m-1}^{(\ell-m)}\prod\limits_{i=0}^{j+m}b_{n-i}M_{n-t-\ell-j-1,k-1}.\nonumber\\
\end{eqnarray}
In consequence, combining (\ref{eq+E+indu}), (\ref{eq+E+indu+1}) and
(\ref{eq+E+indu+2}), we obtain
\begin{eqnarray}\label{eq+E41}
&&M_{n,k}\nonumber\\
&=&\underbrace{\sum_{j=0}^{\ell}a^{(j)}_nM_{n-t-j,k-1}+\sum\limits_{j=1}^{\ell}\left(\sum\limits_{m=1}^{j}a_{n-m}^{(j-m)}\prod\limits_{i=0}^{m-1}b_{n-i}\right)M_{n-t-j,k-1}}+\nonumber\\
&&\underbrace{\left(\sum\limits_{m=0}^{\ell}a_{n-m-1}^{(\ell-m)}\prod\limits_{i=0}^{m}b_{n-i}\right)M_{n-t-\ell-1,k-1}+\sum\limits_{j\geq1}\left(\sum\limits_{m=0}^{\ell}a_{n-j-m-1}^{(\ell-m)}\prod\limits_{i=0}^{j+m}b_{n-i}\right)M_{n-t-\ell-j-1,k-1}}\nonumber\\
&=&\sum\limits_{j=0}^{\ell}\left(\sum\limits_{m=0}^ja_{n-m}^{(j-m)}\prod\limits_{i=0}^{m-1}b_{n-i}\right)M_{n-t-j,k-1}+\sum\limits_{j\geq0}\left(\sum\limits_{m=0}^{\ell}a_{n-j-m-1}^{(l-m)}\prod\limits_{i=0}^{j+m}b_{n-i}\right)M_{n-t-\ell-j-1,k-1}.\nonumber
\end{eqnarray}
This completes the proof by induction on $n$.

\textbf{Combinatorial proof:} For each lattice path $\mathcal {L}$
from the vertex $(0,0)$ to the vertex $(k,n)$, in terms of steps
used, we can see that $\mathcal {L}$ must first arrive at the line
$x=k-1$. So denote by $(k-1,z)$ the intersection of $\mathcal {L}$
and the line $x=k-1$. Obviously, $0\leq z\leq n-t$. Then we can
divide $\mathcal {L}$ into two paths $\mathcal {L}_1$ and $\mathcal
{L}_2$, where $\mathcal {L}_1$ is one path from $(0,0)$ to $(k-1,z)$
and $\mathcal {L}_2$ is one path from $(k-1,z)$ to $(k,n)$. We will
consider the path $\mathcal {L}$ in terms of the following two cases
for $z$:

(1) For $n-t-\ell\leq z \leq n-t$, let $z=n-t-j$ $(0\leq
j\leq\ell)$. So $\mathcal {L}_1$ is one path from $(0,0)$ to
$(k-1,n-t-j)$ and $\mathcal {L}_2$ is one path from $(k-1,n-t-j)$ to
$(k,n)$. Let $0\leq m\leq j$, which implies $n-j\leq n-m\leq n$. For
$\mathcal {L}_2$, it must first use one step $(k-1,n-t-j)\rightarrow
(k,n-m)$, then use other steps $(k,n-m+i)\rightarrow (k,n-m+i+1)$
for $0\leq i\leq m-1$.

 (2) For $0\leq z \leq n-t-\ell-1$, let $z=n-t-\ell-j-1$ $(j\geq0)$.
So $\mathcal {L}_1$ is one path from $(0,0)$ to $(k-1,n-t-\ell-j-1)$
and $\mathcal {L}_2$ is one path from $(k-1,n-t-\ell-j-1)$ to
$(k,n)$. Let $0\leq m\leq \ell$. So $n-\ell-j-1\leq n-m-j-1\leq
n-j-1$. For $\mathcal {L}_2$, it must first use one step
$(k-1,n-t-\ell-j-1)\rightarrow (k,n-m-j-1)$, then use other steps
$(k,n-m-j-1+i)\rightarrow (k,n-m-j+i)$ for $0\leq i\leq m+j$.

Hence, we obtain
\begin{eqnarray}
&&M_{n,k}\nonumber\\
&=&\sum\limits_{j=0}^{\ell}M_{n-t-j,k-1}\left(\sum\limits_{m=0}^ja_{n-m}^{(j-m)}\prod\limits_{i=0}^{m-1}b_{n-i}\right)+\sum\limits_{j\geq0}M_{n-t-\ell-j-1,k-1}\left(\sum\limits_{m=0}^{\ell}a_{n-j-m-1}^{(l-m)}\prod\limits_{i=0}^{j+m}b_{n-i}\right)\nonumber
\end{eqnarray}
for $n\geq t$ and $k\geq 1$. The combinatorial proof is complete.
\end{proof}

\textbf{Proof of  Theorem \ref{thm+main}:} (ii) follows from (i).
Therefore we only need to prove that (i) holds. We will divide its
proof into two cases in terms of $t=0$ or $t\geq1$:

 Case 1: Assume $t=0$.  Let
  ${\overrightarrow{M}}_{k}$ denote the matrix consisting of columns
  from $0$th to $k$th of $M$ and define a matrix $$P=[P_{n,k}]_{n,k},$$
 where
  \begin{eqnarray}\label{formula+P}
         P_{n,k} =\begin{cases}   1                                   & \text{for} ~ k=0,~n=0;\\
                                 \prod\limits_{i=0}^{n-1}b_{n-i}      & \text{for} ~k=0,~n\geq1;\\
                                 0                                    & \text{for} ~ k\geq1,~n<k-1;\\
                                 \sum\limits_{m\geq0}a_{n-m}^{(n-k-m+1)}\prod\limits_{i=0}^{m-1}b_{n-i} & \text{for} ~ k\geq1,~k-1\leq n\leq k+\ell-1;\\
                                 \sum\limits_{m=0}^{\ell}a_{k+\ell-m-1}^{(\ell-m)}\prod\limits_{i=0}^{n+m-k-\ell}b_{n-i} & \text{for} ~k\geq1,~ n>k+\ell-1,
                    \end{cases}
  \end{eqnarray}
  or in other words,
  $$
               P=\left(
                        \begin{array}{cccccc}
                                      1&a^{(0)}_0\\
                                      b_1&a^{(1)}_1+b_1a^{(0)}_0&a^{(0)}_1\\
                                      b_1b_2&a^{(2)}_{2}+b_2a^{(1)}_1+b_2b_1a^{(0)}_0&a^{(1)}_{2}+b_2a^{(0)}_1&a^{(0)}_2\\
                                      \vdots&\vdots&\vdots&\vdots&\ddots\\
                         \end {array}
                 \right).
  $$

  By Proposition \ref{prop+rec+2+M}, we immediately have
  \begin{equation}\label{decomposition}
     {\overrightarrow{M}}_{k}= P \left(
                                   \begin{array}{cc}
                                                  1&0\\
                                                  0&{\overrightarrow{M}}_{k-1}\\
                                   \end{array}
                                 \right).
  \end{equation}

  Let
 $$\Delta:=\left(
                   \begin{array}{cccccccc}
                                1& a_0^{(0)} \\
                                0& a_1^{(1)}& a_1^{(0)} \\
                                0& a_2^{(2)}& a_2^{(1)} & a_2^{(0)}\\
                                0& a_3^{(3)}& a_3^{(2)} & a_3^{(1)} & a_3^{(0)} \\
                              \vdots&\vdots & \vdots & \vdots & \vdots & \ddots\\
                               0&  a_\ell^{(\ell)}& a_\ell^{(\ell-1)} & a_\ell^{(\ell-2)} & a_\ell^{(\ell-3)}&\cdots&a_\ell^{(0)} \\
                               \vdots && \ddots & \ddots & \ddots &\ddots & \cdots& \ddots\\
                  \end{array}
              \right)=(1,\mathcal{A}).$$
 Then it is not hard to check the following connection between matrices $P$
 and $\Delta$:
 \begin{equation}\label{connection}
 P=\prod\limits_{i\geq0} E_{i+1,i}[b_{i+1}]\Delta,
 \end{equation}
  where $E_{i+1,i}[b_{i+1}]$
  denotes the matrix with elements $e_{i+1,i}=b_{i+1},e_{i,i}=1$ and otherwise
  $e_{i,j}=0$, and the order in $\prod\limits_{i\geq0} E_{i+1,i}[b_{i+1}]$ is
  increasing from the right hand side to left hand side, i.e.,
  $$\prod\limits_{i\geq0} E_{i+1,i}[b_{i+1}]=\cdots E_{3,2}[b_{3}]\cdot E_{2,1}[b_{2}]\cdot E_{1,0}[b_{1}].$$
  Obviously, $E_{i+1,i}[b_{i+1}]$ is
  $\textbf{y}$-totally positive.

In what follows we will proceed with the proof of total positivity
of ${\overrightarrow{M}}_{k}$ by induction on $k$. It is trivial for
$k=0$ that ${\overrightarrow{M}}_{0}$ is
$(\textbf{x},\textbf{y})$-totally positive of order $r$.
  Assume that ${\overrightarrow{M}}_{k-1}$ is $(\textbf{x},\textbf{y})$-totally
  positive of order $r$. Then so is $$\left(
                                   \begin{array}{cc}
                                                  1&0\\
                                                  0&{\overrightarrow{M}}_{k-1}\\
                                   \end{array}
                                 \right).$$
 In addition, if $\mathcal{A}$ is $\textbf{x}$-totally
  positive of order $r$, then, combining (\ref{connection}) and the classical Cauchy-Binet formula gives that $P$ is $(\textbf{x},\textbf{y})$-totally
  positive of order $r$. Thus, applying the classical Cauchy-Binet formula to
(\ref{decomposition}), we conclude that ${\overrightarrow{M}}_{k}$
is $(\textbf{x},\textbf{y})$-totally
  positive of order $r$.

Case 2: Assume $t\geq1$. We will prove the desired result by
induction on the row index $n$. Let $\widetilde{P}$ be a submatrix
of $P$ by deleting columns from $1$th to $t$th and
${W}_n=[W_{i,j}]_{0\leq i,j\leq n}$ denote the $n$th order leading
principal submatrix of the matrix $W$. Then by Proposition
\ref{prop+rec+2+M}, we obtain
  \begin{equation}\label{A}
               {M}_{n}                      = \widetilde{P}_n \left(
                                              \begin{array}{cc}
                                                               1&0\\
                                                               0&{M}_{n-1}\\
                                               \end{array}
                                          \right).
  \end{equation}
  Assume that $\mathcal{A}$ is $\textbf{x}$-totally
  positive of order $r$. Then $P$ is $(\textbf{x},\textbf{y})$-totally
  positive of order $r$. So is $\widetilde{P}$. In consequence, its leading principal submatrix $\widetilde{P}_n$ is $(\textbf{x},\textbf{y})$-totally
  positive of order $r$. By induction on $n$, we immediately derive that ${M}_n$ is $(\textbf{x},\textbf{y})$-totally
  positive of order $r$. So is the matrix $M$.

 This completes the proof of Theorem \ref{thm+main}.\qed

\section{The Lindstr\"om--Gessel--Viennot lemma}\label{section+LGV+lemma}

In this section, we will first introduce the
Lindstr\"om--Gessel--Viennot lemma. Recall the
Lindstr\"om--Gessel--Viennot lemma as follows: Let $\Gamma = (V,
\vec{E})$ be a directed graph, where the vertex set $V$ and the edge
set $\vec{E}$ need not be finite, and we assume that for each
ordered pair $(i,j) \in V \times V$ there is at most one edge from
$i$ to $j$. For an edge $(i,j) \in \vec{E}$, give an indeterminate
$w_{ij}$ as the weight on it. Now let ${\bf w} = (w_{ij})_{(i,j) \in
\vec{E}}$ be the set of edge weights of $\Gamma$.
%
For $i,j \in V$, a \emph{walk} from $i$ to $j$ (of length $n \ge 0$)
is a sequence $\gamma = (\gamma_0,\ldots,\gamma_n)$ in $V$ such that
$\gamma_0 = i$, $\gamma_n = j$, and $(\gamma_{k-1},\gamma_k) \in
\vec{E}$ for $1 \le k \le n$. The \emph{weight} of a walk is the
product of the weights of its edges:
\begin{equation}
   W(\gamma)  \;=\;  \prod_{k=1}^n w_{\gamma_{k-1} \gamma_k}
   \;,
\end{equation}
 where the empty product is defined as usual to be 1. Note that
for each $i,j \in V$, the total weight of walks from $i$ to $j$,
namely
\begin{equation}
   b_{ij}  \;=\;  \sum_{\gamma \colon\, i \to j}  W(\gamma)
   \;,
\end{equation} is a well-defined element of the formal-power-series ring
$Z[[{\bf w}]]$, since each monomial in the indeterminates ${\bf w}$
corresponds to at most finitely many walks.
We denote the \emph{walk matrix} by $B = (b_{ij})_{i,j \in V}$.

For an $r \times r$ minor of the walk matrix, we view the
corresponding rows $i_1,\ldots,i_r \in V$ as the \emph{source
vertices} and the corresponding columns $j_1,\ldots,j_r \in V$ as
the \emph{sink vertices}. Here $i_1,\ldots,i_r$ are all distinct,
and $j_1,\ldots,j_r$ are all distinct. We say that walks
$\gamma_1,\ldots,\gamma_r$ are vertex-disjoint if for $1 \leq k <
\ell \leq r$, the walks $\gamma_k$ and $\gamma_\ell$ have no
vertices in common (not even the endpoints). The
Lindstr\"om--Gessel--Viennot lemma
\cite{Lindstrom_73,Gessel-Viennot_89} expresses such a minor as a
sum over {\em vertex-disjoint}\/ systems of walks from the source
vertices to the sink vertices of the acyclic directed graph
$\Gamma$:

\begin{lem}[Lindstr\"om--Gessel--Viennot lemma]
   \label{lem+gessel-viennot}
Suppose that the directed graph $\Gamma$ is acyclic.  Then
\begin{equation}
   \det B\biggl(\!\! \begin{array}{c}
                        i_1,\ldots,i_r \\
                        j_1,\ldots,j_r
                     \end{array}
           \!\!\biggr)
   \;=\;
   \sum_{\sigma \in \mathfrak{S}_r}  \sgn(\sigma)
   \!\!\!
   \sum_{\begin{scarray}
             \gamma_1 \colon\, i_1 \to j_{\sigma(1)}  \\[-1mm]
                   \vdots \\[-1mm]
             \gamma_r \colon\, i_r \to j_{\sigma(r)}  \\[1mm]
             \gamma_k \cap \gamma_\ell = \emptyset
               \hboxscript{ for } k \ne \ell
         \end{scarray}
        }
   \!\!\!
   \prod_{i=1}^r W(\gamma_i),
 \label{eq.lemma.gessel-viennot}
\end{equation} where the sum runs over all $r$-tuples of vertex-disjoint walks
$\gamma_1,\ldots,\gamma_r$ connecting the specified vertices.
\end{lem}

Let $\textbf{i} = (i_1,\ldots,i_r)$ be an ordered $r$-tuple of
distinct vertices of $\Gamma$, and let $\textbf{j} =
(j_1,\ldots,j_r)$ be another such ordered $r$-tuple; we say that the
pair $(\textbf{i}, \textbf{j})$ is \textbf{nonpermutable} if the set
of vertex-disjoint walk systems $\gamma =
(\gamma_1,\ldots,\gamma_r)$ satisfying $\gamma_k \colon\, i_k \to
j_{\sigma(k)}$ is empty whenever $\sigma$ is not the identity
permutation. In this situation we avoid the sum over permutations in
(\ref{eq.lemma.gessel-viennot}); most importantly, we avoid the
possibility of terms with $\sgn(\sigma) = -1$. The nonpermutable
case arises frequently in applications: for instance, if $\Gamma$ is
planar, then the nonpermutability often follows from topological
arguments.

Now let $I$ and $J$ be subsets (not necessarily finite) of $V$,
equipped with total orders $<_I$ and $<_J$, respectively. We say
that the pair $\bigl((I,<_I),(J,<_J)\bigr)$ is \textbf{fully
nonpermutable} if for each $r \ge 1$ and each pair of increasing
$r$-tuples $\textbf{i} = (i_1,\ldots,i_r)$ in $(I,<_I)$ and
$\textbf{j} = (j_1,\ldots,j_r)$ in $(J,<_J)$, the pair $(\textbf{i},
\textbf{j})$ is nonpermutable. Then the Lindstr\"om--Gessel--Viennot
lemma immediately gives the following result for total positivity:

\begin{cor}[Lindstr\"om--Gessel--Viennot lemma and total positivity]
   \label{cor.gessel-viennot.totpos}
Assume that the directed graph $\Gamma$ is acyclic. Let $(I,<_I)$
and $(J,<_J)$ be totally ordered subsets of $V$ such that the pair
$\bigl((I,<_I),(J,<_J)\bigr)$ is fully nonpermutable. Then the
submatrix $B_{IJ}$ (with rows and columns ordered according to $<_I$
and $<_J$) is ${\bf w}$-totally positive.
\end{cor}

We also refer the reader to \cite[pp.~179--180]{Bre95} for examples
of a nonpermutable pair $(\textbf{i}, \textbf{j})$ that is not fully
nonpermutable.

The following ``topologically obvious'' fact will be very useful for
checking the condition of Corollary \ref{cor.gessel-viennot.totpos}.

\begin{lem}   \label{lemma.nonpermutable}
If $\Gamma$ is embedded in the plane and the vertices of $I \cup J$
lie on the boundary of $\Gamma$ in the order ``first $I$ in reverse
order, then $J$ in order'', then the pair $(I,J)$ is fully
nonpermutable.
\end{lem}

\section{A combinatorial proof of Theorem \ref{thm+tridiag}}\label{section+proof+thm+tirdiag}
In this section, we will use the Lindstr\"om--Gessel--Viennot lemma
to give a combinatorial proof of Theorem \ref{thm+tridiag}. The key
issue is to construct the corresponding weighted digraph for a
matrix. So we will give the corresponding weighted digraphs for the
matrix $M$ and the Toeplitz matrix of the each row sequence of $M$,
respectively.

\subsection {Construct weighted digraphs}

In terms of the Lindstr\"om--Gessel--Viennot lemma, in what follows
we will construct the corresponding weighted digraphs for the matrix
$M$ and the Toeplitz matrix of the each row sequence of $M$,
respectively.

In the following, we will first give the weighted digraph for the
matrix $\widetilde{P}_n$ in (\ref{A}).

\begin{definition}\label{def+gamma+TG1}
Let $\Gamma_n= (V, \vec{E})$, where $ V=[n+3]\times[n+1]$, and
$\vec{E}$ is composed of  the following four kinds of arcs:

\begin{itemize}
    \item [(i)]
     $\textsl{a} = (i,j) \rightarrow (i+1,j)$, the weight
      \begin{eqnarray*}
         w_{\textsl{a}} =\begin{cases}   1        &\text{for} ~ (i,j)\in [1,n]\times[1,n]~\text{or}~i\in [1,n+2]~\text{and}~j=n+1;\\
                         \lambda_{n+t-j}      & \text{for} ~i=n+2~\text{and}~1\leq j\leq n.
                    \end{cases}
     \end{eqnarray*}
     \item [(ii)]
     $\textsl{a} = (i,j) \rightarrow (i+1,j+1)$, the weight
     \begin{eqnarray*}
         w_{\textsl{a}} =\begin{cases}  b_{n+1-i}       &\text{for} ~ 1\leq i=j\leq n;\\
                                       \mu_{n+t-j}       &\text{for}~i=n+2~\text{and}~1\leq j\leq n-1.\\
                    \end{cases}
  \end{eqnarray*}
    \item [(iii)]
    $\textsl{a}= (n+1,j) \rightarrow (n+2,j+t-1)$, the weight $w_{\textsl{a}} = \alpha_{n+1-j}$ for $1\leq j\leq n+1-t$.
    \item [(iv)]
    $\textsl{a} = (n+1,j) \rightarrow (n+2,j+t)$, the weight $w_{\textsl{a}} = \beta_{n+1-j}$ for $1\leq j\leq n-t$.
\end{itemize}
Denote the vertex $(1,j)$ (resp. $(n+3,j)$) by $Q_{j-1}^{(n)}$ (resp.
$Q_{j-1}^{(n-1)}$) for $1\leq j\leq n+1$. Furthermore,
vertices $Q_n^{(n)},Q_{n-1}^{(n)},\ldots,Q_0^{(n)}$~(resp. $Q_n^{(n-1)},
Q_{n-1}^{(n-1)},\ldots,Q_0^{(n-1)}$) are viewed as the source vertices
(resp. the sink vertices).
\end{definition}
In order to  understand such $\Gamma_n= (V, \vec{E})$, we give an
example for $t=1,n=3$ in Figure~6.

\begin{center}
\setlength{\unitlength}{0.95cm}
\begin{picture}(9,5)(5,1.7)

 \thicklines\put(5.5,6){\line(1,0){1.5}}\thicklines\put(10,6){\line(1,0){1.5}}
 \thicklines\put(5.5,5){\line(1,0){1.5}}\thicklines\put(10,5){\line(1,0){1.5}}
\thicklines\put(5.5,4){\line(1,0){1.5}}\thicklines\put(10,4){\line(1,0){1.5}}
 \thicklines\put(5.5,3){\line(1,0){1.5}}\thicklines\put(10,3){\line(1,0){1.5}}

\thicklines\put(7,6){\line(1,0){1.5}}
\thicklines\put(8.5,6){\line(1,0){1.5}}
\thicklines\put(7,5){\line(1,0){1.5}}
\thicklines\put(8.5,5){\line(1,0){1.5}}
\thicklines\put(7,4){\line(1,0){1.5}}
\thicklines\put(8.5,4){\line(1,0){1.5}}
\thicklines\put(7,3){\line(1,0){1.5}}
\thicklines\put(8.5,3){\line(1,0){1.5}}

\thicklines\put(11.5,6){\line(1,0){1.5}}
\thicklines\put(11.5,5){\line(1,0){1.5}}
\thicklines\put(11.5,4){\line(1,0){1.5}}
\thicklines\put(11.5,3){\line(1,0){1.5}}

\thicklines\put(5.5,3){\line(3,2){1.5}}
\thicklines\put(10,4){\line(3,2){1.5}}
\thicklines\put(8.5,5){\line(3,2){1.5}}
\thicklines\put(10,3){\line(3,2){1.5}}
  \thicklines\put(7,4){\line(3,2){1.5}}    \thicklines\put(11.5,3){\line(3,2){1.5}}
\thicklines\put(11.5,4){\line(3,2){1.5}}

 \put(5.5,6){\circle*{0.16}} \put(7,6){\circle*{0.16}}
 \put(5.5,5){\circle*{0.16}} \put(7,5){\circle*{0.16}}
 \put(5.5,4){\circle*{0.16}} \put(7,4){\circle*{0.16}}
 \put(5.5,3){\circle*{0.16}} \put(7,3){\circle*{0.16}}

\put(8.5,6){\circle*{0.16}} \put(10,6){\circle*{0.16}}
\put(8.5,5){\circle*{0.16}} \put(10,5){\circle*{0.16}}
\put(8.5,4){\circle*{0.16}} \put(10,4){\circle*{0.16}}
\put(8.5,3){\circle*{0.16}} \put(10,3){\circle*{0.16}}

 \put(11.5,6){\circle*{0.16}}\put(13,6){\circle*{0.16}}
\put(11.5,5){\circle*{0.16}}\put(13,5){\circle*{0.16}}
\put(11.5,4){\circle*{0.16}}\put(13,4){\circle*{0.16}}
\put(11.5,3){\circle*{0.16}}\put(13,3){\circle*{0.16}} \tiny

\put(5,6.15){$Q_3^{(3)}$}      \put(6.2,6.1){1}
\put(5,5.15){$Q_{2}^{(3)}$}  \put(6.2,5.1){1}
\put(5,4.15){$Q_{1}^{(3)}$}  \put(6.2,4.1){1}
\put(5,3.15){$Q_{0}^{(3)}$}  \put(6.2,3.1){1}

 \put(7.7,6.1){1} \put(9.2,6.1){1}\put(12.2,6.1){1}
 \put(7.7,5.1){1} \put(9.2,5.1){1}\put(12.2,5.1){$\lambda_1$}
 \put(7.7,4.1){1} \put(9.2,4.1){1}\put(12.2,4.1){$\lambda_2$}
 \put(7.7,3.1){1} \put(9.2,3.1){1}\put(12.2,3.1){$\lambda_3$}

\put(10.7,6.1){$1$}             \put(13.15,6.1){$Q_3^{(2)}$}
\put(10.7,5.1){$\alpha_1$}     \put(13.15,5.1){$Q_{2}^{(2)}$}
\put(10.7,4.1){$\alpha_2$}     \put(13.15,4.1){$Q_{1}^{(2)}$}
\put(10.7,3.1){$\alpha_3$}     \put(13.15,3.1){$Q_{0}^{(2)}$}

\put(9.1,5.65){$b_1$} \put(7.5,4.6){$b_2$}
\put(10.45,4.6){$\beta_2$} \put(6,3.6){$b_3$}
\put(10.45,3.6){$\beta_3$}

\put(12.2,6.1){1} \put(11.95,4.6){$\mu_2$} \put(11.95,3.6){$\mu_3$}
\normalsize \put(4.1,1.7){Figure~6: $\Gamma_3 $ for $t=1$. All edges
point towards the right.}
\end{picture}
\end{center}

For $\Gamma_n$ in Definition \ref{def+gamma+TG1}, we denote by
$B(n)$ its walk matrix. The following consequence gives the relation
between the walk matrix $B(n)$ and the matrix $\widetilde{P}_n$.

\begin{prop}\label{prop+T1+P}
The walk matrix $B(n)$ for the weighted digraph $\Gamma_n$ equals
the matrix $\widetilde{P}_n$.
\end{prop}
\begin{proof} We shall directly calculate the formula for $B_{i,j}$,
then check that it equals the corresponding entry of
$\widetilde{P}_n$.

Taking the vertex $Q_{i}^{(n)}$ in $V(\Gamma_n)$ as the row index $i$ and
$Q_{j}^{(n-1)}$ in $V(\Gamma_n)$ as the column index $j$, then $B_{i,j}$ is the
total weight of walks from $Q_{n-i}^{(n)}$ to $Q_{n-j}^{(n-1)}$, i.e.,
$B_{i,j}$ is the total weight of walks from $(1,n-i+1)$ to $(n+3,n-j+1)$. In the following,
we shall calculate $B_{i,j}$ according to the walks from $(1,n-i+1)$ to $(n+3,n-j+1)$.

 For $j=0$, if $i=0$, then there is only one path  from $(1,n+1)$ to
 $(n+3,n+1)$. So it is obvious that $B_{0,0}=1$.
 For $j=0$, if $i\geq 1$, then there is only one path  from $(1,n-i+1)$ to
 $(n+3,n+1)$, which immediately implies $B_{i,0}=\prod\limits_{s=0}^{i-1}b_{i-s}$.

 For $j\geq1$, let $m=i-j-t$, in terms of the range of $m$,
 we divide it into four cases to consider the paths from $(1,n-i+1)$ to $(n+3,n-j+1)$.

  (i) If $m<-1$, i.e., $i<j+t-1$, then there is no path from $(1,n-i+1)$ to $(n+3,n-j+1)$. In consequence, we have
  \begin{eqnarray}\label{eq+F511}
  B_{i,j}=0.
  \end{eqnarray}

  (ii) If $m=-1$, i.e., $i=j+t-1$, then the point $(n+3,n-j+1)$~can be
  denoted by~$(n+3,n+t-i)$, and there is only one path of $n+2$ steps from $(1,n-i+1)$ to $(n+3,n+t-i)$ as follows:
  \begin{eqnarray*}
                  \begin{cases}   (s,n-i+1) \rightarrow (s+1,n-i+1)\quad \text{for} ~1\leq s\leq n, \\
                                 (n+1,n-i+1) \rightarrow (n+2,n+t-i),\\
                                  (n+2,n+t-i)\rightarrow (n+3,n+t-i).
                    \end{cases}
  \end{eqnarray*}
  Then we have
  \begin{eqnarray*}
  B_{i,i+1-t}=\alpha_i\lambda_i.
  \end{eqnarray*}
  By Theorem \ref{thm+tridiag} (i), we obtain $\alpha_i\lambda_i=a_i^{(0)}$,~so
  \begin{eqnarray*}
  B_{i,i+1-t}=a_i^{(0)}.
  \end{eqnarray*}
  Noting $i+1-t=j$, we derive
  \begin{eqnarray}\label{eq+F512}
  B_{i,j}=a_i^{(0)}.
  \end{eqnarray}

  (iii) If~$m=0$,~i.e.,~$i=j+t$, then the point $(n+3,n-j+1)$~can be written as~$(n+3,n+t-i+1)$,
  and there are three paths from $(1,n-i+1)$ to $(n+3,n+t-i+1)$. For such paths, we classify them according to
  their intersecting points with the lines $x=n+1$ and $x=n+2$ as follows:
    the first has the following $n+2$ steps
     \begin{eqnarray*}
                  \begin{cases}   (s,n-i+1) \rightarrow (s+1,n-i+1)\quad \text{for} ~1\leq s\leq n, \\
                                 (n+1,n-i+1) \rightarrow (n+2,n+t-i),\\
                                  (n+2,n+t-i)\rightarrow (n+3,n+t-i+1);
                    \end{cases}
  \end{eqnarray*}
   the second  has the following $n+2$ steps
    \begin{eqnarray*}
                  \begin{cases}   (s,n-i+1) \rightarrow (s+1,n-i+1)\quad \text{for} ~1\leq s\leq n, \\
                                (n+1,n-i+1) \rightarrow (n+2,n+t-i+1),\\
                                 (n+2,n+t-i+1) \rightarrow (n+3,n+t-i+1);
                    \end{cases}
     \end{eqnarray*}
    and the third has the following $n+2$ steps
 \begin{eqnarray*}
                  \begin{cases}   (s,n-i+1) \rightarrow (s+1,n-i+1)\quad \text{for} ~1\leq s\leq n-i, \\
                                 (n-i+1,n-i+1) \rightarrow (n-i+2,n-i+2),\\
                                 (n-i+2+s,n-i+2) \rightarrow (n-i+3+s,n-i+2)\quad \text{for} ~0\leq s\leq i-2,\\
                                  (n+1,n-i+2) \rightarrow (n+2,n+t-i+1),\\
                                  (n+2,n+t-i+1) \rightarrow (n+3,n+t-i+1).
                    \end{cases}
     \end{eqnarray*}
  In consequence, by such three paths above,  we have
  \begin{eqnarray}\label{eq+three}
  B_{i,i-t}=\alpha_i\mu_i+\beta_i\lambda_{i-1}+b_{i}\alpha_{i-1}\lambda_{i-1}.
  \end{eqnarray}
  By Theorem \ref{thm+tridiag} (i) and (ii), we obtain
  \begin{eqnarray}\label{eq+three+condition}
   \begin{cases}
 \alpha_{i-1}\lambda_{i-1}=a_{i-1}^{(0)},\\
 \alpha_i\mu_i+\beta_i\lambda_{i-1}=a_i^{(1)}.
  \end{cases}
  \end{eqnarray}
  So combining (\ref{eq+three}) and (\ref{eq+three+condition}) gives
  $$B_{i,i-t}=a_i^{(1)}+a_{i-1}^{(0)}b_i.$$
  Noting $i-t=j$, we obtain
  \begin{eqnarray}\label{eq+F513}
  B_{i,j}=a_i^{(1)}+a_{i-1}^{(0)}b_i.
  \end{eqnarray}

  (iv) If~$m\geq1$,~i.e.,~$i\geq j+t+1$, then the point $(n+3,n-j+1)$~can be denoted by~$(n+3,n+t+m-i+1)$, and
   there are four paths from $(1,n-i+1)$ to $(n+3,n+m+t-i+1)$. For
   such paths, we classify them according to their intersecting points with the lines $x=n+1$ and $x=n+2$ as follows:
    the first has the following $n+2$ steps
    \begin{eqnarray*}
                  \begin{cases}   (s,n-i+1) \rightarrow (s+1,n-i+1)\quad \text{for} ~1\leq s\leq n-i, \\
                                 (n-i+1+s,n-i+1+s) \rightarrow (n-i+2+s,n-i+2+s)\quad \text{for}~ 0\leq s\leq m-2,\\
                                 (n+m-i+s,n+m-i) \rightarrow (n+m-i+s+1,n+m-i)\quad \text{for}~ 0\leq s\leq i-m,\\
                                 (n+1,n+m-i) \rightarrow (n+2,n+t+m-i),\\
                                 (n+2,n+t+m-i) \rightarrow (n+3,n+t+m-i+1);
                    \end{cases}
  \end{eqnarray*}
   the second has the following $n+2$ steps
   \begin{eqnarray*}
                  \begin{cases}   (s,n-i+1) \rightarrow (s+1,n-i+1)\quad \text{for} ~1\leq s\leq n-i, \\
                                 (n-i+1+s,n-i+1+s) \rightarrow (n-i+2+s,n-i+2+s)\quad \text{for}~ 0\leq s\leq m-1,\\
                                 (n+m-i+1+s,n+m-i+1) \rightarrow (n+m-i+2+s,n+m-i+1)\\
                                 \quad\quad\quad\quad\quad\quad\quad\quad\quad\quad\quad\quad\quad\quad\quad\quad\quad\quad\quad\quad\quad\quad\quad\quad\quad\text{for}~ 0\leq s\leq i-m-1,\\
                                 (n+1,n+m-i+1) \rightarrow (n+2,n+t+m-i),\\
                                 (n+2,n+t+m-i) \rightarrow (n+3,n+t+m-i+1);
                    \end{cases}
  \end{eqnarray*}
   the third has the following $n+2$ steps
    \begin{eqnarray*}
                  \begin{cases}   (s,n-i+1) \rightarrow (s+1,n-i+1)\quad \text{for} ~1\leq s\leq n-i, \\
                                 (n-i+1+s,n-i+1+s) \rightarrow (n-i+2+s,n-i+2+s)\quad \text{for}~ 0\leq s\leq m-1,\\
                                 (n+m-i+1+s,n+m-i+1) \rightarrow (n+m-i+2+s,n+m-i+1)\\
                                 \quad\quad\quad\quad\quad\quad\quad\quad\quad\quad\quad\quad\quad\quad\quad\quad\quad\quad\quad\quad\quad\quad\quad\quad\quad\text{for}~ 0\leq s\leq i-m-1,\\
                                 (n+1,n+m-i+1) \rightarrow (n+2,n+t+m-i+1),\\
                                (n+2,n+t+m-i+1) \rightarrow (n+3,n+t+m-i+1);
                    \end{cases}
  \end{eqnarray*}
   the fourth has the following $n+2$ steps
    \begin{eqnarray*}
                  \begin{cases}   (s,n-i+1) \rightarrow (s+1,n-i+1)\quad \text{for} ~1\leq s\leq n-i, \\
                                 (n-i+1+s,n-i+1+s) \rightarrow (n-i+2+s,n-i+2+s)\quad \text{for}~ 0\leq s\leq m,\\
                                 (n+m-i+2+s,n+m-i+2) \rightarrow (n+m-i+3+s,n+m-i+2)\\
                                 \quad\quad\quad\quad\quad\quad\quad\quad\quad\quad\quad\quad\quad\quad\quad\quad\quad\quad\quad\quad\quad\quad\quad\quad\quad\text{for}~ 0\leq s\leq i-m-2,\\
                                 (n+1,n+m-i+2) \rightarrow (n+2,n+t+m-i+1),\\
                                (n+2,n+t+m-i+1) \rightarrow (n+3,n+t+m-i+1).
                    \end{cases}
  \end{eqnarray*}
   In consequence, by such four paths above,  we have
  \begin{eqnarray}\label{eq+four}
  B_{i,i-t-m}&=&\left(\prod\limits_{s=0}^{m-2}b_{i-s}\right)\beta_{i-m+1}\mu_{i-m}+\left(\prod\limits_{s=0}^{m-1}b_{i-s}\right)\alpha_{i-m}\mu_{i-m}+\nonumber\\
  &&\left(\prod\limits_{s=0}^{m-1}b_{i-s}\right)\beta_{i-m}\lambda_{i-m-1} +\left(\prod\limits_{s=0}^{m}b_{i-s}\right)\alpha_{i-m-1}\lambda_{i-m-1}\nonumber\\
  &=&\left(\prod\limits_{s=0}^{m-2}b_{i-s}\right)\beta_{i-m+1}\mu_{i-m}+\left(\prod\limits_{s=0}^{m-1}b_{i-s}\right)\left (\alpha_{i-m}\mu_{i-m}+\beta_{i-m}\lambda_{i-m-1}\right)+\nonumber\\
  &&\left(\prod\limits_{s=0}^{m}b_{i-s}\right)\alpha_{i-m-1}\lambda_{i-m-1}.
  \end{eqnarray}
  By Theorem \ref{thm+tridiag}~(i), (ii) and (iii), we obtain
  \begin{eqnarray}\label{eq+four+condition}
   \begin{cases}
  \alpha_{i-m-1}\lambda_{i-m-1}=a_{i-m-1}^{(0)},\\
  \alpha_{i-m}\mu_{i-m}+\beta_{i-m}\lambda_{i-m-1}=a_{i-m}^{(1)},\\
  \beta_{i-m+1}\mu_{i-m}=a_{i-m+1}^{(2)}.
  \end{cases}
  \end{eqnarray}
  So combining (\ref{eq+four}) and (\ref{eq+four+condition})
  yields
  $$B_{i,i-t-m}=a_{i-m+1}^{(2)}\prod\limits_{s=0}^{m-2}b_{i-s}+a_{i-m}^{(1)}\prod\limits_{s=0}^{m-1}b_{i-s}+a_{i-m-1}^{(0)}\prod\limits_{s=0}^{m}b_{i-s}.$$
  Noting $i-t-m=j$, we derive
  \begin{eqnarray}\label{eq+F514}
  B_{i,j}=a_{j+t+1}^{(2)}\prod\limits_{s=0}^{i-j-t-2}b_{i-s}+a_{j+t}^{(1)}\prod\limits_{s=0}^{i-j-t-1}b_{i-s}+a_{j+t-1}^{(0)}\prod\limits_{s=0}^{i-j-t}b_{i-s}.
  \end{eqnarray}

On the other hand, in terms of the definition of $\widetilde{P}$ and
the formula (\ref{formula+P}), we immediately have
 \begin{eqnarray}\label{eq+combine}
         \widetilde{P}_{i,j} =\begin{cases}   1                                   & \text{for} ~ j=0,i=0;\\
                                 \prod\limits_{s=0}^{i-1}b_{i-s}      & \text{for} ~j=0,i\geq1;\\
                                 0                                    & \text{for} ~ j\geq1,i<j+t-1;\\
                                 \sum\limits_{m\geq0}a_{i-m}^{(i-j-m-t+1)}\prod\limits_{s=0}^{m-1}b_{i-s} & \text{for} ~ j\geq1,j+t-1\leq i< j+t+1;\\
                                 \sum\limits_{m=0}^{2}a_{j+t-m+1}^{(2-m)}\prod\limits_{s=0}^{i+m-j-t-2}b_{i-s} & \text{for} ~j\geq1, i\geq j+t+1.

                    \end{cases}
  \end{eqnarray}
  By simplifying (\ref{eq+combine}), we obtain
  \begin{eqnarray}\label{eq+combine1}
         \widetilde{P}_{i,j} =\begin{cases}   1                                   & \text{for} ~ j=0,i=0;\\
                                 \prod\limits_{s=0}^{i-1}b_{i-s}      & \text{for} ~j=0,i\geq1;\\
                                 0                                    & \text{for} ~ j\geq1,i<j+t-1;\\
                                 a_i^{(0)}                                      & \text{for} ~ j\geq1,i=j+t-1;\\
                                 a_i^{(1)}+a_{i-1}^{(0)}b_i                                     & \text{for} ~ j\geq1,i=j+t;\\
                                 a_{j+t+1}^{(2)}\prod\limits_{s=0}^{i-j-t-2}b_{i-s}+a_{j+t}^{(1)}\prod\limits_{s=0}^{i-j-t-1}b_{i-s}+a_{j+t-1}^{(0)}\prod\limits_{s=0}^{i-j-t}b_{i-s}
                                  & \text{for} ~j\geq1, i\geq j+t+1.
                    \end{cases}
  \end{eqnarray}
Then by (\ref{eq+F511}), (\ref{eq+F512}), (\ref{eq+F513}),
(\ref{eq+F514}) and (\ref{eq+combine1}), we immediately derive
$$ B_n=\widetilde{P}_n.$$
The  proof of Proposition \ref{prop+T1+P} is complete.
\end{proof}
In what follows we will use the weighted digraph $\Gamma_n$ to
construct a weighted digraph for $M$.
\begin{definition}\label{definition+star}
Define recursively a weighted digraph, denoted by
$\Gamma^{\ast}_n=(V^{\ast}_n, \vec{E}^{\ast}_n)$, as follows:
\begin{itemize}
  \item [(i)]
  For $n=1$, we take the weighted digraph $\Gamma_1$ as $\Gamma^{\ast}_1$
  (see Figure~7 for the case $t=1$).
  \begin{center}
  \setlength{\unitlength}{0.85cm}
  \begin{picture}(3,3.2)(2.5,-1.3)
  \thicklines\put(1,1){\line(1,0){1.5}}  \thicklines\put(2.5,1){\line(1,0){1.5}} \thicklines\put(4,1){\line(1,0){1.5}}
  \thicklines\put(1,0){\line(1,0){1.5}}  \thicklines\put(2.5,0){\line(1,0){1.5}}  \thicklines\put(4,0){\line(1,0){1.5}}

  \thicklines\put(1,0){\line(3,2){1.5}}

  \put(1,1){\circle*{0.18}} \put(2.5,1){\circle*{0.18}} \put(4,1){\circle*{0.18}} \put(5.5,1){\circle*{0.18}}
  \put(1,0){\circle*{0.18}} \put(2.5,0){\circle*{0.18}} \put(4,0){\circle*{0.18}} \put(5.5,0){\circle*{0.18}}

  \tiny

  \put(0.47,1.167){$Q_1^{(1)}$} \put(5.64,1.11){$Q_1^{(0)}$}
  \put(0.47,0.167){$Q_0^{(1)}$} \put(5.64,0.11){$Q_0^{(0)}$}

                        \put(1.7,1.1){$1$} \put(3.2,1.1){$1$}        \put(4.7,1.1){$1$}
  \put(1.45,0.6){$b_1$}  \put(1.7,0.1){$1$} \put(3.2,0.1){$\alpha_1$} \put(4.7,0.1){$\lambda_1$}
  \normalsize
  \put(-2.5,-1.2){Figure~7: $\Gamma^{\ast}_1$  for $t=1$. All edges point towards the right.}
  \end{picture}
  \end{center}
  \item [(ii)]
  Suppose that $\Gamma^{\ast}_{n-1}=(V^{\ast}_{n-1}, \vec{E}^{\ast}_{n-1})$ has been constructed for some
  $n\geq2$. Then we shall construct the weighted digraph $\Gamma^{\ast}_n$ as follows:
  Let $\overline{\Gamma}^{\ast}_{n-1}$ is a new weighted digraph obtained from $\Gamma^{\ast}_{n-1}$ by adding vertices
  $Q_n^{(n-1)},\ldots,Q_{n}^{(0)}$ to $V^{\ast}_{n-1}$ and
adding arcs $Q_{n}^{(i+1)} \rightarrow
   Q_n^{(i)}$ with the weight $1$ to $\vec{E}^{\ast}_{n-1}$ $(0\leq i\leq n-2)$.
  Then $\Gamma^{\ast}_n$ is the union of $\Gamma_n$ and $\overline{\Gamma}^{\ast}_{n-1}$  .
\end{itemize}
\end{definition}
By Definition \ref{definition+star}, it is easy to see that
vertices $Q_n^{(n)},Q_{n-1}^{(n)},\ldots,Q_0^{(n)}$~
(resp. $Q_n^{(0)},Q_{n-1}^{(0)},$
$\ldots,Q_0^{(0)}$) are viewed as
the source vertices (resp. the sink vertices) of the weighted digraph
$\Gamma^{\ast}_n$.
In order to understand well how to construct the weighted digraph
$\Gamma^{\ast}_n$, for $t=1$, we draw the weighted digraphs
$\Gamma^{\ast}_2$, $\overline{\Gamma}^{\ast}_2$ and
$\Gamma^{\ast}_3$ (see Figures 8-10), respectively.

\begin{center}
\setlength{\unitlength}{0.85cm}
\begin{picture}(7.4,4.4)(2.5,-1.7)
\thicklines\put(1,2){\line(1,0){1.5}}
\thicklines\put(2.5,2){\line(1,0){1.5}}
\thicklines\put(4,2){\line(1,0){1.5}}
\thicklines\put(1,1){\line(1,0){1.5}}
\thicklines\put(2.5,1){\line(1,0){1.5}}
\thicklines\put(4,1){\line(1,0){1.5}}
\thicklines\put(1,0){\line(1,0){1.5}}
\thicklines\put(2.5,0){\line(1,0){1.5}}
\thicklines\put(4,0){\line(1,0){1.5}}

\thicklines\put(5.5,2){\line(1,0){1.5}}
\thicklines\put(7,2){\line(1,0){1.5}}
\thicklines\put(8.5,2){\line(1,0){1.5}}
\thicklines\put(5.5,1){\line(1,0){1.5}}
\thicklines\put(7,1){\line(1,0){1.5}}
\thicklines\put(8.5,1){\line(1,0){1.5}}
\thicklines\put(5.5,0){\line(1,0){1.5}}
\thicklines\put(7,0){\line(1,0){1.5}}
\thicklines\put(8.5,0){\line(1,0){1.5}}

\thicklines\put(10,2){\line(1,0){1.5}}
\thicklines\put(10,1){\line(1,0){1.5}}
\thicklines\put(10,0){\line(1,0){1.5}}

\thicklines\put(1,0){\line(3,2){1.5}}\thicklines\put(2.5,1){\line(3,2){1.5}}\thicklines\put(4,0){\line(3,2){1.5}}
\thicklines\put(5.5,0){\line(3,2){1.5}}\thicklines\put(7,0){\line(3,2){1.5}}

\put(1,2){\circle*{0.18}} \put(2.5,2){\circle*{0.18}}
\put(4,2){\circle*{0.18}} \put(5.5,2){\circle*{0.18}}
\put(7,2){\circle*{0.18}} \put(1,1){\circle*{0.18}}
\put(2.5,1){\circle*{0.18}} \put(4,1){\circle*{0.18}}
\put(5.5,1){\circle*{0.18}} \put(7,1){\circle*{0.18}}
\put(1,0){\circle*{0.18}} \put(2.5,0){\circle*{0.18}}
\put(4,0){\circle*{0.18}} \put(5.5,0){\circle*{0.18}}
\put(7,0){\circle*{0.18}}
                                                     \put(11.5,2){\circle*{0.18}}
\put(8.5,1){\circle*{0.18}} \put(10,1){\circle*{0.18}}
\put(11.5,1){\circle*{0.18}} \put(8.5,0){\circle*{0.18}}
\put(10,0){\circle*{0.18}} \put(11.5,0){\circle*{0.18}}

\tiny

\put(0.42,2.167){$Q_2^{(2)}$}    \put(7.43,0.6){$b_1$} \put(1.7,2.1){1}
\put(3.2,2.1){1} \put(4.7,2.1){1}
\put(0.42,1.167){$Q_1^{(2)}$}    \put(2.93,1.6){$b_1$} \put(1.7,1.1){1}
\put(3.2,1.1){1}
\put(4.6,1.1){$\alpha_1$}\put(4.4,0.6){$\beta_2$}
\put(0.42,0.167){$Q_0^{(2)}$}  \put(1.43,0.6){$b_2$} \put(1.7,0.1){1}
\put(3.2,0.1){1}  \put(4.6,0.1){$\alpha_2$}

\put(6.1,2.1){$1$}              ¡¢
\put(6.05,1.1){$\lambda_1$}\put(7.7,1.1){1} \put(9.15,2.1){1}
\put(10.65,1.1){$1$}                  \put(11.65,2.11){$Q_2^{(0)}$}
\put(5.9,0.6){$\mu_2$}                 \put(9.15,1.1){1}
\put(11.65,1.11){$Q_1^{(0)}$} \put(6.05,0.1){$\lambda_2$}\put(7.7,0.1){1}
\put(9.15,0.1){$\alpha_1$}\put(10.65,0.1){$\lambda_1$}
\put(11.65,0.11){$Q_0^{(0)}$}
\put(6.65,2.25){$Q_2^{(1)}$} \put(6.65,1.25){$Q_1^{(1)}$}
\put(6.65,0.25){$Q_0^{(1)}$}

\normalsize \put(0.3,-1.3){Figure~8: $\Gamma^{\ast}_2$ for $t=1$. All
edges point towards the right.}
\end{picture}
\end{center}

\begin{center}
\setlength{\unitlength}{0.85cm}
\begin{picture}(7.4,4.8)(2.5,-1.28)

\thicklines\put(1,3){\line(1,0){6}}
\thicklines\put(1,2){\line(1,0){1.5}}
\thicklines\put(2.5,2){\line(1,0){1.5}}
\thicklines\put(4,2){\line(1,0){1.5}}
\thicklines\put(1,1){\line(1,0){1.5}}
\thicklines\put(2.5,1){\line(1,0){1.5}}
\thicklines\put(4,1){\line(1,0){1.5}}
\thicklines\put(1,0){\line(1,0){1.5}}
\thicklines\put(2.5,0){\line(1,0){1.5}}
\thicklines\put(4,0){\line(1,0){1.5}}

                                         \thicklines\put(7,3){\line(1,0){4.5}}
\thicklines\put(5.5,2){\line(1,0){1.5}}
\thicklines\put(7,2){\line(1,0){1.5}}
\thicklines\put(8.5,2){\line(1,0){1.5}}
\thicklines\put(5.5,1){\line(1,0){1.5}}
\thicklines\put(7,1){\line(1,0){1.5}}
\thicklines\put(8.5,1){\line(1,0){1.5}}
\thicklines\put(5.5,0){\line(1,0){1.5}}
\thicklines\put(7,0){\line(1,0){1.5}}
\thicklines\put(8.5,0){\line(1,0){1.5}}

\thicklines\put(10,2){\line(1,0){1.5}}
\thicklines\put(10,1){\line(1,0){1.5}}
\thicklines\put(10,0){\line(1,0){1.5}}

\thicklines\put(1,0){\line(3,2){1.5}}\thicklines\put(2.5,1){\line(3,2){1.5}}\thicklines\put(4,0){\line(3,2){1.5}}
\thicklines\put(5.5,0){\line(3,2){1.5}}\thicklines\put(7,0){\line(3,2){1.5}}

\put(1,3){\circle*{0.18}}
\put(7,3){\circle*{0.18}} \put(1,2){\circle*{0.18}}
\put(2.5,2){\circle*{0.18}} \put(4,2){\circle*{0.18}}
\put(5.5,2){\circle*{0.18}} \put(7,2){\circle*{0.18}}
\put(1,1){\circle*{0.18}} \put(2.5,1){\circle*{0.18}}
\put(4,1){\circle*{0.18}} \put(5.5,1){\circle*{0.18}}
\put(7,1){\circle*{0.18}} \put(1,0){\circle*{0.18}}
\put(2.5,0){\circle*{0.18}} \put(4,0){\circle*{0.18}}
\put(5.5,0){\circle*{0.18}} \put(7,0){\circle*{0.18}}
                           \put(11.5,3){\circle*{0.18}}                           \put(11.5,2){\circle*{0.18}}
\put(8.5,1){\circle*{0.18}} \put(10,1){\circle*{0.18}}
\put(11.5,1){\circle*{0.18}} \put(8.5,0){\circle*{0.18}}
\put(10,0){\circle*{0.18}} \put(11.5,0){\circle*{0.18}}

\tiny \put(0.42,3.167){$Q_3^{(2)}$}
\put(4,3.1){1} \put(0.42,2.167){$Q_2^{(2)}$}    \put(7.43,0.6){$b_1$}
\put(1.7,2.1){1} \put(3.2,2.1){1} \put(4.7,2.1){1}
\put(0.42,1.167){$Q_1^{(2)}$}    \put(2.93,1.6){$b_1$} \put(1.7,1.1){1}
\put(3.2,1.1){1}
\put(4.6,1.1){$\alpha_1$}\put(4.4,0.6){$\beta_2$}
\put(0.42,0.167){$Q_0^{(2)}$}  \put(1.43,0.6){$b_2$} \put(1.7,0.1){1}
\put(3.2,0.1){1}  \put(4.6,0.1){$\alpha_2$}

\put(6.1,2.1){$1$}   \put(9.15,3.1){$1$}              ¡¢
\put(6.05,1.1){$\lambda_1$}\put(7.7,1.1){1} \put(9.15,2.1){1}
\put(10.65,1.1){$1$}                  \put(11.65,2.11){$Q_2^{(0)}$}
\put(5.9,0.6){$\mu_2$}                 \put(9.15,1.1){1}
\put(11.65,1.11){$Q_1^{(0)}$} \put(6.05,0.1){$\lambda_2$}\put(7.7,0.1){1}
\put(9.15,0.1){$\alpha_1$}\put(10.65,0.1){$\lambda_1$}
\put(11.65,0.11){$Q_0^{(0)}$} \put(11.65,3.11){$Q_3^{(0)}$}
\put(6.65,2.25){$Q_2^{(1)}$} \put(6.65,1.25){$Q_1^{(1)}$}
\put(6.65,0.25){$Q_0^{(1)}$}\put(6.65,3.25){$Q_3^{(1)}$}

\normalsize \put(0.3,-1.3){Figure~9: $\overline{\Gamma}^{\ast}_2$ for
$t=1$. All edges point towards the right.}
\end{picture}
\end{center}

\begin{center}
\setlength{\unitlength}{0.83cm}
\begin{picture}(4.1,5.4)(0.5,-1.4)
\thicklines\put(-6.5,3){\line(1,0){1.5}}\thicklines\put(-5,3){\line(1,0){1.5}}
\thicklines\put(-6.5,2){\line(1,0){1.5}}\thicklines\put(-5,2){\line(1,0){1.5}}
\thicklines\put(-6.5,1){\line(1,0){1.5}}\thicklines\put(-5,1){\line(1,0){1.5}}
\thicklines\put(-6.5,0){\line(1,0){1.5}}\thicklines\put(-5,0){\line(1,0){1.5}}

\thicklines\put(-3.5,3){\line(1,0){1.5}}\thicklines\put(-2,3){\line(1,0){1.5}}\thicklines\put(-0.5,3){\line(1,0){1.5}}
\thicklines\put(-3.5,2){\line(1,0){1.5}}\thicklines\put(-2,2){\line(1,0){1.5}}\thicklines\put(-0.5,2){\line(1,0){1.5}}
\thicklines\put(-3.5,1){\line(1,0){1.5}}\thicklines\put(-2,1){\line(1,0){1.5}}\thicklines\put(-0.5,1){\line(1,0){1.5}}
\thicklines\put(-3.5,0){\line(1,0){1.5}}\thicklines\put(-2,0){\line(1,0){1.5}}\thicklines\put(-0.5,0){\line(1,0){1.5}}

\thicklines\put(1,3){\line(1,0){6}}
\thicklines\put(1,2){\line(1,0){1.5}}
\thicklines\put(2.5,2){\line(1,0){1.5}}
\thicklines\put(4,2){\line(1,0){1.5}}
\thicklines\put(1,1){\line(1,0){1.5}}
\thicklines\put(2.5,1){\line(1,0){1.5}}
\thicklines\put(4,1){\line(1,0){1.5}}
\thicklines\put(1,0){\line(1,0){1.5}}
\thicklines\put(2.5,0){\line(1,0){1.5}}
\thicklines\put(4,0){\line(1,0){1.5}}

                                         \thicklines\put(7,3){\line(1,0){4.5}}
\thicklines\put(5.5,2){\line(1,0){1.5}}
\thicklines\put(7,2){\line(1,0){1.5}}
\thicklines\put(8.5,2){\line(1,0){1.5}}
\thicklines\put(5.5,1){\line(1,0){1.5}}
\thicklines\put(7,1){\line(1,0){1.5}}
\thicklines\put(8.5,1){\line(1,0){1.5}}
\thicklines\put(5.5,0){\line(1,0){1.5}}
\thicklines\put(7,0){\line(1,0){1.5}}
\thicklines\put(8.5,0){\line(1,0){1.5}}

\thicklines\put(10,2){\line(1,0){1.5}}
\thicklines\put(10,1){\line(1,0){1.5}}
\thicklines\put(10,0){\line(1,0){1.5}}

\thicklines\put(-6.5,0){\line(3,2){1.5}}\thicklines\put(-5,1){\line(3,2){1.5}}\thicklines\put(-3.5,2){\line(3,2){1.5}}
\thicklines\put(-2,0){\line(3,2){1.5}}\thicklines\put(-2,1){\line(3,2){1.5}}\thicklines\put(-0.5,1){\line(3,2){1.5}}
\thicklines\put(1,0){\line(3,2){1.5}}\thicklines\put(2.5,1){\line(3,2){1.5}}\thicklines\put(4,0){\line(3,2){1.5}}
\thicklines\put(5.5,0){\line(3,2){1.5}}\thicklines\put(7,0){\line(3,2){1.5}}\thicklines\put(-0.5,0){\line(3,2){1.5}}

\put(1,3){\circle*{0.18}}
\put(7,3){\circle*{0.18}} \put(1,2){\circle*{0.18}}
\put(2.5,2){\circle*{0.18}} \put(4,2){\circle*{0.18}}
\put(5.5,2){\circle*{0.18}} \put(7,2){\circle*{0.18}}
\put(1,1){\circle*{0.18}} \put(2.5,1){\circle*{0.18}}
\put(4,1){\circle*{0.18}} \put(5.5,1){\circle*{0.18}}
\put(7,1){\circle*{0.18}} \put(1,0){\circle*{0.18}}
\put(2.5,0){\circle*{0.18}} \put(4,0){\circle*{0.18}}
\put(5.5,0){\circle*{0.18}} \put(7,0){\circle*{0.18}}
                           \put(11.5,3){\circle*{0.18}}                           \put(11.5,2){\circle*{0.18}}
\put(8.5,1){\circle*{0.18}} \put(10,1){\circle*{0.18}}
\put(11.5,1){\circle*{0.18}} \put(8.5,0){\circle*{0.18}}
\put(10,0){\circle*{0.18}} \put(11.5,0){\circle*{0.18}}

\put(-0.5,3){\circle*{0.18}}\put(-2,3){\circle*{0.18}}\put(-3.5,3){\circle*{0.18}}\put(-5,3){\circle*{0.18}}\put(-6.5,3){\circle*{0.18}}
\put(-0.5,2){\circle*{0.18}}\put(-2,2){\circle*{0.18}}\put(-3.5,2){\circle*{0.18}}\put(-5,2){\circle*{0.18}}\put(-6.5,2){\circle*{0.18}}
\put(-0.5,1){\circle*{0.18}}\put(-2,1){\circle*{0.18}}\put(-3.5,1){\circle*{0.18}}\put(-5,1){\circle*{0.18}}\put(-6.5,1){\circle*{0.18}}
\put(-0.5,0){\circle*{0.18}} \put(-2,0){\circle*{0.18}}
\put(-3.5,0){\circle*{0.18}} \put(-5,0){\circle*{0.18}}
\put(-6.5,0){\circle*{0.18}} \tiny \put(0.6,3.32){$Q_3^{(2)}$}
\put(4,3.1){1} \put(0.6,2.32){$Q_2^{(2)}$}    \put(7.4,0.6){$b_1$}
\put(1.7,2.1){1} \put(3.2,2.1){1} \put(4.7,2.1){1}
\put(0.6,1.32){$Q_1^{(2)}$}    \put(2.9,1.6){$b_1$} \put(1.7,1.1){1}
\put(3.2,1.1){1}
\put(4.6,1.1){$\alpha_1$}\put(4.3,0.6){$\beta_2$}
\put(0.6,0.32){$Q_0^{(2)}$}  \put(1.4,0.6){$b_2$} \put(1.7,0.1){1}
\put(3.2,0.1){1}  \put(4.6,0.1){$\alpha_2$}

\put(6.1,2.1){$1$}                 \put(9.2,3.1){1}
\put(11.685,3.12){$Q_3^{(0)}$} \put(6.1,1.1){$\lambda_1$}\put(7.7,1.1){1}
\put(9.2,2.1){1}         \put(10.7,1.1){$1$}
\put(11.685,2.12){$Q_2^{(0)}$} \put(5.8,0.6){$\mu_2$}
\put(9.2,1.1){1}
\put(11.685,1.12){$Q_1^{(0)}$} \put(6.1,0.1){$\lambda_2$}\put(7.7,0.1){1}
\put(9.2,0.1){$\alpha_1$}\put(10.7,0.1){$\lambda_1$}
\put(11.685,0.12){$Q_0^{(0)}$}
\put(6.6,3.32){$Q_3^{(1)}$}\put(6.6,2.32){$Q_2^{(1)}$}
\put(6.6,1.32){$Q_1^{(1)}$} \put(6.6,0.32){$Q_0^{(1)}$}

                                                                                         \put(-1.3,3.1){$1$}
 \put(-7.1,3.185){$Q_3^{(3)}$}   \put(-5.8,3.1){$1$} \put(-4.3,3.1){$1$}\put(-2.8,3.1){$1$}
 \put(-7.1,2.185){$Q_2^{(3)}$}   \put(-5.8,2.1){$1$} \put(-4.3,2.1){$1$}\put(-2.8,2.1){$1$}\put(-3.1,2.6){$b_1$}
 \put(-7.1,1.185){$Q_1^{(3)}$}   \put(-5.8,1.1){$1$} \put(-4.3,1.1){$1$}\put(-2.8,1.1){$1$}\put(-4.6,1.6){$b_2$}
 \put(-7.1,0.185){$Q_0^{(3)}$}   \put(-5.8,0.1){$1$} \put(-4.3,0.1){$1$}\put(-2.8,0.1){$1$}\put(-6.1,0.6){$b_3$}

                                                      \put(0.2,3.1){$1$}
\put(-1.4,2.1){$\alpha_1$}
                                                   \put(0.1,2.1){$\lambda_1$}
\put(-1.4,1.1){$\alpha_2$}\put(-1.7,1.6){$\beta_2$}\put(0.1,1.1){$\lambda_2$}\put(-0.2,1.6){$\mu_2$}
\put(-1.4,0.1){$\alpha_3$}\put(-1.7,0.6){$\beta_3$}\put(0.1,0.1){$\lambda_3$}\put(-0.2,0.6){$\mu_3$}

\normalsize \put(-3.5,-1.3){Figure~10: $\Gamma^{\ast}_3$ for $t=1$. All
edges point towards the right.}
\end{picture}
\end{center}

Let $\overline{B}^{\ast}(n-1)$ be the walk matrix of the weighted
digraph $\overline{\Gamma}^{\ast}_{n-1}$. Then it is easy to
observe
$$\overline{B}^{\ast}(n-1)=\left(\begin{array}{cc}
                                                               1&0\\
                                                               0&{M}_{n-1}\\
                                 \end{array}
                            \right).$$
Applying the transfer-matrix method (see \cite[Theorem
4.7.1]{Stanley12} for instance) to  equality (\ref{A}), we obtain
the following result.
\begin{prop}\label{prop+T1+M}
The walk matrix $B^{\ast}(n)$ for the weighted digraph
$\Gamma^{\ast}_n$ equals the matrix $M_n$.
\end{prop}

Finally, using the weighted digraph $\Gamma^{*}_n$, we shall
construct a weighted digraph for the Toeplitz matrix of the each row
sequence of $M$.

\begin{definition}\label{def+gamma+TG+TM+trid}
Define a weighted digraph, denoted by
$\Gamma^{\diamond}_k=(V^{\diamond}_k, \vec{E}^{\diamond}_k)$, as
follows: Let $\Gamma^{\diamond}_{k}$ denote the subgraph of
$\Gamma^{\ast}_{n+k}$ induced by the union of arcs of all directed
paths from $Q_{i}^{(n+i)}$ to $Q_{n+j}^{(0)} ~(0\leq i,j\leq k)$,
where vertices $Q_{0}^{(n)},Q_{1}^{(n+1)},\ldots,Q_{k}^{(n+k)}$
(resp. $Q_{n}^{(0)}, Q_{n+1}^{(0)},\ldots,Q_{n+k}^{(0)}$) are viewed
as the source vertices (resp. the sink vertices) of
$\Gamma^{\diamond}_k$.
\end{definition}

For $t=1,n=1,k=2$, we draw $\Gamma^{\diamond}_2$ in Figure $11$.

\begin{center}
\setlength{\unitlength}{0.83cm}
\begin{picture}(3.57,5)(0.8,-1.3)

\thicklines\put(-6.5,2){\line(1,0){1.5}}\thicklines\put(-5,2){\line(1,0){1.5}}

                                        \thicklines\put(-2,3){\line(1,0){1.5}}
\thicklines\put(-3.5,2){\line(1,0){1.5}}\thicklines\put(-2,2){\line(1,0){1.5}}

\thicklines\put(1,3){\line(1,0){6}}    \thicklines\put(-0.5,3){\line(1,0){1.5}}
\thicklines\put(1,2){\line(1,0){1.5}}  \thicklines\put(-0.5,2){\line(1,0){1.5}}
\thicklines\put(2.5,2){\line(1,0){1.5}} \thicklines\put(7,3){\line(1,0){4.5}}
\thicklines\put(4,2){\line(1,0){1.5}}\thicklines\put(1,1){\line(1,0){1.5}}
\thicklines\put(2.5,1){\line(1,0){1.5}}\thicklines\put(4,1){\line(1,0){1.5}}

\thicklines\put(5.5,2){\line(1,0){1.5}}
\thicklines\put(7,2){\line(1,0){1.5}}\thicklines\put(8.5,2){\line(1,0){1.5}}
\thicklines\put(5.5,1){\line(1,0){1.5}}\thicklines\put(7,1){\line(1,0){1.5}}
\thicklines\put(8.5,1){\line(1,0){1.5}}\thicklines\put(10,2){\line(1,0){1.5}}
\thicklines\put(10,1){\line(1,0){1.5}}

\thicklines\put(2.5,1){\line(3,2){1.5}}\thicklines\put(-3.5,2){\line(3,2){1.5}}
\thicklines\put(7,0){\line(3,2){1.5}}

\put(1,3){\circle*{0.18}}
\put(7,3){\circle*{0.18}} \put(1,2){\circle*{0.18}}
\put(2.5,2){\circle*{0.18}} \put(4,2){\circle*{0.18}}
\put(5.5,2){\circle*{0.18}} \put(7,2){\circle*{0.18}}
\put(1,1){\circle*{0.18}} \put(2.5,1){\circle*{0.18}}
\put(4,1){\circle*{0.18}} \put(5.5,1){\circle*{0.18}}
\put(7,1){\circle*{0.18}} \put(7,0){\circle*{0.18}}\put(11.5,3){\circle*{0.18}}
\put(8.5,1){\circle*{0.18}} \put(10,1){\circle*{0.18}}
\put(11.5,1){\circle*{0.18}}  \put(11.5,2){\circle*{0.18}}

\put(-0.5,3){\circle*{0.18}}\put(-2,3){\circle*{0.18}}
\put(-0.5,2){\circle*{0.18}}\put(-2,2){\circle*{0.18}}
\put(-3.5,2){\circle*{0.18}}\put(-5,2){\circle*{0.18}}\put(-6.5,2){\circle*{0.18}}

 \tiny \put(0.6,3.32){$Q_3^{(2)}$}
\put(4,3.1){1} \put(0.6,2.32){$Q_2^{(2)}$}    \put(7.4,0.6){$b_1$}
\put(1.7,2.1){1} \put(3.2,2.1){1} \put(4.7,2.1){1}
\put(0.6,1.32){$Q_1^{(2)}$}    \put(2.9,1.6){$b_1$} \put(1.7,1.1){1}
\put(3.2,1.1){1}\put(4.6,1.1){$\alpha_1$}

\put(6.1,2.1){$1$} \put(9.2,3.1){1}
\put(11.685,3.12){$Q_3^{(0)}$} \put(6.1,1.1){$\lambda_1$}\put(7.7,1.1){1}
\put(9.2,2.1){1}         \put(10.7,1.1){$1$}
\put(11.685,2.12){$Q_2^{(0)}$} \put(9.2,1.1){1}
\put(11.685,1.12){$Q_1^{(0)}$} \put(6.6,3.32){$Q_3^{(1)}$}\put(6.6,2.32){$Q_2^{(1)}$}
\put(6.6,1.32){$Q_1^{(1)}$} \put(6.6,0.32){$Q_0^{(1)}$}

\put(-1.3,3.1){$1$} \put(-7.1,2.185){$Q_2^{(3)}$}  \put(-5.8,2.1){$1$}
\put(-4.3,2.1){$1$}\put(-2.8,2.1){$1$}\put(-3.1,2.6){$b_1$}  \put(0.2,3.1){$1$}
\put(-1.4,2.1){$\alpha_1$}\put(0.1,2.1){$\lambda_1$}

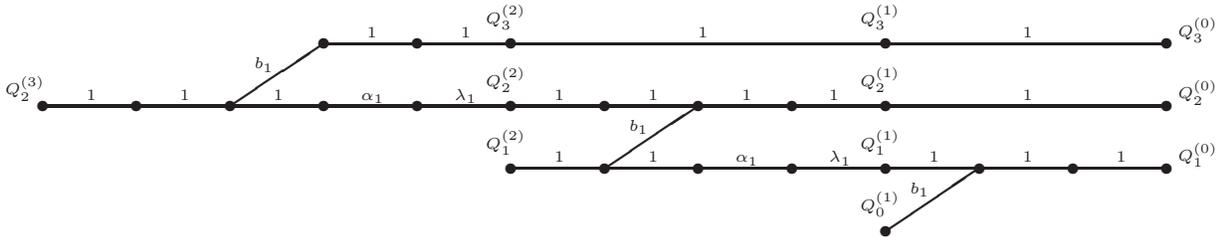
\normalsize \put(-5,-1.3){Figure~11: $\Gamma^{\diamond}_2$ for
$t=1,n=1,k=2$. All edges point towards the right.}
\end{picture}
\end{center}
We denote by $B^\diamond(k)$ the walk matrix of $\Gamma^{\diamond}_k$. The
following consequence indicates the relation between $B^\diamond(k)$
and the Toeplitz matrix of the $n$th row sequence of $M$.

\begin{prop}\label{prop+T+TG+trid}
For the Toeplitz matrix of the $n$th row sequence of $M$, its $k$th
order leading principal minor equals the walk matrix
$B^{\diamond}(k)$ for the weighted digraph $\Gamma^{\diamond}_k$.
\end{prop}
\textbf{Proof:} Note that the $(i,j)$ element of the Toeplitz matrix
of the $n$th row sequence of $M$ equals $M_{n,i-j}$. So, taking the
vertex $Q_{i}^{(n+i)}$ in ${V}^{\diamond}_{k}$ as the row index $i$
and $Q_{n+j}^{(0)}$ as the column index $j$ of the walk matrix
$B^{\diamond}(k)$, we just need to prove
$$ M_{n,i-j}=B_{i,j}^{\diamond}.$$

By Proposition \ref{prop+T1+M}, for the walk matrix $B^{\ast}(k)$,
taking the vertex $Q_{i}^{(n)}$ in ${V}^{\ast}_{n}$ as the row index
$i$ and $Q_{n-i+j}^{(0)}$ in ${V}^{\ast}_{n}$ as the column index
$j$, we have
\begin{equation}\label{t++}
  M_{n-i,i-j}=B_{n-i,i-j}^{\ast}.
\end{equation}
Replacing $n$ with $n+i$ in (\ref{t++}) yields
$$M_{n,i-j}=B_{n,i-j}^{\ast},$$
which is the total weight of walks from $Q_{i}^{(n+i)}$ in
${V}^{\ast}_{n+i}$ to $Q_{n+j}^{(0)}$ in ${V}^{\ast}_{n+i}$. Observe
that both $\Gamma^{\diamond}_{k}$ and $\Gamma^{\ast}_{n+i}$ are
subgraphs of $\Gamma^{\ast}_{n+k}$ and have the same set of walks
from $Q_{i}^{(n+i)}$ to $Q_{n+j}^{(0)}$. Therefore, we obtain
$$B_{i,j}^{\diamond}=B_{n,i-j}^{\ast}.$$
This completes the  proof of Proposition \ref{prop+T+TG+trid}.\qed\\

\subsection {Proof of  Theorem \ref{thm+tridiag}. }

  ~~~~~~~(i) Fix $n\geq0$. For any $m\times m$ minor $M_{I,J}$ of the matrix $M_{n}$, where $I=(i_1,\ldots,i_m)$ with $0\leq i_1< \cdots < i_m\leq n$ and $J=(j_1,\ldots,j_m)$ with $0\leq j_1<\cdots< j_m\leq n$,
 we will construct the corresponding weighted digraph.
  Let ${\Gamma^{\ast}_{I,J}}$ be the subgraph of
  $\Gamma^{\ast}_{n}$ induced by  the union of arcs of all directed paths from
  $Q_{n-i_s}^{(n)}$ to $Q_{n-j_k}^{(0)}$ $(1\leq s,k\leq m)$. By  Proposition \ref{prop+T1+M}
  we immediately have
  $$B_{I,J}^{\ast}=M_{I,J}.$$
  And it is obvious that
  $(Q_{n-i_1}^{(n)},Q_{n-i_2}^{(n)},\ldots,Q_{n-i_m}^{(n)})$ and
  $(Q_{n-j_1}^{(0)},Q_{n-j_2}^{(0)},\ldots,Q_{n-j_m}^{(0)})$ form a fully nonpermutable pair.
  In consequence, by  Corollary \ref{cor.gessel-viennot.totpos}, we establish the $(\textbf{x},\textbf{y})$-total positivity
  of the matrix $M$.

   (ii) Let $\mathcal{T}$ be the Toeplitz matrix of the $n$th row sequence of $M$. Fix $k\geq0$.
   For any $m\times m$ minor $\mathcal{T}_{I,J}$ of $\mathcal{T}_{k}$, where $I=(i_1,\ldots,i_m)$ with $0\leq i_1< \cdots < i_m\leq k$
   and $J=(j_1,\ldots,j_m)$ with $0\leq j_1<\cdots< j_m\leq k$, we will construct the corresponding weighted digraph.
   Let $\Gamma^{\diamond}_{I,J}$ be the subgraph of
  $\Gamma^{\diamond}_{k}$ induced by the union of arcs of all directed paths from
  $Q_{i_s}^{(n+i_s)}$ to $Q_{n+j_r}^{(0)}$ $(1\leq s,r\leq m)$. By Proposition \ref{prop+T+TG+trid},
  we immediately have
  $${B_{I,J}^{\diamond}}=\mathcal{T}_{I,J}.$$  Combining Lemma
  \ref{lemma.nonpermutable} and
  Corollary \ref{cor.gessel-viennot.totpos} gives the $(\textbf{x},\textbf{y})$-total positivity for the Toeplitz matrix.

This completes the proof of Theorem \ref{thm+tridiag}.\qed

\section{The proof of Theorem \ref{thm+fun+ell}}\label{section+Proof+Thm+full+ell}
In this section, we shall give the proof of Theorem
\ref{thm+fun+ell}.

\textbf{Proof of Theorem \ref{thm+fun+ell}:}

(i) In terms of the assumption that the weight $a^{(i)}_n$ in
lattice paths is independent of $n$,
let $\textsl{a}_i:=a^{(i)}_n$
for $0\leq i\leq \ell$.  In consequence, we have
 $$\sum_{i= 0}^{\ell}\textsl{a}_iz^i=\prod_{j= 1}^{\ell}(\alpha_jz+\beta_j)$$
 and
 $$\mathcal{A}=\left(
                   \begin{array}{ccccc}
                                 \textsl{a}_0 \\
                                \textsl{a}_1& \textsl{a}_0 \\
                               \textsl{a}_2& \textsl{a}_1 & \textsl{a}_0\\
                                \textsl{a}_3& \textsl{a}_2 & \textsl{a}_1 & \textsl{a}_0 \\
                                \vdots & \vdots & \vdots & \vdots & \ddots\\
                  \end{array}
   \right),$$
  which is exactly the Toeplitz matrix of the sequence $(\textsl{a}_i)_{i=0}^{\ell}$.

  Let
   $$W_i
    :=\left(
               \begin{array}{ccccccc}
                                  \beta_i&&&&&& \\
                                  \alpha_i&\beta_i \\
                                  &\alpha_i&\beta_i\\
                                  &&\alpha_i&\beta_i\\
                                  &&&\ddots&\ddots\\
               \end{array}
            \right).$$
  Therefore, we derive
  \begin{equation}\label{Wi}
    \mathcal{A}=\prod\limits_{i=1}^{\ell}W_i.
  \end{equation}
Obviously, in terms of the classical Cauchy-Binet formula,
$\mathcal{A}$ is $(\balpha,\bbeta)$-totally positive. So (i) is immediate by
Theorem \ref{thm+main} (ii).

(ii) and (iv)
  Let $h_k(z)=\sum_{n}M_{n,k}z^n$ be the $k$th column-generating function of $[M_{n,k}]_{n,k}$.
  Notice that $M_{0,0}=1$ and $M_{n,0}={\gamma}^{n}$. So $$h_0(z)=\frac{1}{1-\gamma
  z}.$$
  Multiplying both sides of (\ref{rec+1+M}) by $z^n$ and summing over $n\geq0$, we have
  $$h_k(z)=\gamma zh_k(z)+h_{k-1}(z)\sum_{i= 0}^{\ell}\textsl{a}_iz^{t+i}$$
  for $k\geq1$,
  which immediately implies for $k\geq0$ that
  \begin{eqnarray}\label{function+column}
 h_k(z)&=&\frac{\left(\sum_{i= 0}^{\ell}\textsl{a}_iz^{t+i}\right)^k}{(1-\gamma z)^{k+1}}.
\end{eqnarray}
In consequence, the matrix $M$ is the Riordan array
$\mathcal {R}\left(\frac{1}{1-\gamma z},\frac{\sum_{i=
0}^{\ell}a^{(i)}_nz^{t+i}}{1-\gamma z}\right)$, which is (iv).

Note that the coefficient sequence of the product of two formal powerful
series is the convolution of their coefficient sequences. Let
$\textbf{\textsl{a}}=(\textsl{a}_i)_{i=0}^{\ell}$, $\bgamma=(\gamma^i)_{i\geq0}$,
and $\bmu=(M_{n,k})_{n\geq kt}$. By (\ref{function+column}), for the
Toeplitz matrix $\mathcal{T}(\bmu)$ of the column sequence $\bmu$, we have
$$\mathcal{T}(\bmu)=\mathcal{T}^{k+1}(\gamma)\mathcal{T}^k(\textbf{\textsl{a}})=\mathcal{T}^{k+1}(\gamma)\prod\limits_{i=1}^{\ell}\left(
               \begin{array}{ccccccc}
                                  \beta_i&&&&&& \\
                                  \alpha_i&\beta_i \\
                                  &\alpha_i&\beta_i\\
                                  &&\alpha_i&\beta_i\\
                                  &&&\ddots&\ddots\\
               \end{array}
            \right)^k.$$
In consequence, the Toeplitz matrix $\mathcal{T}(\bmu)$ is
$(\balpha,\bbeta,\gamma)$-totally positive from the classical Cauchy-Binet
formula. This proves that (ii) holds.

 (v)  By (iv) and the definition of the Riordan array, we have
  \begin{eqnarray*}
   M_{n,k}&=&[z^n]\frac{1}{1-\gamma z}\left(\frac{\sum_{i=0}^{\ell}a^{(i)}_nz^{t+i}}{1-\gamma z}\right)^k\nonumber\\
   &=&[z^n]\frac{z^{tk}\prod_{j= 1}^{\ell}(\alpha_jz+\beta_j)^k}{(1-\gamma z)^{k+1}}\nonumber\\
   &=&[z^{n-tk}]\sum\limits_{i\geq0}\binom{k+i}{i}(\gamma z)^i\sum\limits_{0\leq c_1,c_2,\ldots,c_{\ell}\leq k}\prod_{j= 1}^{\ell}\binom{k}{c_j}(\alpha_jz)^{c_j}{\beta_j}^{k-c_j}\nonumber\\
   &=&[z^{n-tk}]
   \!\!\!
   \sum_{\begin{scarray}
            i\geq 0  \\[-1mm]
            0\leq c_1,c_2,\ldots,c_{\ell}\leq k  \\[-1mm]
         \end{scarray}
        }
   \!\!\!
   \binom{k+i}{i}{\gamma}^i{\left(\prod_{j= 1}^{\ell}\binom{k}{c_j}{\alpha_j}^{c_j}{\beta_j}^{k-c_j}\right)}z^{i+\sum\limits_{j=1}^{\ell}c_j}\nonumber\\
   &=&
    \!\!\!
   \sum_{\begin{scarray}
            i\geq 0  \\[-1mm]
            0\leq c_1,c_2,\ldots,c_{\ell}\leq k  \\[-1mm]
            i+\sum\limits_{j=1}^{\ell}c_j=n-tk \\[-1mm]
         \end{scarray}
        }
   \!\!\!
   \binom{k+i}{i}{\gamma}^i\prod_{j=
   1}^{\ell}\binom{k}{c_j}{\alpha_j}^{c_j}{\beta_j}^{k-c_j}.
\end{eqnarray*}

  (vi) It immediately follows from (iv) and (\ref{generating funtion+GR}).

\textbf{Finally, we will present a combinatorial proof for (i) and (iii)
 of Theorem \ref{thm+fun+ell} in terms of the
Lindstr\"{o}m-Gessel-Viennot lemma. }First, we construct a weighted
digraph for the matrix $\widetilde{P}_n$ in (\ref{A}).
\begin{definition}\label{def+gamma+GG1}
Let $\Gamma_n= (V, \vec{E})$, where $ V=[n+\ell+1]\times[n+1]$, and
$\vec{E}$ is composed of  the following four kinds of arcs:
\begin{itemize}
    \item [(i)]
     $\textsl{a} =(i,j) \rightarrow (i+1,j)$, the weight
     \begin{eqnarray*}
         w_{\textsl{a}} =\begin{cases}  1       &\text{for} ~ (i,j)\in [1,n]\times[1,n]~\text{or}~i\in [1,n+\ell]~\text{and}~j=n+1 ;\\
                                       \beta_{i-n}      &\text{for}~(i,j)\in [n+2,n+\ell]\times[1,n].
                    \end{cases}
  \end{eqnarray*}
    \item [(ii)]
    $\textsl{a} = (i,j) \rightarrow (i+1,j+1)$, the weight
    \begin{eqnarray*}
         w_{\textsl{a}} =\begin{cases}   b_{n+1-i}       &\text{for} ~ 1\leq i=j\leq n;\\
                                       \alpha_{i-n}      &\text{for}~(i,j)\in [n+2,n+\ell]\times[1,n-1].
                    \end{cases}
  \end{eqnarray*}
    \item [(iii)]
    $\textsl{a} = (n+1,j) \rightarrow (n+2,j+t-1)$, the weight $w_{\textsl{a}} = \beta_1$ for $1\leq j\leq n+1-t$.
    \item [(iv)]
    $\textsl{a} = (n+1,j) \rightarrow (n+2,j+t)$, the weight $w_{\textsl{a}} = \alpha_1$ for $1\leq j\leq n-t$ .
\end{itemize}
Denote the vertex $(1,j)$ (resp. $(n+\ell+1,j)$, $(j,j)$) by $Q_{j-1}^{(n)}$ (resp.
$Q_{j-1}^{(n-1)}$, $R_{j-1}^{(n)}$) for $1\leq j\leq n+1$. Furthermore,
vertices $Q_n^{(n)},Q_{n-1}^{(n)},\ldots,Q_0^{(n)}$~(resp. $Q_n^{(n-1)},Q_{n-1}^{(n-1)},\ldots,Q_0^{(n-1)}$) are viewed as the source vertices (resp. the sink vertices).
\end{definition}

See Figure~12 for the case that $t=1,n=3$.

\begin{center}
\setlength{\unitlength}{0.95cm}
\begin{picture}(7,5)(7,1.6)

\thicklines\put(5.5,6){\line(1,0){1.5}}\thicklines\put(10,6){\line(1,0){1.5}}
 \thicklines\put(5.5,5){\line(1,0){1.5}}\thicklines\put(10,5){\line(1,0){1.5}}
\thicklines\put(5.5,4){\line(1,0){1.5}}\thicklines\put(10,4){\line(1,0){1.5}}
 \thicklines\put(5.5,3){\line(1,0){1.5}}\thicklines\put(10,3){\line(1,0){1.5}}

\thicklines\put(7,6){\line(1,0){1.5}} \thicklines\put(8.5,6){\line(1,0){1.5}}
\thicklines\put(7,5){\line(1,0){1.5}} \thicklines\put(8.5,5){\line(1,0){1.5}}
\thicklines\put(7,4){\line(1,0){1.5}} \thicklines\put(8.5,4){\line(1,0){1.5}}
\thicklines\put(7,3){\line(1,0){1.5}} \thicklines\put(8.5,3){\line(1,0){1.5}}
\thicklines\put(11.5,6){\line(1,0){1.5}}\thicklines\put(14,6){\line(1,0){1.5}}
\thicklines\put(11.5,5){\line(1,0){1.5}}\thicklines\put(14,5){\line(1,0){1.5}}
\thicklines\put(11.5,4){\line(1,0){1.5}}\thicklines\put(14,4){\line(1,0){1.5}}
\thicklines\put(11.5,3){\line(1,0){1.5}}\thicklines\put(14,3){\line(1,0){1.5}}

\thicklines\put(5.5,3){\line(3,2){1.5}}   \thicklines\put(10,4){\line(3,2){1.5}}
\thicklines\put(8.5,5){\line(3,2){1.5}}  \thicklines\put(10,3){\line(3,2){1.5}}
   \thicklines\put(7,4){\line(3,2){1.5}}   \thicklines\put(14,4){\line(3,2){1.5}}
\thicklines\put(11.5,3){\line(3,2){1.5}}
 \thicklines\put(11.5,4){\line(3,2){1.5}} \thicklines\put(14,3){\line(3,2){1.5}}

 \put(5.5,6){\circle*{0.16}} \put(7,6){\circle*{0.16}}
 \put(5.5,5){\circle*{0.16}} \put(7,5){\circle*{0.16}}
 \put(5.5,4){\circle*{0.16}} \put(7,4){\circle*{0.16}}
 \put(5.5,3){\circle*{0.16}} \put(7,3){\circle*{0.16}}

\put(8.5,6){\circle*{0.16}} \put(10,6){\circle*{0.16}} \put(11.5,6){\circle*{0.16}}
\put(8.5,5){\circle*{0.16}} \put(10,5){\circle*{0.16}} \put(11.5,5){\circle*{0.16}}
\put(8.5,4){\circle*{0.16}} \put(10,4){\circle*{0.16}} \put(11.5,4){\circle*{0.16}}
\put(8.5,3){\circle*{0.16}} \put(10,3){\circle*{0.16}} \put(11.5,3){\circle*{0.16}}

\put(13,6){\circle*{0.16}}\put(14,6){\circle*{0.16}}\put(15.5,6){\circle*{0.16}}
\put(13,5){\circle*{0.16}}\put(14,5){\circle*{0.16}}\put(15.5,5){\circle*{0.16}}
\put(13,4){\circle*{0.16}}\put(14,4){\circle*{0.16}}\put(15.5,4){\circle*{0.16}}
\put(13,3){\circle*{0.16}}\put(14,3){\circle*{0.16}}\put(15.5,3){\circle*{0.16}}
\put(13,6){\circle*{0.16}}\put(14,6){\circle*{0.16}}\put(15.5,6){\circle*{0.16}}
\put(13.3,6){\circle*{0.08}}\put(13.5,6){\circle*{0.08}}\put(13.7,6){\circle*{0.08}}
\put(13.3,5){\circle*{0.08}}\put(13.5,5){\circle*{0.08}}\put(13.7,5){\circle*{0.08}}
\put(13.3,4){\circle*{0.08}}\put(13.5,4){\circle*{0.08}}\put(13.7,4){\circle*{0.08}}
\put(13.3,3){\circle*{0.08}}\put(13.5,3){\circle*{0.08}}\put(13.7,3){\circle*{0.08}}

\tiny

\put(4.92,6.158){$Q_3^{(3)}$}      \put(6.2,6.1){1} \put(7.7,6.1){1} \put(9.2,6.1){1}
\put(4.92,5.158){$Q_{2}^{(3)}$}  \put(6.2,5.1){1} \put(7.7,5.1){1} \put(9.2,5.1){1}
\put(4.92,4.158){$Q_{1}^{(3)}$}  \put(6.2,4.1){1} \put(7.7,4.1){1} \put(9.2,4.1){1}
\put(4.92,3.158){$Q_{0}^{(3)}$}  \put(6.2,3.1){1} \put(7.7,3.1){1} \put(9.2,3.1){1}
\put(4.92,2.55){$(R_{0}^{(3)})$} \put(6.7,3.55){$R_{1}^{(3)}$}\put(8.2,4.55){$R_{2}^{(3)}$}
\put(9.7,5.55){$R_{3}^{(3)}$}  \put(15.65,2.55){$(R_{0}^{(2)})$}

\put(10.7,6.1){$1$}              \put(15.65,6.105){$Q_3^{(2)}$}
\put(10.7,5.1){$\beta_1$}     \put(15.65,5.105){$Q_{2}^{(2)}$}
\put(10.7,4.1){$\beta_1$}     \put(15.65,4.105){$Q_{1}^{(2)}$}   \put(8.93,5.6){$b_1$}       \put(10.7,6.1){1}
\put(10.7,3.1){$\beta_1$}     \put(15.65,3.105){$Q_{0}^{(2)}$}   \put(7.43,4.6){$b_2$}      \put(10.4,4.6){$\alpha_1$}
                                                           \put(5.93,3.6){$b_3$}        \put(10.4,3.6){$\alpha_1$}

\put(12.2,6.1){$1$}         \put(14.7,6.1){$1$}
\put(12.2,5.1){$\beta_2$}\put(14.7,5.1){$\beta_{\ell}$}
\put(12.2,4.1){$\beta_2$} \put(14.7,4.1){$\beta_{\ell}$}            \put(11.9,4.6){$\alpha_2$}    \put(14.4,4.6){$\alpha_{\ell}$}
\put(12.2,3.1){$\beta_2$} \put(14.7,3.1){$\beta_{\ell}$}            \put(11.9,3.6){$\alpha_2$}    \put(14.4,3.6){$\alpha_{\ell}$}

\normalsize
\put(5.2,1.7){Figure~12: $\Gamma_3 $ for $t=1$. All edges point towards the right.}
\end{picture}
\end{center}

Here we use the sequence of $(Q_{i}^{(n)})_{n,i\geq0}$ to construct
the corresponding weighted digraphs for the matrix $M$ and the
Toeplitz matrix of the each row sequence of $M$, and we use the
sequence of $(R_{i}^{(n)})_{n,i\geq0}$ to construct the
corresponding weighted digraph for the Toeplitz matrix of the
sequence $(M_{n+\delta i,k+\sigma i})_{i}$ with $\sigma>\delta$.

Denote by $B(n)$ the walk matrix of weighted digraph $\Gamma_n$ in
Definition \ref{def+gamma+GG1}, then the transfer-matrix method
\cite[Theorem 4.7.1]{Stanley12} establishes the relation between the
walk matrix $B(n)$ and the matrix $\widetilde{P}_n$.

\begin{prop}\label{prop+Pn}
The walk matrix $B(n)$ for the weighted digraph $\Gamma_n$ equals
the matrix $\widetilde{P}_n$.
\end{prop}
\textbf{Proof:}
By (\ref{connection}) and (\ref{Wi}), we have
\begin{eqnarray*}\label{Pn}
  P&=&\prod\limits_{i\geq0} E_{i+1,i}[b_{i+1}]\Delta\nonumber\\
   &=&\prod\limits_{i\geq0} E_{i+1,i}[b_{i+1}](1,\mathcal{A})\nonumber\\
   &=&\prod\limits_{i\geq0} E_{i+1,i}[b_{i+1}](1,W_1)\left(
                    \begin{array}{ccccccc}
                                  1& \\
                                  &\prod\limits_{i=2}^{\ell}W_i \\
                  \end{array}
            \right).
\end{eqnarray*}
Let $\widetilde{I}$ be a submatrix of identity matrix $I$ by
deleting columns from $1$th to $t$th. Since $\widetilde{P}$ is a
submatrix of $P$ by deleting columns from $1$th to $t$th, we obtain
\begin{eqnarray*}
  \widetilde{P}&=&\prod\limits_{i\geq0} E_{i+1,i}[b_{i+1}](1,W_1)\left(
                    \begin{array}{ccccccc}
                                  1& \\
                                  &\prod\limits_{i=2}^{\ell}W_i \\
                  \end{array}
            \right)\widetilde{I}.
\end{eqnarray*}
It is obvious that the matrix
$\left( \begin{array}{ccccccc}
                                  1& \\
                                  &\prod\limits_{i=2}^{\ell}W_i \\
        \end{array}
\right)$ deleting columns from $1$th to $t$th is the same as the
matrix $\left( \begin{array}{ccccccc}
                                  1& \\
                                  &\prod\limits_{i=2}^{\ell}W_i \\
        \end{array}
\right)$ shifting down $t$ rows from the first row. Therefore, we
derive
\begin{eqnarray*}
  \widetilde{P} &=&\prod\limits_{i\geq0} E_{i+1,i}[b_{i+1}](1,W_1)\widetilde{I}\left(
                    \begin{array}{ccccccc}
                                  1& \\
                                  &\prod\limits_{i=2}^{\ell}W_i \\
                  \end{array}
            \right).
\end{eqnarray*}
Let
$${\Omega}_1=(1,W_1)~\text{and}~{\Omega}_i=\left(
                    \begin{array}{ccccccc}
                                  1& \\
                                  &W_i \\
                  \end{array}
            \right)~\text{for}~2\leq i\leq \ell.$$
Hence, $\widetilde{P}$ has the following decomposition
\begin{eqnarray}\label{L}
  \widetilde{P}&=&\prod\limits_{i\geq0} E_{i+1,i}[b_{i+1}]\widetilde{{\Omega}}_1\prod\limits_{i=2}^{\ell}{\Omega}_i,
\end{eqnarray}
where $\widetilde{\Omega}_{1}$ is a submatrix of ${\Omega}_{1}$ by
deleting columns from $1$th to $t$th. Applying the transfer-matrix
method \cite[Theorem 4.7.1]{Stanley12} to equality (\ref{L}), we
derive the desired result.\qed

Now we will present a weighted digraph $\Gamma^{\ast}_n$ for $M$
with the help of $\Gamma_n$.

\begin{definition}\label{definition+star+g}
Define recursively a weighted digraph, denoted by
$\Gamma^{\ast}_n=(V^{\ast}_n, \vec{E}^{\ast}_n)$, as follows:
\begin{itemize}
  \item [(i)]
  For $n=1$, we take the weighted digraph $\Gamma_1$ as
  $\Gamma^{\ast}_1$ (see Figure~13 for the case $t=1$).
  \begin{center}
  \setlength{\unitlength}{0.6cm}
  \begin{picture}(3,3.3)(3.5,-1.4)
  \thicklines\put(1,1){\line(1,0){1.5}}  \thicklines\put(2.5,1){\line(1,0){1.5}}\thicklines\put(4,1){\line(1,0){1.5}}
  \thicklines\put(1,0){\line(1,0){1.5}}  \thicklines\put(2.5,0){\line(1,0){1.5}}\thicklines\put(4,0){\line(1,0){1.5}}

  \thicklines\put(6.5,0){\line(1,0){1.5}} \thicklines\put(6.5,1){\line(1,0){1.5}}\thicklines\put(1,0){\line(3,2){1.5}}

  \put(1,1){\circle*{0.2}} \put(2.5,1){\circle*{0.2}} \put(4,1){\circle*{0.2}}\put(5.5,1){\circle*{0.2}}\put(6.5,1){\circle*{0.2}}
  \put(1,0){\circle*{0.2}} \put(2.5,0){\circle*{0.2}} \put(4,0){\circle*{0.2}}\put(5.5,0){\circle*{0.2}}\put(6.5,0){\circle*{0.2}}
  \put(5.8,0){\circle*{0.1}}\put(6,0){\circle*{0.1}}\put(6.2,0){\circle*{0.1}}
  \put(5.8,1){\circle*{0.1}}\put(6,1){\circle*{0.1}}\put(6.2,1){\circle*{0.1}}
  \put(8,0){\circle*{0.2}}
  \put(8,1){\circle*{0.2}}
  \tiny

  \put(0.1,1.1){$Q_1^{(1)}$} \put(8.2,1.105){$Q_1^{(0)}$}\put(2.2,0.5){$R_1^{(1)}$}
  \put(0.1,0.1){$Q_0^{(1)}$} \put(8.2,0.105){$Q_0^{(0)}$}
                        \put(1.7,1.1){$1$} \put(3.2,1.1){$1$}         \put(4.7,1.1){$1$}        \put(7.2,1.1){$1$}
  \put(1.35,0.6){$b_1$}  \put(1.7,0.1){$1$} \put(3.2,0.15){$\beta_1$}\put(4.7,0.15){$\beta_2$}\put(7.2,0.15){$\beta_{\ell}$}

  \normalsize
  \put(-4,-1.3){Figure~13: $\Gamma^{\ast}_1$  for $t=1$. All edges point towards the right.}
  \end{picture}
  \end{center}
  \item [(ii)]
  Suppose that $\Gamma^{\ast}_{n-1}=(V^{\ast}_{n-1}, \vec{E}^{\ast}_{n-1})$ has been constructed for some
  $n\geq2$.
  Then $\Gamma^{\ast}_n$ is the union of $\Gamma_n$ and $\overline{\Gamma}^{\ast}_{n-1}$ by gluing vertices $Q^{(n-1)}_{i}$ $(0\leq i\leq n)$, where
   $\overline{\Gamma}^{\ast}_{n-1}$ is obtained from $\Gamma^{\ast}_{n-1}$ by adding vertices
  $Q_n^{(n-1)},\ldots,Q_{n}^{(0)}$ to $V^{\ast}_{n-1}$ and
   arcs $Q_{n}^{(i+1)} \rightarrow Q_n^{(i)}$ with weight $1$ to
   $\vec{E}^{\ast}_{n-1}$ $(0\leq i\leq n-2)$.
\end{itemize}
\end{definition}
By Definition \ref{definition+star+g}, we observe that vertices
$Q_n^{(n)},Q_{n-1}^{(n)},\ldots,Q_0^{(n)}$~ (resp.
$Q_n^{(0)},Q_{n-1}^{(0)},\ldots,Q_0^{(0)}$) are viewed as the source
vertices (resp. the sink vertices) of the weighted digraph
$\Gamma^{\ast}_n$. In order to understand well how to construct the
weighted digraph $\Gamma^{\ast}_n$, for $t=1$, we draw the weighted
digraphs $\Gamma^{\ast}_2$, $\overline{\Gamma}^{\ast}_2$ and
$\Gamma^{\ast}_3$ (see Figures 14-16 with omitting weight $1$),
respectively.

 \begin{center}
  \setlength{\unitlength}{0.6cm}
  \begin{picture}(3,5)(-1.2,-2)

  \thicklines\put(-7.5,2){\line(1,0){1.5}} \thicklines\put(-6,2){\line(1,0){1.5}}
  \thicklines\put(-7.5,1){\line(1,0){1.5}}\thicklines\put(-6,1){\line(1,0){1.5}}
  \thicklines\put(-7.5,0){\line(1,0){1.5}}\thicklines\put(-6,0){\line(1,0){1.5}}

  \thicklines\put(-4.5,2){\line(1,0){1.5}}\thicklines\put(-3,2){\line(1,0){1.5}}\thicklines\put(-0.5,2){\line(1,0){1.5}}
 \thicklines\put(-4.5,1){\line(1,0){1.5}}\thicklines\put(-3,1){\line(1,0){1.5}}
 \thicklines\put(-0.5,1){\line(1,0){1.5}}
  \thicklines\put(-4.5,0){\line(1,0){1.5}}\thicklines\put(-3,0){\line(1,0){1.5}}\thicklines\put(-0.5,0){\line(1,0){1.5}}

  \thicklines\put(1,2){\line(1,0){7}}
 \thicklines\put(1,1){\line(1,0){1.5}}  \thicklines\put(2.5,1){\line(1,0){1.5}}\thicklines\put(4,1){\line(1,0){1.5}}
  \thicklines\put(1,0){\line(1,0){1.5}} \thicklines\put(2.5,0){\line(1,0){1.5}}\thicklines\put(4,0){\line(1,0){1.5}}

  \thicklines\put(6.5,0){\line(1,0){1.5}} \thicklines\put(6.5,1){\line(1,0){1.5}} \thicklines\put(1,0){\line(3,2){1.5}}
  \thicklines\put(-7.5,0){\line(3,2){1.5}} \thicklines\put(-6,1){\line(3,2){1.5}}\thicklines\put(-4.5,0){\line(3,2){1.5}}
   \thicklines\put(-3,0){\line(3,2){1.5}}\thicklines\put(-0.5,0){\line(3,2){1.5}}

  \put(-0.5,2){\circle*{0.2}}\put(-1.5,2){\circle*{0.2}}\put(-3,2){\circle*{0.2}}\put(-4.5,2){\circle*{0.2}}\put(-6,2){\circle*{0.2}}
  \put(-0.5,1){\circle*{0.2}}\put(-1.5,1){\circle*{0.2}}\put(-3,1){\circle*{0.2}}\put(-4.5,1){\circle*{0.2}}\put(-6,1){\circle*{0.2}}
  \put(-0.5,0){\circle*{0.2}}\put(-1.5,0){\circle*{0.2}}\put(-3,0){\circle*{0.2}}\put(-4.5,0){\circle*{0.2}}\put(-6,0){\circle*{0.2}}

  \put(1,2){\circle*{0.2}}                                                                       \put(8,2){\circle*{0.2}}
  \put(1,1){\circle*{0.2}} \put(2.5,1){\circle*{0.2}} \put(4,1){\circle*{0.2}}\put(5.5,1){\circle*{0.2}}\put(6.5,1){\circle*{0.2}}
  \put(1,0){\circle*{0.2}} \put(2.5,0){\circle*{0.2}} \put(4,0){\circle*{0.2}}\put(5.5,0){\circle*{0.2}}\put(6.5,0){\circle*{0.2}}
  \put(-1.2,0){\circle*{0.1}}\put(-1,0){\circle*{0.1}}\put(-0.8,0){\circle*{0.1}}
  \put(-1.2,1){\circle*{0.1}}\put(-1,1){\circle*{0.1}}\put(-0.8,1){\circle*{0.1}}
  \put(-1.2,2){\circle*{0.1}}\put(-1,2){\circle*{0.1}}\put(-0.8,2){\circle*{0.1}}
  \put(5.8,0){\circle*{0.1}}\put(6,0){\circle*{0.1}}\put(6.2,0){\circle*{0.1}}
  \put(5.8,1){\circle*{0.1}}\put(6,1){\circle*{0.1}}\put(6.2,1){\circle*{0.1}}

  \put(-7.5,0){\circle*{0.2}}\put(8,0){\circle*{0.2}}
  \put(-7.5,1){\circle*{0.2}} \put(8,1){\circle*{0.2}}
  \put(-7.5,2){\circle*{0.2}}
  \tiny
\put(8.2,0){$Q_0^{(0)}$}\put(8.2,1){$Q_1^{(0)}$}\put(8.2,2){$Q_2^{(0)}$}
  \put(0.7,-0.6){$Q_0^{(1)}$}  \put(2.2,0.5){$R_1^{(1)}$}
  \put(-8.45,0){$Q_0^{(2)}$} \put(-6.2,0.5){$R_1^{(2)}$}\put(-4.8,1.5){$R_2^{(2)}$}
\put(1.3,0.6){$b_1$}
\put(-5.7,1.6){$b_1$}   \put(-7.2,0.6){$b_2$}
\put(-4.3,0.6){$\alpha_1$}\put(-2.8,0.6){$\alpha_2$}\put(-0.3,0.6){$\alpha_{\ell}$}
\put(0.7,1.3){$Q_1^{(1)}$} \put(0.7,2.3){$Q_2^{(1)}$}
\put(-8.45,1){$Q_1^{(2)}$} \put(-8.45,2){$Q_2^{(2)}$}
\put(-4,-0.35){$\beta_1$}\put(-2.5,-0.35){$\beta_2$}\put(0,-0.35){$\beta_{\ell}$}
\put(-4,1.15){$\beta_1$}\put(-2.5,1.15){$\beta_2$}\put(0,1.15){$\beta_{\ell}$}
\put(3,-0.35){$\beta_1$}\put(4.5,-0.35){$\beta_2$}\put(7,-0.35){$\beta_{\ell}$}
  \normalsize
  \put(-8,-1.9){Figure~14: ${\Gamma}^{\ast}_2$  for $ t=1$. All edges point towards the right.}
  \end{picture}
  \end{center}

  \begin{center}
  \setlength{\unitlength}{0.6cm}
  \begin{picture}(3,6)(-1.2,-2)

 \thicklines\put(-7.5,3){\line(1,0){15.5}}
  \thicklines\put(-7.5,2){\line(1,0){1.5}} \thicklines\put(-6,2){\line(1,0){1.5}}
  \thicklines\put(-7.5,1){\line(1,0){1.5}}\thicklines\put(-6,1){\line(1,0){1.5}}
  \thicklines\put(-7.5,0){\line(1,0){1.5}}\thicklines\put(-6,0){\line(1,0){1.5}}

  \thicklines\put(-4.5,2){\line(1,0){1.5}}\thicklines\put(-3,2){\line(1,0){1.5}}\thicklines\put(-0.5,2){\line(1,0){1.5}}
 \thicklines\put(-4.5,1){\line(1,0){1.5}}\thicklines\put(-3,1){\line(1,0){1.5}}
 \thicklines\put(-0.5,1){\line(1,0){1.5}}
  \thicklines\put(-4.5,0){\line(1,0){1.5}}\thicklines\put(-3,0){\line(1,0){1.5}}\thicklines\put(-0.5,0){\line(1,0){1.5}}

  \thicklines\put(1,2){\line(1,0){7}}
 \thicklines\put(1,1){\line(1,0){1.5}}  \thicklines\put(2.5,1){\line(1,0){1.5}}\thicklines\put(4,1){\line(1,0){1.5}}
  \thicklines\put(1,0){\line(1,0){1.5}} \thicklines\put(2.5,0){\line(1,0){1.5}}\thicklines\put(4,0){\line(1,0){1.5}}

  \thicklines\put(6.5,0){\line(1,0){1.5}} \thicklines\put(6.5,1){\line(1,0){1.5}} \thicklines\put(1,0){\line(3,2){1.5}}
  \thicklines\put(-7.5,0){\line(3,2){1.5}} \thicklines\put(-6,1){\line(3,2){1.5}}\thicklines\put(-4.5,0){\line(3,2){1.5}}
   \thicklines\put(-3,0){\line(3,2){1.5}}\thicklines\put(-0.5,0){\line(3,2){1.5}}
 \put(-7.5,3){\circle*{0.2}}\put(8,3){\circle*{0.2}}

  \put(-0.5,2){\circle*{0.2}}\put(-1.5,2){\circle*{0.2}}\put(-3,2){\circle*{0.2}}\put(-4.5,2){\circle*{0.2}}\put(-6,2){\circle*{0.2}}
  \put(-0.5,1){\circle*{0.2}}\put(-1.5,1){\circle*{0.2}}\put(-3,1){\circle*{0.2}}\put(-4.5,1){\circle*{0.2}}\put(-6,1){\circle*{0.2}}
  \put(-0.5,0){\circle*{0.2}}\put(-1.5,0){\circle*{0.2}}\put(-3,0){\circle*{0.2}}\put(-4.5,0){\circle*{0.2}}\put(-6,0){\circle*{0.2}}

  \put(1,2){\circle*{0.2}}  \put(1,3){\circle*{0.2}}                                                                         \put(8,2){\circle*{0.2}}
  \put(1,1){\circle*{0.2}} \put(2.5,1){\circle*{0.2}} \put(4,1){\circle*{0.2}}\put(5.5,1){\circle*{0.2}}\put(6.5,1){\circle*{0.2}}
  \put(1,0){\circle*{0.2}} \put(2.5,0){\circle*{0.2}} \put(4,0){\circle*{0.2}}\put(5.5,0){\circle*{0.2}}\put(6.5,0){\circle*{0.2}}
  \put(-1.2,0){\circle*{0.1}}\put(-1,0){\circle*{0.1}}\put(-0.8,0){\circle*{0.1}}
  \put(-1.2,1){\circle*{0.1}}\put(-1,1){\circle*{0.1}}\put(-0.8,1){\circle*{0.1}}
  \put(-1.2,2){\circle*{0.1}}\put(-1,2){\circle*{0.1}}\put(-0.8,2){\circle*{0.1}}
  \put(5.8,0){\circle*{0.1}}\put(6,0){\circle*{0.1}}\put(6.2,0){\circle*{0.1}}
  \put(5.8,1){\circle*{0.1}}\put(6,1){\circle*{0.1}}\put(6.2,1){\circle*{0.1}}

  \put(-7.5,0){\circle*{0.2}}\put(8,0){\circle*{0.2}}
  \put(-7.5,1){\circle*{0.2}} \put(8,1){\circle*{0.2}}
  \put(-7.5,2){\circle*{0.2}}
  \tiny
\put(8.2,0){$Q_0^{(0)}$}\put(8.2,1){$Q_1^{(0)}$}\put(8.2,2){$Q_2^{(0)}$}\put(8.2,3){$Q_3^{(0)}$}
  \put(0.7,-0.6){$Q_0^{(1)}$}  \put(2.2,0.5){$R_1^{(1)}$}
  \put(-8.45,0){$Q_0^{(2)}$} \put(-6.2,0.5){$R_1^{(2)}$}\put(-4.8,1.5){$R_2^{(2)}$}
\put(1.3,0.6){$b_1$}
\put(-5.7,1.6){$b_1$}   \put(-7.2,0.6){$b_2$}
\put(-4.3,0.6){$\alpha_1$}\put(-2.8,0.6){$\alpha_2$}\put(-0.3,0.6){$\alpha_{\ell}$}
\put(0.7,1.3){$Q_1^{(1)}$} \put(0.7,2.3){$Q_2^{(1)}$} \put(0.7,3.3){$Q_3^{(1)}$}
\put(-8.45,1){$Q_1^{(2)}$} \put(-8.45,2){$Q_2^{(2)}$} \put(-8.45,3){$Q_3^{(2)}$}
\put(-4,-0.35){$\beta_1$}\put(-2.5,-0.35){$\beta_2$}\put(0,-0.35){$\beta_{\ell}$}
\put(-4,1.15){$\beta_1$}\put(-2.5,1.15){$\beta_2$}\put(0,1.15){$\beta_{\ell}$}
\put(3,-0.35){$\beta_1$}\put(4.5,-0.35){$\beta_2$}\put(7,-0.35){$\beta_{\ell}$}
  \normalsize
  \put(-8,-1.9){Figure~15: $\overline{\Gamma}^{\ast}_2$  for $ t=1$. All edges point towards the right.}
  \end{picture}
  \end{center}

   \begin{center}
  \setlength{\unitlength}{0.58cm}
  \begin{picture}(-7,6)(-1.2,-2)
  \thicklines\put(-17.5,3){\line(1,0){6}}\thicklines\put(-11.5,3){\line(1,0){1.5}} \thicklines\put(-9,3){\line(1,0){1.5}}
 \thicklines\put(-17.5,2){\line(1,0){6}}\thicklines\put(-11.5,2){\line(1,0){1.5}} \thicklines\put(-9,2){\line(1,0){1.5}}
 \thicklines\put(-17.5,1){\line(1,0){6}}\thicklines\put(-11.5,1){\line(1,0){1.5}} \thicklines\put(-9,1){\line(1,0){1.5}}
 \thicklines\put(-17.5,0){\line(1,0){6}}\thicklines\put(-11.5,0){\line(1,0){1.5}} \thicklines\put(-9,0){\line(1,0){1.5}}
 \thicklines\put(-7.5,3){\line(1,0){15.5}}
  \thicklines\put(-7.5,2){\line(1,0){1.5}} \thicklines\put(-6,2){\line(1,0){1.5}}
  \thicklines\put(-7.5,1){\line(1,0){1.5}}\thicklines\put(-6,1){\line(1,0){1.5}}
  \thicklines\put(-7.5,0){\line(1,0){1.5}}\thicklines\put(-6,0){\line(1,0){1.5}}

  \thicklines\put(-17.5,0){\line(3,2){1.5}}\thicklines\put(-16,1){\line(3,2){1.5}}\thicklines\put(-14.5,2){\line(3,2){1.5}}
 \thicklines\put(-13,0){\line(3,2){1.5}}\thicklines\put(-13,1){\line(3,2){1.5}}\thicklines\put(-11.5,0){\line(3,2){1.5}}
 \thicklines\put(-11.5,1){\line(3,2){1.5}} \thicklines\put(-9,1){\line(3,2){1.5}}\thicklines\put(-9,0){\line(3,2){1.5}}

  \thicklines\put(-4.5,2){\line(1,0){1.5}}\thicklines\put(-3,2){\line(1,0){1.5}}\thicklines\put(-0.5,2){\line(1,0){1.5}}
 \thicklines\put(-4.5,1){\line(1,0){1.5}}\thicklines\put(-3,1){\line(1,0){1.5}}
 \thicklines\put(-0.5,1){\line(1,0){1.5}}
  \thicklines\put(-4.5,0){\line(1,0){1.5}}\thicklines\put(-3,0){\line(1,0){1.5}}\thicklines\put(-0.5,0){\line(1,0){1.5}}

  \thicklines\put(1,2){\line(1,0){7}}
 \thicklines\put(1,1){\line(1,0){1.5}}  \thicklines\put(2.5,1){\line(1,0){1.5}}\thicklines\put(4,1){\line(1,0){1.5}}
  \thicklines\put(1,0){\line(1,0){1.5}} \thicklines\put(2.5,0){\line(1,0){1.5}}\thicklines\put(4,0){\line(1,0){1.5}}

  \thicklines\put(6.5,0){\line(1,0){1.5}} \thicklines\put(6.5,1){\line(1,0){1.5}} \thicklines\put(1,0){\line(3,2){1.5}}
  \thicklines\put(-7.5,0){\line(3,2){1.5}} \thicklines\put(-6,1){\line(3,2){1.5}}\thicklines\put(-4.5,0){\line(3,2){1.5}}
   \thicklines\put(-3,0){\line(3,2){1.5}}\thicklines\put(-0.5,0){\line(3,2){1.5}}
  \put(-17.5,3){\circle*{0.2}}\put(-16,3){\circle*{0.2}}
  \put(-17.5,2){\circle*{0.2}}\put(-16,2){\circle*{0.2}}
  \put(-17.5,1){\circle*{0.2}}\put(-16,1){\circle*{0.2}}
  \put(-17.5,0){\circle*{0.2}}\put(-16,0){\circle*{0.2}}\put(-7.5,3){\circle*{0.2}}\put(8,3){\circle*{0.2}}

 \put(-14.5,3){\circle*{0.2}}\put(-13,3){\circle*{0.2}}\put(-11.5,3){\circle*{0.2}}\put(-10,3){\circle*{0.2}}\put(-9,3){\circle*{0.2}}
  \put(-14.5,2){\circle*{0.2}}\put(-13,2){\circle*{0.2}}\put(-11.5,2){\circle*{0.2}}\put(-10,2){\circle*{0.2}}\put(-9,2){\circle*{0.2}}
  \put(-14.5,1){\circle*{0.2}}\put(-13,1){\circle*{0.2}}\put(-11.5,1){\circle*{0.2}}\put(-10,1){\circle*{0.2}}\put(-9,1){\circle*{0.2}}
  \put(-14.5,0){\circle*{0.2}}\put(-13,0){\circle*{0.2}}\put(-11.5,0){\circle*{0.2}}\put(-10,0){\circle*{0.2}}\put(-9,0){\circle*{0.2}}

  \put(-0.5,2){\circle*{0.2}}\put(-1.5,2){\circle*{0.2}}\put(-3,2){\circle*{0.2}}\put(-4.5,2){\circle*{0.2}}\put(-6,2){\circle*{0.2}}
  \put(-0.5,1){\circle*{0.2}}\put(-1.5,1){\circle*{0.2}}\put(-3,1){\circle*{0.2}}\put(-4.5,1){\circle*{0.2}}\put(-6,1){\circle*{0.2}}
  \put(-0.5,0){\circle*{0.2}}\put(-1.5,0){\circle*{0.2}}\put(-3,0){\circle*{0.2}}\put(-4.5,0){\circle*{0.2}}\put(-6,0){\circle*{0.2}}

  \put(1,2){\circle*{0.2}}  \put(1,3){\circle*{0.2}}                                                                         \put(8,2){\circle*{0.2}}
  \put(1,1){\circle*{0.2}} \put(2.5,1){\circle*{0.2}} \put(4,1){\circle*{0.2}}\put(5.5,1){\circle*{0.2}}\put(6.5,1){\circle*{0.2}}
  \put(1,0){\circle*{0.2}} \put(2.5,0){\circle*{0.2}} \put(4,0){\circle*{0.2}}\put(5.5,0){\circle*{0.2}}\put(6.5,0){\circle*{0.2}}
 \put(-1.2,0){\circle*{0.1}}\put(-1,0){\circle*{0.1}}\put(-0.8,0){\circle*{0.1}}
  \put(-1.2,1){\circle*{0.1}}\put(-1,1){\circle*{0.1}}\put(-0.8,1){\circle*{0.1}}
  \put(-1.2,2){\circle*{0.1}}\put(-1,2){\circle*{0.1}}\put(-0.8,2){\circle*{0.1}}
  \put(5.8,0){\circle*{0.1}}\put(6,0){\circle*{0.1}}\put(6.2,0){\circle*{0.1}}
  \put(5.8,1){\circle*{0.1}}\put(6,1){\circle*{0.1}}\put(6.2,1){\circle*{0.1}}

  \put(-9.3,0){\circle*{0.1}}\put(-9.5,0){\circle*{0.1}}\put(-9.7,0){\circle*{0.1}}
  \put(-9.3,1){\circle*{0.1}}\put(-9.5,1){\circle*{0.1}}\put(-9.7,1){\circle*{0.1}}
  \put(-9.3,2){\circle*{0.1}}\put(-9.5,2){\circle*{0.1}}\put(-9.7,2){\circle*{0.1}}
  \put(-9.3,3){\circle*{0.1}}\put(-9.5,3){\circle*{0.1}}\put(-9.7,3){\circle*{0.1}}

  \put(-7.5,0){\circle*{0.2}}\put(8,0){\circle*{0.2}}
  \put(-7.5,1){\circle*{0.2}} \put(8,1){\circle*{0.2}}
  \put(-7.5,2){\circle*{0.2}}
  \tiny
\put(8.2,0){$Q_0^{(0)}$}\put(8.2,1){$Q_1^{(0)}$}\put(8.2,2){$Q_2^{(0)}$}\put(8.2,3){$Q_3^{(0)}$}
  \put(0.7,-0.6){$Q_0^{(1)}$}  \put(2.2,0.5){$R_1^{(1)}$}
  \put(-7.8,-0.6){$Q_0^{(2)}$} \put(-6.2,0.5){$R_1^{(2)}$}\put(-4.8,1.5){$R_2^{(2)}$}
  \put(-18.4,0.1){$Q_0^{(3)}$} \put(-16.2,0.5){$R_1^{(3)}$}\put(-14.7,1.5){$R_2^{(3)}$} \put(-13.2,2.5){$R_3^{(3)}$}
\put(1.3,0.6){$b_1$} \put(-5.7,1.6){$b_1$}   \put(-7.2,0.6){$b_2$}
\put(-17.2,0.6){$b_3$}   \put(-15.7,1.6){$b_2$}
\put(-14.2,2.6){$b_1$}
\put(-4.3,0.6){$\alpha_1$}\put(-2.8,0.6){$\alpha_2$}\put(-0.3,0.6){$\alpha_{\ell}$}
\put(-12.9,0.6){$\alpha_1$}\put(-11.3,0.6){$\alpha_2$}\put(-8.8,0.6){$\alpha_{\ell}$}
\put(-12.9,1.6){$\alpha_1$}\put(-11.3,1.6){$\alpha_2$}\put(-8.8,1.6){$\alpha_{\ell}$}
\put(0.7,1.3){$Q_1^{(1)}$} \put(0.7,2.3){$Q_2^{(1)}$}
\put(0.7,3.3){$Q_3^{(1)}$} \put(-7.7,1.3){$Q_1^{(2)}$}
\put(-7.7,2.3){$Q_2^{(2)}$} \put(-7.7,3.3){$Q_3^{(2)}$}
\put(-18.4,1.1){$Q_1^{(3)}$} \put(-18.4,2.1){$Q_2^{(3)}$}
\put(-18.4,3.1){$Q_3^{(3)}$}
\put(-4,-0.35){$\beta_1$}\put(-2.5,-0.35){$\beta_2$}\put(0,-0.35){$\beta_{\ell}$}
\put(-4,1.15){$\beta_1$}\put(-2.5,1.15){$\beta_2$}\put(0,1.15){$\beta_{\ell}$}
\put(3,-0.35){$\beta_1$}\put(4.5,-0.35){$\beta_2$}\put(7,-0.35){$\beta_{\ell}$}
\put(-12.3,1.15){${\beta}_1$}\put(-10.8,1.15){$\beta_2$}\put(-8.3,1.15){$\beta_{\ell}$}
\put(-12.3,-0.35){$\beta_1$}\put(-10.8,-0.35){$\beta_2$}\put(-8.3,-0.35){$\beta_{\ell}$}
\put(-12.3,2.15){$\beta_1$}\put(-10.8,2.15){$\beta_2$}\put(-8.3,2.15){$\beta_{\ell}$}
  \normalsize
  \put(-13,-1.9){Figure~16: $\Gamma^{\ast}_3$  for $ t=1$. All edges point towards the right.}
  \end{picture}
  \end{center}

From the definition of the weighted digraph $\overline{\Gamma}^{\ast}_{n-1}$,
then it is easy to see that its walk matrix $\overline{B}^{\ast}(n-1)$ satisfies the following equality:
$$\overline{B}^{\ast}(n-1)=\left(\begin{array}{cc}
                                                               1&0\\
                                                               0&{M}_{n-1}\\
                                 \end{array}
                            \right).$$
The next result follows from equality (\ref{A}) and
the transfer-matrix method \cite[Theorem 4.7.1]{Stanley12}.
\begin{prop}\label{prop+GG+M}
The walk matrix $B^{\ast}(n)$ for the weighted digraph
$\Gamma^{\ast}_n$ equals the matrix $M_n$.
\end{prop}

Giving the weighted digraph interpretation of the matrix $M$, we
have the following combinatorial interpretation for the
corresponding Toeplitz matrix $\mathcal{T}$ of the each row sequence
of $M$.

\begin{definition}\label{def+gamma+GG}
Define a weighted digraph, denoted by
$\Gamma^{\diamond}_k=(V^{\diamond}_k, \vec{E}^{\diamond}_k)$, as
follows: Let $\Gamma^{\diamond}_{k}$ denote the subgraph of
$\Gamma^{\ast}_{n+k}$ induced by the union of arcs of all directed
paths from $Q_{i}^{(n+i)}$ to $Q_{n+j}^{(0)} ~(0\leq i,j\leq k)$,
where vertices $Q_{0}^{(n)},Q_{1}^{(n+1)},\ldots,Q_{k}^{(n+k)}$
(resp. $Q_{n}^{(0)}, Q_{n+1}^{(0)},\ldots,Q_{n+k}^{(0)}$) are viewed
as the source vertices (resp. the sink vertices) of
$\Gamma^{\diamond}_k$.
\end{definition}

For the weighted digraph $\Gamma^{\diamond}_k$ in Definition \ref{def+gamma+GG},
we denote by $B^\diamond(k)$ its walk matrix. In what follows
we will present a relation between $B^\diamond(k)$
and the Toeplitz matrix $\mathcal{T}$ of the $n$th row sequence of $M$.

\begin{prop}\label{prop+Tk+g}
For the Toeplitz matrix $\mathcal{T}$ of the $n$th row sequence of
$M$, its $k$th order leading principal minor equals the walk matrix
$B^{\diamond}(k)$ for the weighted digraph $\Gamma^{\diamond}_k$.
\end{prop}
\textbf{Proof:} Note $\mathcal{T}_{i,j}=M_{n,i-j}$. It suffices to
prove that
$$ M_{n,i-j}=B_{i,j}^{\diamond},$$ where the
vertex $Q_{i}^{(n+i)}$ (resp. $Q_{n+j}^{(0)}$) of
${V}^{\diamond}_{k}$ is used as the row (resp. column) index $i$
(resp. $j$). This immediately follows from the next two claims.

\begin{cl}\label{TClaim 1}
For $0\leq j\leq i\leq n$, let $B_{n,i-j}^{\ast}$ be the sum of the weights of walks from
$Q_{i}^{(n+i)}$ to $Q_{n+j}^{(0)}$ in ${V}^{\ast}_{n+i}$. Then
$$M_{n,i-j}=B_{n,i-j}^{\ast}.$$
\end{cl}
\textbf{Proof:} In fact, the vertex $Q_{n-i}^{(n)}$ (resp. $Q_{n-j}^{(0)}$) of ${V}^{\ast}_{n}$
is regarded as the row (resp. column) index $i$ (resp. $j$) of the walk matrix $B^{\ast}(n)$, then
$B^{\ast}_{i,j}$ is the total weight of walks from $Q_{n-i}^{(n)}$ to $Q_{n-j}^{(0)}$.
After taking $n-i \rightarrow i$ and $i-j\rightarrow j$, we
obtain that $B^{\ast}_{n-i,i-j}$ is the total weight of walks from $Q_{i}^{(n)}$ to $Q_{n-i+j}^{(0)}$.
By Proposition \ref{prop+GG+M}, we have
\begin{equation}\label{TY}
  M_{n-i,i-j}=B_{n-i,i-j}^{\ast}.
\end{equation}
Taking $n+i \rightarrow n$ in (\ref{TY}) yields Claim \ref{TClaim
1}.

\begin{cl}\label{TClaim 2}
For $0\leq j\leq i\leq n$, we have
$$B_{n,i-j}^{\ast}=B_{i,j}^{\diamond}.$$
\end{cl}

\textbf{Proof:} In fact, it is easy to see that the set of walks from
$Q_{i}^{(n+i)}$ to $Q_{n+j}^{(0)}$ in $\Gamma^{\diamond}_{k}$ is the same
as that in $\Gamma^{\ast}_{n+i}$, which is Claim \ref{TClaim 2}.

This completes the  proof of Proposition \ref{prop+Tk+g}.\qed

Finally, if $b_n=\gamma$ for $n\geq0$, using the weighted digraph
$\Gamma^{\ast}_n$, we shall construct a weighted digraph for the
Toeplitz matrix of the sequence $(M_{n+\delta i,k+\sigma i})_{i}$
with $\sigma>\delta$.

\begin{definition}\label{def+gamma+GG1}
Define a weighted digraph, denoted by
$\Gamma^{\circ}_m=(V^{\circ}_m, \vec{E}^{\circ}_m)$, as
follows: Let $\Gamma^{\circ}_{m}$ denote the subgraph of
$\Gamma^{\ast}_{n+m\sigma}$ induced by the union of arcs of all directed
paths from $R_{(\sigma-\delta)i}^{(n+i\sigma)}$ to
$R_{n-k+(\sigma-\delta)j}^{(n-k+j\sigma)} ~(0\leq i,j\leq m)$,
where vertices $R_{0}^{(n)},R_{\sigma-\delta}^{(n+\sigma)},
\ldots,R_{(\sigma-\delta)m}^{(n+m\sigma)}$ (resp. $R_{n-k}^{(n-k)},
R_{n-k+\sigma-\delta}^{(n-k+\sigma)},\ldots,R_{n-k+(\sigma-\delta)m}^{(n-k+m\sigma)}$) are viewed as the source vertices
(resp. the sink vertices) of $\Gamma^{\circ}_m$.
\end{definition}

We denote by $B^\circ(m)$ the walk matrix of $\Gamma^{\circ}_m$. The
following consequence indicates the relation between $B^\circ(m)$
and the Toeplitz matrix of the sequence $(M_{n+\delta i,k+\sigma i})_{i}$
with $\sigma>\delta$.

\begin{prop}\label{prop+Tm+the sequence+Toeplitz}
For the Toeplitz matrix of the sequence $(M_{n+\delta i,k+\sigma
i})_{i}$ with $\sigma>\delta$, its $m$th order leading principal
minor equals the walk matrix $B^{\circ}(m)$ for $\Gamma^{\circ}_m$.
\end{prop}

\textbf{Proof:} The $(i,j)$ element of the Toeplitz matrix of of the
sequence $(M_{n+\delta i,k+\sigma i})_{i}$ equals $M_{n+(i-j)\delta,k+(i-j)\sigma}$.
So taking the vertex $R_{(\sigma-\delta)i}^{(n+i\sigma)}$ in
$V^{\circ}_m$ as the row index $i$ and $R_{n-k+(\sigma-\delta)j}^{(n-k+j\sigma)}$
as the column index $j$ of the walk matrix $B^{\circ}(m)$, we just need to prove
\begin{eqnarray}\label{eq+Toeplitz}
B^{\circ}_{i,j}=M_{n+(i-j)\delta,k+(i-j)\sigma}.
\end{eqnarray}
Before do it, we will need some auxiliary claims.

\begin{cl}\label{Claim 1}
Choose vertices $R^{(n)}_n,R^{(n)}_{n-1},\ldots,R^{(n)}_0$ (resp.
$R_{n}^{(n)},R_{n-1}^{(n-1)},\ldots,R_{0}^{(0)}$) as the source
vertices (resp. the sink vertices) in $\Gamma^{\ast}_n$. Then the
walk matrix for the new weighted  digraph is also equal to $M_n$.
\end{cl}

\textbf{Proof:} Let $P(u,v)$ denote the sum of the weights of all
directed paths from $u$ to $v$. By Definition
\ref{definition+star+g}, vertices
$Q^{(n)}_n,Q^{(n)}_{n-1},\ldots,Q^{(n)}_0$ (resp.
$Q_{n}^{(0)},Q_{n-1}^{(0)},\ldots,Q_{0}^{(0)}$) are viewed as the
source vertices (resp. the sink vertices) of $\Gamma^{\ast}_n$. So
$B^{*}_{i,j}=P(Q^{(n)}_{n-i},Q^{(0)}_{n-j})$. It follows from
Proposition \ref{prop+GG+M} that
\begin{eqnarray}\label{eq+Tm5}
P(Q^{(n)}_{n-i},Q^{(0)}_{n-j})=M_{i,j}.
\end{eqnarray}
The directed paths from $Q^{(n)}_{n-i}$ to $Q^{(0)}_{n-j}$
can be divided into three segments by the two points $R^{(n)}_{n-i}$ and $R^{(n-j)}_{n-j}$.
In fact, there is exactly one directed path from $Q_{n-i}^{(n)}$ (resp. $R^{(n-j)}_{n-j}$)
to $R_{n-i}^{(n)}$ (resp. $Q^{(0)}_{n-j}$) for $0\leq i\leq n$ (resp. $0\leq j\leq n$)
and the weight of all the arcs in this path are $1$, then we obtain
\begin{eqnarray}\label{eq+Tm6}
P(Q^{(n)}_{n-i},Q^{(0)}_{n-j})=P(R^{(n)}_{n-i},R^{(n-j)}_{n-j}).
\end{eqnarray}
By (\ref{eq+Tm5}) and (\ref{eq+Tm6}), we have
$$P(R^{(n)}_{n-i},R^{(n-j)}_{n-j})=M_{i,j}.$$
This shows that Claim \ref{Claim 1} holds.\qed

\begin{cl}\label{Claim 2}
For the weighted digraph $\Gamma^{\ast}_{n}$,
 if $0\leq a\leq b \leq c\leq d \leq n$, then we have
 $$P(R^{(d)}_{a},R^{(c)}_{b})=P(R^{(b+d-c)}_{a},R^{(b)}_{b}).$$
 \end{cl}

\textbf{Proof:} From the recursive construction of the weighted
digraph $\Gamma^{\ast}_{n}$, it is easy to obtain that (1) vertices
$R_a^{(d)}$ and $R_a^{(b+d-c)}$ (resp., $R_b^{(c)}$ and $R_a^{(b)}$)
locate in the same horizontal line; (2) there are $b-a$ layers
between vertices $R_a^{(c)}$ and $R_b^{(d)}$; (3) the connected
components between vertices $R_a^{(d)}$ and $R_b^{(c)}$ consisting
of those edges with weight $1$ are the same as those of between
vertices $R_a^{(b+d-c)}$ and $R_b^{(b)}$.

For all directed paths from $R^{(d)}_{a}$ to $R^{(c)}_{b}$, shrink
their sub-paths with consistent weights $1$ into an edge with weight
$1$. Do the same operations for all directed paths from
$R^{(b+d-c)}_{a}$ to $R^{(b)}_{b}$. In consequence, we have the same
set of walks from $R^{(d)}_{a}$ to $R^{(c)}_{b}$ and from
$R^{(b+d-c)}_{a}$ to $R^{(b)}_{b}$, i.e.,
$$P(R^{(d)}_{a},R^{(c)}_{b})=P(R^{(b+d-c)}_{a},R^{(b)}_{b}).$$
For example, for $t=1$, if $a=0,b=1,c=2,d=n=3$, then in Figure~17,
replacing the sub-paths with consistent weights $1$ on the right
side of point $R^{(3)}_{0}$ (resp. $R^{(3)}_{1}$ and $R^{(2)}_{0}$)
with one edge with weight $1$. Then it is easy to find that the
sub-digraph labeled blue consisting of paths from point
$R^{(3)}_{0}$ to point $R^{(2)}_{1}$ is the same as that labeled red
consisting of paths from point $R^{(2)}_{0}$ to point $R^{(1)}_{1}$.
\qed

  \begin{center}
  \setlength{\unitlength}{0.6cm}
  \begin{picture}(-7,8.5)(-1.2,-5)

 \thicklines\put(-14.5,2){\line(1,0){3}}\thicklines\put(-11.5,2){\line(1,0){1.5}}  \thicklines\put(-9,2) {\line(1,0){1.5}}
 \thicklines\put(-16,1) {\color{blue}{\line(1,0){4.5}}}\thicklines\put(-11.5,1) {\color{blue}{\line(1,0){1.5}}}  {\color{blue}\thicklines\put(-9,1){\line(1,0){1.5}}}
 \thicklines\put(-17.5,0){\color{blue}{\line(1,0){6}}}\thicklines\put(-11.5,0){\color{blue}{\line(1,0){1.5}}} {\color{blue}\thicklines\put(-9,0){\line(1,0){1.5}}}

\thicklines\put(-7.5,2){\line(1,0){1.5}} \thicklines\put(-6,2) {\line(1,0){1.5}}
  {\color{blue} \thicklines\put(-7.5,1){\line(1,0){1.5}}}\thicklines\put(-6,1){\color{red}{\line(1,0){1.5}}}
  \thicklines\put(-7.5,0){\color{red}{\line(1,0){1.5}}}\thicklines\put(-6,0){\color{red}{\line(1,0){1.5}}}

  \thicklines\put(-17.5,0){\color{blue}{\line(3,2){1.5}}}\thicklines\put(-16,1){\line(3,2){1.5}}\thicklines\put(-14.5,2){\line(3,2){1.5}}
 \thicklines\put(-13,0){\color{blue}{\line(3,2){1.5}}}\thicklines\put(-13,1){\line(3,2){1.5}}\thicklines\put(-11.5,0){\color{blue}{\line(3,2){1.5}}}
 \thicklines\put(-11.5,1){\line(3,2){1.5}} \thicklines\put(-9,1){\line(3,2){1.5}}\thicklines\put(-9,0){\color{blue}{\line(3,2){1.5}}}

 \thicklines\put(-4.5,1){\color{red}{\line(1,0){1.5}}}\thicklines\put(-3,1){\color{red}{\line(1,0){1.5}}}
 \thicklines\put(-0.5,1){\color{red}{\line(1,0){1.5}}}
  \thicklines\put(-4.5,0){\color{red}{\line(1,0){1.5}}}\thicklines\put(-3,0){\color{red}{\line(1,0){1.5}}}\thicklines\put(-0.5,0){\color{red}{\line(1,0){1.5}}}

 {\color{red}\thicklines\put(1,1){\line(1,0){1.5}}}
\thicklines\put(1,0){\line(1,0){1.5}} \thicklines\put(2.5,0){\line(1,0){1.5}}\thicklines\put(4,0){\line(1,0){1.5}}

  \thicklines\put(6.5,0){\line(1,0){1.5}}  \thicklines\put(1,0){\color{red}{\line(3,2){1.5}}}
  \thicklines\put(-7.5,-0.02){\color{red}{\line(3,2){1.5}} }\thicklines\put(-7.5,0.03){\color{blue}{\line(3,2){1.5}} } \thicklines\put(-6,1){\line(3,2){1.5}}\thicklines\put(-4.5,0){\color{red}{\line(3,2){1.5}}}
   \thicklines\put(-3,0){\color{red}{\line(3,2){1.5}}}\thicklines\put(-0.5,0){\color{red}{\line(3,2){1.5}}}

                              \put(-16,1){\circle*{0.2}}
  \put(-17.5,0){\circle*{0.2}}\put(-16,0){\circle*{0.2}}

                             \put(-13,3){\circle*{0.2}}
  \put(-14.5,2){\circle*{0.2}}\put(-13,2){\circle*{0.2}}\put(-11.5,2){\circle*{0.2}}\put(-10,2){\circle*{0.2}}\put(-9,2){\circle*{0.2}}
  \put(-14.5,1){\circle*{0.2}}\put(-13,1){\circle*{0.2}}\put(-11.5,1){\circle*{0.2}}\put(-10,1){\circle*{0.2}}\put(-9,1){\circle*{0.2}}
  \put(-14.5,0){\circle*{0.2}}\put(-13,0){\circle*{0.2}}\put(-11.5,0){\circle*{0.2}}\put(-10,0){\circle*{0.2}}\put(-9,0){\circle*{0.2}}

                                                                                  \put(-4.5,2){\circle*{0.2}}\put(-6,2){\circle*{0.2}}
  \put(-0.5,1){\circle*{0.2}}\put(-1.5,1){\circle*{0.2}}\put(-3,1){\circle*{0.2}}\put(-4.5,1){\circle*{0.2}}\put(-6,1){\circle*{0.2}}
  \put(-0.5,0){\circle*{0.2}}\put(-1.5,0){\circle*{0.2}}\put(-3,0){\circle*{0.2}}\put(-4.5,0){\circle*{0.2}}\put(-6,0){\circle*{0.2}}

  \put(1,1){\circle*{0.2}} \put(2.5,1){\circle*{0.2}}
  \put(1,0){\circle*{0.2}} \put(2.5,0){\circle*{0.2}} \put(4,0){\circle*{0.2}}\put(5.5,0){\circle*{0.2}}\put(6.5,0){\circle*{0.2}}
  \put(-1.2,0){\circle*{0.1}}\put(-1,0){\circle*{0.1}}\put(-0.8,0){\circle*{0.1}}
\put(-1.2,1){\circle*{0.1}}\put(-1,1){\circle*{0.1}}\put(-0.8,1){\circle*{0.1}}

  \put(-9.3,0){\circle*{0.1}}\put(-9.5,0){\circle*{0.1}}\put(-9.7,0){\circle*{0.1}}
  \put(-9.3,1){\circle*{0.1}}\put(-9.5,1){\circle*{0.1}}\put(-9.7,1){\circle*{0.1}}
  \put(-9.3,2){\circle*{0.1}}\put(-9.5,2){\circle*{0.1}}\put(-9.7,2){\circle*{0.1}}
  \put(5.8,0){\circle*{0.1}}\put(6,0){\circle*{0.1}}\put(6.2,0){\circle*{0.1}}

  \put(-7.5,0){\circle*{0.2}}\put(8,0){\circle*{0.2}}
  \put(-7.5,1){\circle*{0.2}}
  \put(-7.5,2){\circle*{0.2}}
  \tiny
\put(7.7,-0.6){$R_0^{(0)}$}
  \put(1.1,-0.5){$R_0^{(1)}$}  \put(2.2,0.5){$R_1^{(1)}$}
  \put(-7.5,-0.5){$R_0^{(2)}$} \put(-6.2,0.5){$R_1^{(2)}$}\put(-4.8,1.5){$R_2^{(2)}$}
  \put(-17.7,-0.5){$R_0^{(3)}$} \put(-16.2,0.5){$R_1^{(3)}$}\put(-14.7,1.5){$R_2^{(3)}$} \put(-13.2,2.5){$R_3^{(3)}$}
\put(1.4,0.6){$\gamma$}
\put(-5.7,1.6){$\gamma$}   \put(-7.2,0.6){$\gamma$}
\put(-17.2,0.6){$\gamma$}   \put(-15.7,1.6){$\gamma$}  \put(-14.2,2.6){$\gamma$}
\put(-4.4,0.6){$\alpha_1$}\put(-2.9,0.6){$\alpha_2$}\put(-0.4,0.6){$\alpha_{\ell}$}
\put(-12.9,0.6){$\alpha_1$}\put(-11.4,0.6){$\alpha_2$}\put(-8.9,0.6){$\alpha_{\ell}$}
\put(-12.9,1.6){$\alpha_1$}\put(-11.4,1.6){$\alpha_2$}\put(-8.9,1.6){$\alpha_{\ell}$}

\put(-4,-0.35){$\beta_1$}\put(-2.5,-0.35){$\beta_2$}\put(0,-0.35){$\beta_{\ell}$}
\put(-4,1.15){$\beta_1$}\put(-2.5,1.15){$\beta_2$}\put(0,1.15){$\beta_{\ell}$}
\put(3,-0.35){$\beta_1$}\put(4.5,-0.35){$\beta_2$}\put(7,-0.35){$\beta_{\ell}$}
\put(-12.3,1.15){${\beta}_1$}\put(-10.8,1.15){$\beta_2$}\put(-8.3,1.15){$\beta_{\ell}$}
\put(-12.3,-0.35){$\beta_1$}\put(-10.8,-0.35){$\beta_2$}\put(-8.3,-0.35){$\beta_{\ell}$}
\put(-12.3,2.15){$\beta_1$}\put(-10.8,2.15){$\beta_2$}\put(-8.3,2.15){$\beta_{\ell}$}
  \normalsize
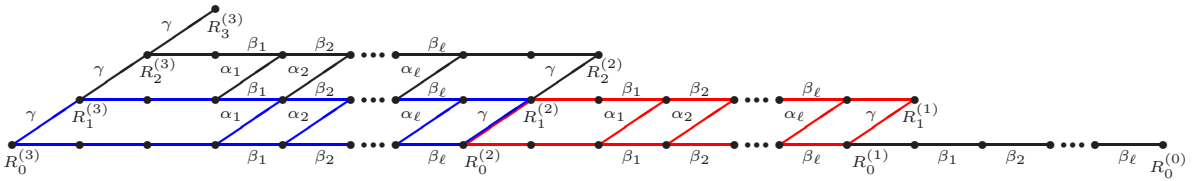
  \put(-8.4,-1.8){Figure~17:  $\Gamma^{\ast}_{3}$ omitting weight $1$ for $t=1$.}
  \end{picture}
  \end{center}

 \begin{cl}\label{Claim 3}
For the weighted digraph $\Gamma^{\ast}_{n}$, if $0\leq a\leq d \leq
n$ and
 $0\leq b \leq c\leq n$, then
  \begin{eqnarray}\label{eq+Tm}
 P(R^{(d)}_{a},R^{(c)}_{b})=M_{b+d-a-c,d-c}.
\end{eqnarray}
\end{cl}

\textbf{Proof:} If $a>b$ or $c>d$, both sides of the equality
(\ref{eq+Tm}) are $0$. So (\ref{eq+Tm}) clearly holds. In
consequence, we can assume $0\leq a \leq b \leq c \leq d \leq n$. By
Claim \ref{Claim 2}, we have
 \begin{eqnarray}\label{eq+Tm1}
P(R^{(d)}_{a},R^{(c)}_{b})=P(R^{(b+d-c)}_{a},R^{(b)}_{b}).
 \end{eqnarray}
It is easy to observe that $P(R^{(b+d-c)}_{a},R^{(b)}_{b})$ in
$\Gamma^{\ast}_{n}$ is the same as that in $\Gamma^{\ast}_{b+d-c}$.
Now we calculate $P(R^{(b+d-c)}_{a},R^{(b)}_{b})$ in
$\Gamma^{\ast}_{b+d-c}$. By Claim \ref{Claim 1}, after we choose
$R^{(b+d-c)}_{b+d-c},R^{(b+d-c)}_{b+d-c-1}, \ldots,$
$R^{(b+d-c)}_1,R^{(b+d-c)}_0$ (resp. $R_{b+d-c}^{(b+d-c)}$,
$R_{b+d-c-1}^{(b+d-c-1)},\ldots,R_{1}^{(1)},R_{0}^{(0)}$) as the
source vertices (resp. the sink vertices) in
$\Gamma^{\ast}_{b+d-c}$, we have
\begin{eqnarray}\label{eq+Tm2-}
P(R^{(b+d-c)}_{b+d-c-i},R^{(b+d-c-j)}_{b+d-c-j})=M_{i,j}.
\end{eqnarray}
Taking $i=b+d-a-c$ and $j=d-c$ in (\ref{eq+Tm2-}) yields
\begin{eqnarray}\label{eq+Tm2}
 P(R^{(b+d-c)}_{a},R^{(b)}_{b})=M_{b+d-a-c,d-c}.
\end{eqnarray}
Thus, combining (\ref{eq+Tm1}) and (\ref{eq+Tm2}) gives
(\ref{eq+Tm}).\qed

Finally, we will use (\ref{eq+Tm}) to prove (\ref{eq+Toeplitz}).
Replacing $n$ with $n+m\sigma$, we have equality (\ref{eq+Tm}) holds
in $\Gamma^{\ast}_{n+m\sigma}$. Taking
$a=(\sigma-\delta)i,b=n-k+(\sigma-\delta)j,c=n-k+j\sigma$ and
$d=n+i\sigma$ in (\ref{eq+Tm}), we derive
\begin{eqnarray}\label{eq+Tm3}
P(R_{(\sigma-\delta)i}^{(n+i\sigma)},R_{n-k+(\sigma-\delta)j}^{(n-k+j\sigma)})=M_{n+(i-j)\delta,k+(i-j)\sigma}.
\end{eqnarray}
From Definition \ref{def+gamma+GG}, we find
\begin{eqnarray}\label{eq+Tm4}
B^{\circ}_{i,j}=P(R_{(\sigma-\delta)i}^{(n+i\sigma)},R_{n-k+(\sigma-\delta)j}^{(n-k+j\sigma)}).
\end{eqnarray}
In consequence, combining (\ref{eq+Tm3}) and (\ref{eq+Tm4}), we obtain (\ref{eq+Toeplitz}).

This completes the proof of Proposition \ref{prop+Tm+the sequence+Toeplitz}.\qed

\textbf{Now we are in a position to present the proof of Theorem
\ref{thm+fun+ell} (i) and (iii):}

For (i), by Corollary \ref{cor.gessel-viennot.totpos}, the
$(\balpha,\bbeta,\textbf{y})$-total positivity of the matrix $M$ is
immediate from the following two points:

  (1) Fix $n\geq0$. For any $m\times m$
  minor $M_{I,J}$ of the matrix $M_{n}$, where $I=(i_1,\ldots,i_m)$ with
  $0\leq i_1< \cdots < i_m\leq n$ and $J=(j_1,\ldots,j_m)$ with
  $0\leq j_1<\cdots< j_m\leq n$, we will construct the corresponding
  weighted digraph. Let ${\Gamma^{\ast}_{I,J}}$ be the subgraph of
  $\Gamma^{\ast}_{n}$ induced by the union of arcs of all directed paths from
  $Q_{n-i_s}^{(n)}$ to $Q_{n-j_k}^{(0)}$ $(1\leq s,k\leq m)$. By Proposition \ref{prop+GG+M},
  we immediately have
  $$B_{I,J}^{\ast}=M_{I,J}.$$

  (2) It is obvious that
  $\left(Q_{n-i_1}^{(n)},Q_{n-i_2}^{(n)},\ldots,Q_{n-i_m}^{(n)}\right)$ and
  $\left(Q_{n-j_1}^{(0)},Q_{n-j_2}^{(0)},\ldots,Q_{n-j_m}^{(0)}\right)$ form a
  fully nonpermutable pair.

Similarly, the $(\balpha,\bbeta,\textbf{y})$-total positivity of
 the Toeplitz matrix $\mathcal{T}$ of the $n$th row sequence of $M$ can also be established
by the following two points:

(1) Fix $k\geq0$. For any $m\times m$ minor $\mathcal{T}_{I,J}$ of
$\mathcal{T}_{k}$, where
  $I=(i_1,\ldots,i_m)$ with $0\leq i_1< \cdots < i_m\leq k$
  and $J=(j_1,\ldots,j_m)$ with $0\leq j_1<\cdots< j_m\leq k$, we will
  construct the corresponding weighted digraph.
  Let $\Gamma^{\diamond}_{I,J}$ be the subgraph of
  $\Gamma^{\diamond}_{k}$ induced by the union of arcs of all directed paths from
  $Q_{i_s}^{(n+i_s)}$ to $Q_{n+j_r}^{(0)}$ $(1\leq s,r\leq m)$. By
  Proposition \ref{prop+Tk+g}, we immediately have
  $${B_{I,J}^{\diamond}}=\mathcal{T}_{I,J}.$$

(2) By Lemma \ref{lemma.nonpermutable},
  $(Q_{i_1}^{(n+i_1)},Q_{i_2}^{(n+i_2)},\ldots,Q_{i_m}^{(n+i_m)})$ and
  $(Q_{n+j_1}^{(0)},Q_{n+j_2}^{(0)},\ldots,Q_{n+j_m}^{(0)})$ form a fully nonpermutable pair.

This proof of (i) is complete.

 Finally, for (iii), we will also apply Corollary \ref{cor.gessel-viennot.totpos}
 to establish the $(\balpha,\bbeta,\gamma)$-total positivity of
 the Toeplitz matrix $\textsl{T}$ of the sequence $(M_{n+\delta i,k+\sigma i})_{i}$
with $\sigma>\delta$ in terms of the following two points:

  (1) Fix $d\geq0$. For any $d\times d$ minor $\textsl{T}_{I,J}$ of $\textsl{T}_{m}$, where
  $I=(i_1,\ldots,i_d)$ with $0\leq i_1< \cdots < i_d\leq m$
  and $J=(j_1,\ldots,j_d)$ with $0\leq j_1<\cdots< j_d\leq m$, we will
  construct the corresponding weighted digraph.
  Let $\Gamma^{\circ}_{I,J}$ be the subgraph of
  $\Gamma^{\circ}_{m}$ induced by the union of arcs of all directed paths from
  $R_{(\sigma-\delta)i_s}^{(n+i_s\sigma)}$ to
  $R_{n-k+(\sigma-\delta)j_r}^{(n-k+j_r\sigma)}$ $(1\leq s,r\leq d)$. By
  Proposition \ref{prop+Tm+the sequence+Toeplitz}, we immediately have
  $${B_{I,J}^{\circ}}=\textsl{T}_{I,J}.$$

  (2) It is obvious that
  $(R_{(\sigma-\delta)i_1}^{(n+i_1\sigma)},R_{(\sigma-\delta)i_2}^{(n+i_2\sigma)},\ldots,R_{(\sigma-\delta)i_d}^{(n+i_d\sigma)})$ and
  $(R_{n-k+(\sigma-\delta)j_1}^{(n-k+j_1\sigma)},R_{n-k+(\sigma-\delta)j_2}^{(n-k+j_2\sigma)},\ldots,$ $R_{n-k+(\sigma-\delta)j_d}^{(n-k+j_d\sigma)})$  form a fully nonpermutable pair.

This completes the proof of Theorem \ref{thm+fun+ell}.\qed

\section{Total positivity of the transpose of $M$ }
Note that the transpose of a totally positive matrix is also totally
positive. For the transpose of $M$, we have the similar total
positivity. By the recurrence relation (\ref{rec+1+M}), it is easy
to obtain the transpose of $M$ satisfying the following recurrence
relation.
\begin{definition}\label{transpose}
Define a $k$-recursive matrix $T=[T_{n,k}]_{n,k}$ satisfying the
recurrence relation: for $n\geq t$,
\begin{equation}\label{rec+1+T}
T_{n,k}=a^{(0)}_k(\textbf{x})T_{n-1,k-t}+a^{(1)}_k(\textbf{x})T_{n-1,k-t-1}+\ldots+a^{(\ell)}_k(\textbf{x})T_{n-1,k-t-\ell}+b_k(\textbf{y})T_{n,k-1},\\
\end{equation}
where initial conditions $T_{0,0}=1$, $T_{n,k}=0$ for $n<0$ or $k<0$, and
$$T_{n,k} =\begin{cases}
                                (a^{(0)}_0)^{k}      & \text{for} ~t=0,~n\geq1,~k=0;\\
                                \prod\limits_{i=0}^{n-1}b_{n-i}      & \text{for}~t\geq2,~n=0,~1\leq k \leq t-1;\\
                                 0                                    & \text{for}~t\geq1,~n\geq1,~0\leq k \leq t-1.
         \end{cases}$$
\end{definition}

It follows from the results for $M$ that we immediately have the
following results.
\begin{thm}\label{thm+main+transpose}
Let $T$ be the matrix in Definition \ref{transpose}. Then we have
  \begin{itemize}
       \item [\rm (i)]
if $\mathcal{A}$ is $\textbf{x}$-totally positive of order $r$, then
$T$ is $(\textbf{x},\textbf{y})$-totally positive of order $r$;
         \item [\rm (ii)]
if $\mathcal{A}$ is $\textbf{x}$-totally positive, then $T$ is $(\textbf{x},\textbf{y})$-totally positive.
 \end{itemize}
\end{thm}

\begin{thm}\label{thm+bidiag+transpose}
Let $T$ be the matrix in Definition \ref{transpose} with $\ell=1$.
Then
\begin{itemize}
\item [\rm (i)]
the matrix $T$ is $(\textbf{x},\textbf{y})$-totally positive;
\item [\rm (ii)]
the Toeplitz matrix of the each column sequence of $T$ with $t\geq 1$ is
$(\textbf{x},\textbf{y})$-totally positive.
\end{itemize}
\end{thm}

\begin{thm}\label{thm+tridiag+transpose}
Let $T$ be the matrix in Definition \ref{transpose} with $\ell=2$.
 Suppose that there exists polynomial sequences  with
 nonnegative coefficients $(\alpha_k(\textbf{x}))_k$, $(\beta_k(\textbf{x}))_k$,
$(\lambda_k(\textbf{x}))_k$ and $(\mu_k(\textbf{x}))_k$ such that
\begin{itemize}
\item [\rm (i)]
   $a_k^{(0)}=\alpha_k\lambda_k$ for $k\geq0$;
\item [\rm (ii)]
   $a_k^{(1)}=\alpha_k\mu_k+\beta_k\lambda_{k-1}$ for $k\geq1$;
\item [\rm (iii)]
   $a_k^{(2)}=\beta_k\mu_{k-1}$ for $k\geq2$,
\end{itemize}
then \begin{itemize}
\item [\rm (i)]
 the matrix $T$ is $(\textbf{x},\textbf{y})$-totally positive;
\item [\rm (ii)]
  the Toeplitz matrix of the each column sequence of $T$ with $t\geq 1$ is
$(\textbf{x},\textbf{y})$-totally positive.
\end{itemize}
\end{thm}

\begin{thm}\label{thm+fun+ell+transpose}
Let $T$ be the matrix in Definition \ref{transpose} and
$(b_i(\textbf{y}))_{i\geq0}$ be a polynomial sequence with
nonnegative coefficients in $\textbf{y}$. If there exists
indeterminates $\balpha=(\alpha_i)_{i=1}^{\ell}$ and
$\bbeta=(\beta_i)_{i=1}^{\ell}$ such that the generating function
$$\sum_{i= 0}^{\ell}a^{(i)}_kz^i
=\prod_{j= 1}^{\ell}(\alpha_jz+\beta_j)$$ for all $k\in \mathbb{N}$,
then we have the following results:
  \begin{itemize}
\item [\rm (i)]
the matrix $T$ and the Toeplitz matrix of the each column sequence of $T$ for
$t\geq1$ are $(\balpha,\bbeta,\textbf{y})$-totally positive;
\item [\rm (ii)]
if $b_k=\gamma$ for $k\geq0$, then the Toeplitz matrix of the each
row sequence of $T$ is $(\balpha,\bbeta,\gamma)$-totally positive;
\item [\rm (iii)]
if $b_k=\gamma$ for $k\geq0$ and $t\geq1$, then the Toeplitz matrix
of the sequence $(T_{n+\sigma i,k+\delta i})_{i}$ is
$(\balpha,\bbeta,\gamma)$-totally positive for $0\leq n\leq k$,
$0<\delta$ and $max\{n,\delta\}<\sigma$;
\item [\rm (iv)]
if $b_k=\gamma$ for $k\geq0$, then
$$
   T_{n,k}=
    \!\!\!
   \sum_{\begin{scarray}
            i\geq 0  \\[-1mm]
            0\leq c_1,c_2,\ldots,c_{\ell}\leq n  \\[-1mm]
            i+\sum\limits_{j=1}^{\ell}c_j=k-tn \\[-1mm]
         \end{scarray}
        }
   \!\!\!
   \binom{n+i}{i}{\gamma}^i\prod_{j= 1}^{\ell}\binom{n}{c_j}{\alpha_j}^{c_j}{\beta_j}^{n-c_j};$$
\item [\rm (v)]
if $b_k=\gamma$ for $k\geq0$, then we have
$$\sum_{n,k\geq0}T_{n,k}q^nz^k=\frac{1}{1-\gamma z-q\sum_{i= 0}^{\ell}a^{(i)}_kz^{t+i}}.$$
 \end{itemize}
\end{thm}

\section{Applications}
In this section, we will present many applications of our results.

\subsection{The Delannoy-like triangle}
In \cite{MZ16}, Mu and Zheng studied the Delannoy-like triangle
$D(e, h) = [\widetilde{d}_{n,k}]_{n,k}$ defined by the recurrence
relation
$$\widetilde{d}_{0,0} = \widetilde{d}_{1,0} = \widetilde{d}_{1,1}=1, \widetilde{d}_{n,k} = \widetilde{d}_{n-1,k-1}  + h\,\widetilde{d}_{n-2,k-1}+ e\,\widetilde{d}_{n-1,k},$$
where both $e$ and $h$ are nonnegative and $\widetilde{d}_{n,k} = 0$
unless $n\geq k \geq0$.
Some well-known combinatorial triangles are such matrices, including
the Pascal triangle $D(1, 0)$ \cite[A007318]{Slo}, the Fibonacci
triangle $D(0, 1)$ \cite[A026729]{Slo}, and the Delannoy triangle
$D(1, 1)$. Furthermore, the Schr\"{o}der triangle
\cite[A132372]{Slo} and Catalan triangle \cite[A033184]{Slo} also
arise as signless inverses of Delannoy-like triangles.

Mu and Zheng \cite[Theorem 1]{MZ16} used the Riordan array method to
prove the total positivity of the Delannoy-like triangle $D(e, h)$.
In addition, they \cite[Theorem 3]{MZ16} obtained that each row of
Delannoy-like triangles is a P\'olya frequency sequence by showing
that the row generating function
$\sum_{k\geq0}\widetilde{d}_{n,k}x^k$ has only real zeros. Taking
$t=\ell=a_n^{(0)}=1$ and $a_n^{(1)}=h$ for $n\geq0$, and $b_1=1$ and
$b_n=e$ for $n\geq2$ in Proposition \ref{prop+rec+1+M}, we find that
$D(e, h)$ is a special case of $M$. In consequence, these total
positivity results are immediate from Theorem \ref{thm+fun+ell} (i)
with $t=\ell=\beta_1=1$ and $\alpha_1=h$.

\subsection{The generalized Delannoy triangle}
Denote by $D_w(n, k)$ the total sum of the weights of all paths in
the plane from $(0,0)$ to $(k, n)$ with steps $(1,0)$, $(0,1)$,
$(1,1)$, and with nonnegative integer weights $a$, $b$, $c$,
respectively. Many properties of $D_w(n, k)$ have been studied in
the literature. From \cite {Raz02}, we know that
           \begin{equation}\label{formula+Del+w}
           D_w(n, k)=\sum\limits_{d\geq0}\binom{n+k-d}{k}\binom{k}{d}a^{k-d}b^{n-d}c^d
           \end{equation}
and
\begin{equation}\label{rec+Del+w}
D_w(n, k)=a\,D_w(n-1,k)+b\,D_w(n,k-1)+c\,D_w(n-1,k-1)
           \end{equation}
with the boundary conditions $D_w(n, 0) = a^n$ for $n\geq0$ and
$D_w(0,k) = b^k$ for $k\geq0$.

Let $d(n,k)=D_w(n-k, k)$. It follows from (\ref{rec+Del+w}) that we
have the following recurrence relation
$$d(n,k)=a\,d(n-1,k-1)+c\,d(n-2,k-1)+b\,d(n-1,k).$$
The infinite triangular matrix $D_w(a, b, c)= [d(n,k)]_{n,k}$ is
called {\it the generalized Delannoy triangle} and the square matrix
$D^{\ulcorner}_w(a, b, c)= [D_w(n,k)]_{n,k}$ is called {\it the
generalized Delannoy square}. In particular, $D^{\ulcorner}_w(1, 1,
0)$ is the Pascal square and $D^{\ulcorner}_w(1,1,1)$ is the
Delannoy square.

In general, many properties of the Pascal triangle and Delannoy
triangle can be extended to those of the generalized Delannoy
triangle. For example, similar to binomial coefficients, Razpet
\cite {Raz02} showed that $D_w(n, k)$ is one of the double sequences
satisfying the Lucas property:
$$D_w(\xi p+\eta,\lambda p+\mu)\equiv D_w(\xi,\lambda)D_w(\eta,\mu)~(mod\, p)$$
where $\xi, \eta, \lambda, \mu$ are nonnegative integers such that
 $0 \leq \eta < p, 0 \leq \mu< p$ for a prime number $p$. In \cite{CKS09},
 it was proved that $D_w(a, b, c)$ is the Riordan array
$\mathcal{R}\left(\frac{1}{1-bz},\frac{z(a+cz)}{1-bz}\right)$ and
the generating function $\varphi(z)$ for the row sums of the Riordan
array $D_w(a, b, c)$ is
$$ \varphi(z)=\frac{1}{1-(a+b)z-cz^2}.$$
It follows from (\ref{formula+Del+w}) that we also have
$$d_{n,k}=\sum\limits_{d\geq0}\binom{n-d}{k}\binom{k}{d}a^{k-d}b^{n-k-d}c^d.$$
For some other positivity properties of the generalized Delannoy
triangle can be stated in the following result, which are immediate
from Theorem \ref{thm+fun+ell} with $t=\ell=1,\alpha_1=c,\beta_1=a$
and $b_n=\gamma=b$.
\begin{prop}\label{Generalized Delannoy matrix}
Let $D_w(a, b, c)$ be the generalized Delannoy triangle. Then we
have
 \begin{itemize}
\item [\rm (i)]
the matrix $D_w(a, b, c)$ is totally positive;
\item [\rm (ii)]
the Toeplitz matrix of the each row sequence of $D_w(a, b, c)$ is
totally positive;
\item [\rm (iii)]
the Toeplitz matrix of the each column sequence of $D_w(a, b, c)$ is
totally positive;
\item [\rm (iv)]
 the Toeplitz matrix of the sequence $(d(n+\delta i,k+\sigma i))_{i}$
 is totally positive when $0\leq k\leq n$, $0<\delta$ and $max\{k,\delta\}<\sigma$;
 \item [\rm (v)]
the sequence $(d(n+\delta i,k+\sigma i))_{i}$
 is a P\'olya frequency sequence and infinitely log-concave when $0\leq k\leq n$, $0<\delta$ and $max\{k,\delta\}<\sigma$.
 \end{itemize}
\end{prop}

\subsection{The recursive matrix of Brenti}
For the Brenti's matrix $A=[A_{n,k}]_{n,k}$ in (\ref{rec+Brenti}),
Brenti \cite[Theorem 4.3]{Bre95} used a planar network method to
prove that the matrix $[A_{n,k}]_{n,k}$ and the Toeplitz matrix of
the each row sequence of $A$ are $(\textbf{x,y,z})$-totally
positive. By Proposition \ref{prop+rec+1+M} and Theorem
\ref{thm+bidiag} with $\ell=1,a_n^{(0)}=z_n,a_n^{(1)}=y_n,b_n=x_n$,
we immediately obtain:

\begin{cor}\emph{\cite[Theorem~4.3]{Bre95}}
The matrix $A$ and the Toeplitz matrix of the each row sequence of
$A$ with $t\geq 1$ are
 $(\textbf{x,y,z})$-totally positive.
\end{cor}

\begin{rem}
For our approaches of the proof, except that one is algebraic
method, the construction of the planar network of our combinatorial
proof is different from that of Brenti.
\end{rem}

\subsection{The signless Stirling triangle of the first kind}

Let $c_{n,k}$ be the signless Stirling number of the first kind,
i.e., the number of permutations of $[n]$ containing exactly $k$
cycles. In addition, signless Stirling numbers of the first kind
$c_{n,k}$ satisfy the recurrence relation
$$c(n,k)=c(n-1,k-1)+(n-1)c(n-1,k),$$
where $c(n,0)=c(0,k)=0$ except $c(0,0)=1$. The signless Stirling
triangle of the first kind $[c(n,k)]_{n,k}$ is formed by signless
Stirling numbers of the first kind. Its row-generating polynomial
has a precise formula $$C_n(q):=\sum_{k=0}^nc(n,k)q^k=q(q+1)\cdots
(q+n-1)$$ and obviously has only real zeros. Thus each row is a
P\'olya frequency sequence and infinitely log-concave. Signless
Stirling numbers of the first kind have many nice properties (see
Comtet \cite{Com74} for details). By Proposition \ref{prop+rec+1+M}
and Theorem \ref{thm+bidiag} with $\ell=t=1,a_n^{(0)}=1,a_n^{(1)}=0$
and $b_n=n-1$, we immediately obtain:

\begin{cor}\emph{\cite{Bre95}}
The signless Stirling triangle of the first kind $[c(n,k)]_{n,k}$
and the Toeplitz matrix of the each row sequence of $[c(n,k)]_{n,k}$ are
totally positive.
\end{cor}

\subsection{The Legendre-Stirling triangle of the first kind}

The Legendre-Stirling numbers of the first kind  ${Ps}_n^{(k)}$
satisfy the following recurrence relation:
$${Ps}_n^{(k)}={Ps}_{n-1}^{(k-1)}+n(n-1){Ps}_{n-1}^{(k)}~~(n,k\geq1)$$
with the initial conditions
${Ps}_n^{(0)}={Ps}_0^{(k)}=0,{Ps}_0^{(0)}=1$. Similar to the
combinatorial interpretation of Stirling numbers of the first kind,
${Ps}_n^{(k)}$ counts the number of Legendre-Stirling permutation
pairs $(\pi_1, \pi_2)$ of length $n$ in which $\pi_2$ has exactly
$k$ cycles \cite{Egg10}. The Legendre-Stirling triangle of the first
kind $[{Ps}_n^{(k)}]_{n,k}$ is formed by Legendre-Stirling numbers
of the first kind. From Andrews and Littlejohn \cite{AL09}, we know
that the row-generating polynomial of Legendre-Stirling triangle of
the first kind equals
$$\sum\limits_{k=0}^n{Ps}_n^{(k)}x^k=x(x+2)\cdots(x+n(n-1)).$$
Obviously, the row-generating polynomial has only real zeros and is
infinitely log-concave. For more details on Legendre-Stirling
numbers of the first kind ${Ps}_n^{(k)}$, we refer the reader to
\cite{AGL11,AL09,Egg10,ELW02}. By Proposition \ref{prop+rec+1+M} and
Theorem \ref{thm+bidiag} with $\ell=t=1,a_n^{(0)}=1,a_n^{(1)}=0$ and
$b_n=(n-1)n$, we immediately obtain:
\begin{cor}
The Legendre-Stirling triangle of the first kind
$[{Ps}_n^{(k)}]_{n,k}$ and the Toeplitz matrix of the each row sequence of
$[{Ps}_n^{(k)}]_{n,k}$  are totally positive.
\end{cor}

\subsection{The Jacobi-Stirling triangle of the first kind}
The Jacobi-Stirling numbers ${Jc}_n^{(k)}(z)$ of the first kind
introduced in \cite{EKLWY07} satisfy the following recurrence
relation:
$${Jc}_n^{(k)}(z)={Jc}_{n-1}^{(k-1)}(z)+(n-1)(n-1+z){Jc}_{n-1}^{(k)}(z)~~(n,k\geq1)$$
with the initial conditions
${Jc}_n^{(0)}(z)={Jc}_0^{(k)}(z)=0,{Jc}_0^{(0)}(z)=1$ and
$z=\alpha+\beta+1>-1$. For $\alpha=\beta=0$, the Jacobi-Stirling
numbers of the first kind reduce to Legendre-Stirling numbers of the
first kind. Mongelli \cite{Mon122} showed that the Jacobi-Stirling
numbers of the first kind are specializations of the elementary
symmetric functions by
$${Jc}_n^{(k)}(z)=e_{n-k}(1(1+z),2(2+z),\ldots,(n-1)(n-1+z)).$$
We call the triangle $[{Jc}_n^{(k)}(z)]_{n,k}$ {\it the
Jacobi-Stirling triangle of the first kind.} Its row-generating
polynomial is
$$\sum\limits_{k=0}^n{Jc}_n^{(k)}x^k=x(x+1+z)\cdots[x+(n-1+z)(n-1)].$$
Obviously, the row-generating polynomial has only real zeros, and is
therefore infinitely log-concave. Note a fact that a matrix is
totally positive if and only if its singless inverse (see \cite[pp.
6]{Pin10} for instance). In \cite[Theorem~5]{Mon12}, Mongelli
obtained the total positivity of $[{Jc}_n^{(k)}(z)]_{n,k}$ by using
a planar network method to prove that the singless inverse matrix of
$[{Jc}_n^{(k)}(z)]_{n,k}$ is totally positive. Mongelli
\cite[Proposition~3]{Mon12} also used generating functions to prove
that the Toeplitz matrix of the each row sequence of
$[{Jc}_n^{(k)}(z)]_{n,k}$ is totally positive. These total
positivity results are also immediate from Proposition
\ref{prop+rec+1+M} and Theorem \ref{thm+bidiag} with
$\ell=t=1,a_n^{(0)}=1,a_n^{(1)}=0,$ and $b_n=(n-1)(n-1+z)$. Finally,
we refer the reader to \cite{AEGL13,EKLWY07,GZ10,Mon12} for more
properties of the Jacobi-Stirling numbers of the first kind.

\subsection{The Stirling triangle of the second kind}
Let $S(n,k)$ denote the Stirling number of the second kind, which enumerates the number of
partitions of a set with $n$ elements consisting of $k$ disjoint nonempty sets. It is well-known
that the Stirling number of the second kind satisfies the recurrence relation
$$ S(n,k)=S(n-1,k-1)+kS(n-1,k)$$
where initial conditions $S(0,0)=1$ and $S(0,k)=0$ for $k\geq1$ or
$k<0$. It is not hard to find that $S(n,k)k!$ is the number of
surjections from an $n$-element set to a $k$-element set. The
triangular array $[S(n,k)]_{n,k}$ is called {\it the Stirling
triangle of the second kind.} Its row-generating polynomial
$B_n(q)=\sum_{k=0}^nS(n,k)q^k$ is called {\it the Bell polynomial}
and has only real zeros. In particular, its each row sequence is
log-concave. The Bell polynomial is closely related to the
differential operator $D=\frac{d}{dx}$ by
$$\left(xD\right)^n=\sum_{k=1}^nS(n,k)x^kD^k.$$
The famous Bell number $B_n(1)$ counts the total number of
partitions of a set with $n$ elements consisting of disjoint
nonempty sets. The Stirling triangle of the second kind
$[S(n,k)]_{n,k}$ is totally positive, which can be proved by
different methods \cite{Bre95,CLW15,Zhu14}. Such total positivity
can also be implied by Theorem \ref{thm+main+transpose} with
$t=0,\ell=1,a_k^{(0)}=k,a_k^{(1)}=1$ and $b_k=0$. Finally, for more
other properties of Stirling numbers of the second kind, we refer
the reader to Comtet \cite{Com74} and \cite[A008277]{Slo}.

\subsection{The generalized Jacobi-Stirling triangle of the second kind}

The Jacobi-Stirling numbers ${JS}_n^{(k)}(z)$ of the second kind,
which were introduced in \cite{EKLWY07}, are the coefficients of the
integral composite powers of the Jacobi differential operator:
$$ L_{\alpha,\beta}[y](u)=\frac{1}{(1-u)^{\alpha}(1+u)^{\beta}}(-(1-u)^{\alpha+1}(1+u)^{\beta+1}y'(u))'$$
with fixed real parameters $\alpha,\beta\geq -1$. They also satisfy
the following recurrence relation:
$${JS}_n^{(k)}(z)={JS}_{n-1}^{(k-1)}(z)+k(k+z){JS}_{n-1}^{(k)}(z)~~(n,k\geq1)$$
with the initial conditions
${JS}_n^{(0)}(z)={JS}_0^{(k)}(z)=0,{JS}_0^{(0)}(z)=1$ and
$z=\alpha+\beta+1$. For $\alpha=\beta=0$, the Jacobi-Stirling
numbers of the second kind reduce to the Legendre-Stirling numbers
of the second kind \cite{AGL11,Egg10,ELW02}. Similar to the
Jacobi-Stirling numbers of the first kind, in \cite{Mon122}, it was
demonstrated Jacobi-Stirling numbers of the second kind are
specializations of the complete symmetric functions by
$${JS}_n^{(k)}(z)=h_{n-k}(1(1+z),2(2+z),\ldots,k(k+z)).$$
In addition, many properties of the classical Stirling numbers have
been extended to those of the Jacobi-Stirling numbers and
Legendre-Stirling numbers, e.g., Jacobi-Stirling numbers of two
kinds satisfy the orthogonal relation:
$$\sum\limits_{k=1}^n(-1)^{k+j}{JS}_i^{(k)}{Jc}_k^{(j)}=\delta_{i,j}.$$
Mongelli \cite{Mon12} established total positivity of the
Jacobi-Stirling triangle $[{JS}_n^{(k)}(z)]_{n,k}$. We refer the
reader to Andrews {\it et al.}~\cite{AEGL13}, Everitt {\it et
al.}~\cite{EKLWY07}, Gelineau and Zeng~\cite{GZ10} for more
properties of the Jacobi-Stirling numbers of the second kind.

In \cite{Zhu14}, Zhu  proposed a generalization of Jacobi-Stirling numbers of the second kind.
Let $[J_{n,k}]_{n,k}$ be an array of nonnegative numbers satisfying the recurrence
relation
$$J_{n,k}=(a_1k^2+a_2k+a_3)J_{n-1,k}+(b_1k^2+b_2k+b_3)J_{n-1,k-1}$$
with $a_1k^2+a_2k+a_3\geq0$ and  $b_1k^2+b_2k+b_3\geq0$ for $k\geq0$, where initial conditions
 $J_{n,k}=0$ unless $0\leq k\leq n$ and $J_{0,0}=1$.
Except proving the P\'olya frequency properties of the row and
column sequences, Zhu \cite[Theorem~2.1]{Zhu14} also used an
algebraic method to prove that the matrix $[J_{n,k}]_{n,k}$ is
totally positive. Now such total positivity also follows from
Theorem \ref{thm+main+transpose} with
$t=0,\ell=1,a_k^{(0)}=a_1k^2+a_2k+a_3$, $a_k^{(1)}=b_1k^2+b_2k+b_3$
and $b_k=0$.

\section{Remarks}
If we replace the polynomial weights $a_n^i(x)$ and $b_n(y)$ in
definition \ref{def+path} with the indeterminates $a_n^i$ and $b_n$
in a partially ordered commutative ring, then we can generalize our
results for total positivity to total positivity in the ring
equipped with the coefficientwise order.






\begin{thebibliography}{99}

\bibitem{AH18}
K. Adiprasito, J. Huh and E. Katz, Hodge theory for combinatorial
geometries, Ann. of Math. (2) 188 (2018) 381--452.

\bibitem{Aig01}
M. Aigner, Lattice paths and determinants, in: Computational Discrete Mathematics: Advanced Lectures,
edited by H.  Alt (Lecture Notes in Computer Science $\#$2122) (Springer-Verlag, Berlin, 2001), pp.1--12.



\bibitem{AEGL13}
G.E. Andrews, E.S. Egge, W. Gawronski and L.L. Littlejohn, The
Jacobi-Stirling numbers, J. Combin. Theory Ser. A 120 (2013)
288--303.

\bibitem{AGL11}
G.E. Andrews, W. Gawronski, and L.L. Littlejohn,
The Legendre-Stirling numbers, Discrete Math. 311 (2011) 1255--1272.

\bibitem{AL09}
G.E. Andrews and L.L. Littlejohn, A combinatorial interpretation of the Legendre-Stirling numbers,
Proc. Amer. Math. Soc. 137 (8) (2009) 2581--2590.

\bibitem{BS05}
C. Banderier and S. Schwer, Why Delannoy numbers ?, J. Statist. Plann. Inference 135 (2005) 40--54.


\bibitem{BFZ96}
A. Berenstein, S. Fomin and A. Zelevinsky, Parametrizations of
canonical bases and totally positive matrices, Adv. Math. 122 (1996)
49--149.



\bibitem{Bra15}
P. Br\"{a}nd\'{e}n, Unimodality, log-concavity, real-rootedness and beyond, in Handbook of Combinatorics (M. Bona, ed.), CRC Press, 2015, pp. 437--483.
\bibitem{Bra09}
P. Br\"{a}nd\'{e}n, Iterated sequeces and the geometry of zeros, J.
Reine Angew. Math 658 (2011) 115--131.

\bibitem{BH20}
P. Br\"{a}nd\'{e}n and J. Huh, Lorentzian polynomials, Ann. of Math.
(2) 192 (2020) 821--891
\bibitem{Bre89}
F. Brenti, Unimodal, log-concave, and P\'olya frequency sequences in
combinatorics, Mem. Amer. Math. Soc. 413 (1989).

\bibitem{Bre94}
F. Brenti, Log-concave and unimodal sequences in algebra, combinatorics, and geometry: an update, Contemp. Math. 178 (1994) 71--89.

\bibitem{Bre95}
F. Brenti, Combinatorics and total positivity, J. Combin. Theory Ser. A 71 (1995) 175--218.



\bibitem{CKS09}
G.-S. Cheon, H. Kim and L.-W. Shapiro, A generalization of Lucas
polynomial sequence, Discrete Appl. Math. 157 (2009) 920--927.

\bibitem{CLW15}
X. Chen, H. Liang and Y. Wang, Total positivity of recursive
matrices, Linear Algebra Appl. 471 (2015) 383--393.

\bibitem{CLW152}
X. Chen, H. Liang and Y. Wang, Total positivity of Riordan arrays,
European J. Combin. 46 (2015) 68--74.

\bibitem{CW19}
X. Chen and Y. Wang, Notes on the total positivity of Riordan arrays, Linear Algebra Appl. 569 (2019) 156--161.


\bibitem{Com74}
L. Comtet, Advanced Combinatorics, Revised and enlarged edition, Reidel, Dordrecht, 1974.



\bibitem{Egg10}
E.S. Egge, Legendre-Stirling permutations, European J. Combin. 31 (7) (2010) 1735--1750.

\bibitem{EH20}
C. Eur and J. Huh, Logarithmic concavity for morphisms of matroids,
Adv. Math. 367 (2020), article: 107094, 19pp.


\bibitem{EKLWY07}
W.N. Everitt, K.H. Kwon, L.L. Littlejohn, R. Wellman and G.J. Yoon,
Jacobi-Stirling numbers, Jacobi polynomials, and the left definite
analysis of the classical Jacobi differential expression, J. Comput.
Appl. Math. 208 (2007) 29--56.

\bibitem{ELW02}
W.N. Everitt, L.L. Littlejohn and R. Wellman, Legendre polynomials,
Legendre-Stirling numbers, and the left-definite analysis of the
Legendre diffierential expression, J. Comput. Appl. Math. 148 (1)
(2002) 213--238.

\bibitem{FZ99}
S. Fomin and A. Zelevinsky, Double Bruhat cells and total positivity, J. Amer. Math. Soc. 12 (1999) 335--380.

\bibitem{FZ00}
S. Fomin and A. Zelevinsky, Total positivity: tests and parameterizations, Math. Intell. 22 (2000) 23--33.


\bibitem{GZ10}
Y. Gelineau and J. Zeng, Combinatorial interpretations of the Jacobi-Stirling numbers, Electron. J. Combin. 17 (2010), Research Paper 70.

\bibitem{Gessel-Viennot_89}
I.M. Gessel and X.G. Viennot, Binomial determinants, path, and hook length formulae, Adv. Math. 58 (1985) 300--321.

\bibitem{HS16}
T.-X. He and L.W. Shapiro, Row sums and alternating sums of Riordan arrays, Linear Algebra Appl. 507 (2016) 77--95.

\bibitem{HS09}
T.-X. He and R. Sprugnoli, Sequence characterization of Riordan arrays, Discrete Math. 309 (2009) 3962--3974.

\bibitem{Huh12}
J. Huh, Milnor numbers of projective hypersurfaces and the chromatic polynomial of graphs, J. Amer. Math. Soc. 25 (2012) 907--927.

\bibitem{Huh15}
J. Huh, h-vectors of matroids and logarithmic concavity, Adv. Math.
270 (2015) 49--59.

\bibitem{Kar68}
S. Karlin, Total Positivity, Volume 1, Stanford University Press, Stanford, 1968.

\bibitem{KW11}
Y. Kodama and L.K.Williams, KP solitons, total positivity and cluster algebras, Proc. Natl. Acad. Sci. USA 108 (2011) 8984--8989.

\bibitem{KW14}
Y. Kodama and L.K. Williams, KP solitons and total positivity for the Grassmannian, Invent. Math. 198 (3) (2014) 637--699.


\bibitem{KP99}
C. Krattenthaler and M. Prohaska, A remarkable formula for counting nonintersecting lattice paths in a ladder with respect to turns, Trans. Amer. Math. Soc. 351 (1999) 1015--1042.

\bibitem{Lindstrom_73}
B. Lindstr\"{o}m, On the vector representations of induced matroids, Bull. London Math. Soc. 5 (1973) 85--90.


\bibitem{Lus94}
G. Lusztig, Total positivity in reductive groups, Lie theory and geometry, Progr. Math., vol. 123, Birkh\"{a}user Boston, Boston, MA, 1994, pp. 531--568.

\bibitem{Lusztig_98}
G. Lusztig, Introduction to total positivity, in Positivity in Lie Theory: Open Problems, edited by J. Hilgert, J.D. Lawson, K.-H. Neeb and E.B. Vinberg (de Gruyter, Berlin, 1998), pp. 133--145.

\bibitem{Lusztig_08}
G. Lusztig, A survey of total positivity, Milan J. Math. 76 (2008) 125--134.

\bibitem{MMW22}
J. Mao, L. Mu and Y. Wang, Yet another criterion for the total
positivity of Riordan arrays, Linear Algebra Appl. 634 (2022)
106--111.


\bibitem{Moh79}
S.G. Mohanty, Lattice path counting and applications, Academic Press, New York, 1979.

\bibitem{Mon12}
P. Mongelli, Total positivity properties of Jacobi-Stirling numbers, Adv. Appl. Math. 48 (2012) 354--364.

\bibitem{Mon122} P. Mongelli, Combinatorial
interpretations of particular evaluations of complete and elementary
symmetric functions, Electron. J. Combin. 19 (2012), $\#$ R60.

\bibitem{MZ16}
L. Mu and S.-N. Zheng, On the total positivity of Delannoy-like triangles, J. Integer Seq. 20 (2017) 17.1.6.

\bibitem{Nar79}
T.V. Narayana, Lattice path combinatorics with statistical
applications, Toronto Press, Toronto, 1979.

\bibitem{PZ16}
Q. Pan and J. Zeng, On total positivity of Catalan-Stieltjes matrices, Electron. J. Combin. 23 (2016), 4.33.


\bibitem{Pin10}
A. Pinkus, Totally Positive Matrices, Cambridge University Press, Cambridge, 2010.

\bibitem{Pos06}
A. Postnikov, Total positivity, Grassmannians, and networks, Preprint, arXiv:math/0609764, 2006.

\bibitem{Raz02}
M. Razpet, The Lucas property of a number array, Discrete Math. 248 (2002) 157--168.

\bibitem{RS15}
J.L. Ram\'{i}rez and V.F. Sirvent, A generalization of the
$k$-bonacci sequence from Riordan arrays, Electron. J. Combin.  22
(2015), P1.38.

\bibitem{Rie03}
K. Rietsch, Totally positive Toeplitz matrices and quantum cohomology of partial flag varieties, J. Amer. Math. Soc. 16 (2) (2003) 363--392.

\bibitem{Sc30}
I. Schoenberg, \"{U}ber variationsvermindernde lineare Transformationen, Math. Z. 32 (1930) 321--322.

\bibitem{SGWW91}
L.W. Shapiro, S. Getu, W.-J. Woan and L.C. Woodson, The Riordan
group, Discrete Appl. Math. 34 (1--3) (1991) 229--239.


\bibitem{Slo}
N.J.A. Sloane, The On-Line Encyclopedia of Integer Sequences,
https://oeis.org.


\bibitem{Spr94}
R. Sprugnoli, Riordan arrays and combinatorial sums, Discrete Math. 132 (1994) 267--290.


\bibitem{Sta89}
R.P. Stanley, Log-concave and unimodal sequences in algebra, combinatorics, and geometry, Ann. New York Acad. Sci. 576 (1989) 500--534.

\bibitem{Stanley12}
R.P. Stanley, Enumerative Combinatorics, Volume 1, Second edition, Cambridge Studies in Advanced Mathematics, 49. Cambridge University Press, Cambridge, 2012.

\bibitem{SW08}
X.-T. Su and Y. Wang, On the unimodality problems in Pascal's triangle, Electron. J. Combin. 15 (2008), $\#$ R113.

\bibitem{SX08}
R. Sulanke and G. Xin, Hankel determinants for some common lattice paths, Adv. Appl. Math. 40 (2008) 149--167.

\bibitem{WY18}
Y. Wang and A.L.B. Yang, Total positivity of Narayana matrices, Discrete Math. 341 (2018) 1264--1269.


\bibitem{WYjcta05}
Y. Wang and Y.-N. Yeh, Polynomials with real zeros and P\'olya frequency sequences, J. Combin. Theory Ser. A 109 (2005) 63--74.

\bibitem{WY07}
Y. Wang and Y.-N. Yeh, Log-concavity and LC-positivity, J. Combin. Theory Ser. A 114 (2007) 195--210.

\bibitem{WZC19}
Y. Wang, S.-N. Zheng and X. Chen, Analytic aspects of Delannoy numbers, Discrete
Math. 342 (2019) 2270--2277.


\bibitem{Yu09}
Y. Yu, Confirming two conjectures of Su and Wang on binomial coefficients, Adv. Appl. Math. 43 (2009) 317--322.


\bibitem{Zhu14}
B.-X. Zhu, Some positivities in certain triangular array, Proc. Amer. Math. Soc. 142 (9) (2014) 2943--2952.

\bibitem{Zhu21}
B.-X. Zhu, Total positivity from the exponential Riordan arrays, SIAM J. Discrete Math. 35 (2021) 2971--3003.





\end{thebibliography}
\end{document}